\documentclass{amsart}
\usepackage{amsmath,amscd,amssymb,amsthm}
\usepackage{amsfonts}
\usepackage{mathabx}
\usepackage[all]{xy}
\usepackage{mathtools}
\usepackage{enumerate}

\usepackage{hyperref}

\usepackage{adjustbox}

\usepackage[usenames, dvipsnames]{color}

\definecolor{NTNUblue}{RGB}{0,80,158}
\definecolor{NTNUbluesupport}{RGB}{62,98,138}
\definecolor{NTNUorange}{RGB}{239,129,20}

\newcommand{\into}{\hookrightarrow}

\newcommand{\xto}{\xrightarrow}

\newcommand{\F}{\mathbb{F}}

\newcommand{\Q}{\mathbb{Q}}
\newcommand{\R}{\mathbb{R}}
\newcommand{\Z}{\mathbb{Z}}

\newcommand{\BBB}{\mathbf{B}}

\newcommand{\DDD}{\mathbf{D}}

\newcommand{\DDDw}{\widetilde{\DDD}}

\newcommand{\EEE}{\mathbf{E}}

\newcommand{\SSS}{\mathbf{S}}

\newcommand{\TTT}{\mathbf{T}}

\newcommand{\XXX}{\mathbf{X}}

\newcommand{\BII}{\BBB^{\mathrm{II}}}

\newcommand{\kII}{\kappa^{\mathrm{II}}}

\newcommand{\PsiII}{\Psi^{\mathrm{II}}}

\newcommand{\SII}{\SSS^{\mathrm{II}}}

\newcommand{\DII}{\DDD^{\mathrm{II}}}

\newcommand{\EII}{\EEE^{\mathrm{II}}}

\newcommand{\TII}{\TTT^{\mathrm{II}}}

\newcommand{\TIIw}{\widetilde{\TII}}

\newcommand{\TIIt}{T^{\mathrm{II}}(\F_2)}

\newcommand{\WIIt}{W^{\mathrm{II}}(\F_2)}

\newcommand{\tauII}{\tau^{\mathrm{II}}}

\newcommand{\rrII}{\rr^{\mathrm{II}}}

\newcommand{\brrII}{\brr^{\mathrm{II}}}

\newcommand{\BIII}{\BBB^{\mathrm{III}}}

\newcommand{\kIII}{\kappa^{\mathrm{III}}}

\newcommand{\PsiIII}{\Psi^{\mathrm{III}}}

\newcommand{\SIII}{\SSS^{\mathrm{III}}}

\newcommand{\DIII}{\DDD^{\mathrm{III}}}

\newcommand{\EIII}{\EEE^{\mathrm{III}}}

\newcommand{\TIII}{\TTT^{\mathrm{III}}}

\newcommand{\TIIIw}{\widetilde{\TIII}}

\newcommand{\XIII}{\XXX^{\mathrm{III}}}

\newcommand{\VIIIt}{V^{\mathrm{III}}(\F_2)}

\newcommand{\WIIIt}{W^{\mathrm{III}}(\F_2)}

\newcommand{\tauIII}{\tau^{\mathrm{III}}}

\newcommand{\rrIII}{\rr^{\mathrm{III}}}

\newcommand{\brrIII}{\brr^{\mathrm{III}}}

\newcommand{\Ah}{\mathcal{A}}

\newcommand{\Ch}{\mathcal{C}}

\newcommand{\Zh}{\mathcal{Z}}

\newcommand{\Ext}{\mathrm{Ext}}

\newcommand{\Hom}{\mathrm{Hom}}

\newcommand{\uHom}{\underline{\Hom}}

\newcommand{\gExt}{\underline{\Ext}}

\newcommand{\gHom}{\underline{\Hom}}

\newcommand{\op}{\mathrm{op}}

\newcommand{\HH}{\mathrm{HH}}

\newcommand{\Imm}{\mathrm{Im}\,}

\newcommand{\Ker}{\mathrm{Ker}\,}

\newcommand{\bbb}{\bullet}

\newcommand{\Cb}{\Ch^\bbb}

\newcommand{\Hb}{H^\bbb}

\newcommand{\dee}{\partial}

\newcommand{\dd}{\delta}

\newcommand{\One}{\mathbf{1}}

\newcommand{\Jw}{\widetilde{J}}

\newcommand{\gag}{\gamma_G}

\newcommand{\ot}{\otimes}

\newcommand{\bU}{\overline{U}}

\newcommand{\rr}{\rho} 

\newcommand{\brr}{\overline{\rho}}

\newcommand{\Tlt}{T_4^l(\F_2)} 

\newcommand{\TTlt}{T_5^l(\F_2)} 

\newcommand{\Wlt}{W_4^l(\F_2)} 

\newcommand{\WWlt}{W_5^l(\F_2)}

\newcommand{\Trt}{T_4^r(\F_2)} 

\newcommand{\TTrt}{T_5^r(\F_2)} 

\newcommand{\Wrt}{W_4^r(\F_2)} 

\newcommand{\WWrt}{W_5^r(\F_2)}

\newcommand{\taul}{\tau^l} 

\newcommand{\taur}{\tau^r} 

\theoremstyle{plain}

\newtheorem{theorem}{Theorem}[section]

\newtheorem{prop}[theorem]{Proposition}
\newtheorem{proposition}[theorem]{Proposition}
\newtheorem{lemma}[theorem]{Lemma}
\newtheorem{cor}[theorem]{Corollary}
\newtheorem{corollary}[theorem]{Corollary}

\theoremstyle{definition}
\newtheorem{definition}[theorem]{Definition}
\newtheorem{defn}[theorem]{Definition}
\newtheorem{example}[theorem]{Example}

\newtheorem{notn}[theorem]{Notation}
\newtheorem{convention}[theorem]{Convention}
\theoremstyle{remark}

\newtheorem{remark}[theorem]{Remark}

\begin{document}

\title{$A_3$-formality for pro-$2$ Demushkin groups} 

\author{Ambrus P\'al}
\address{Mathematical Institute, E\"{o}tv\"{o}s Lor\'{a}nd University, H-1117 Budapest, Hungary} 
\email{ambrus.pal@ttk.elte.hu}

\author{Gereon Quick}
\address{Department of Mathematical Sciences, NTNU, NO-7491 Trondheim, Norway}
\email{gereon.quick@ntnu.no}
\thanks{Both authors were partially supported by RCN Project No.\,313472 {\it Equations in Motivic Homotopy}, 
and the project \emph{Pure Mathematics in Norway} funded by the Trond Mohn Foundation.}


\begin{abstract} 
We study a weak form of formality for differential graded algebras, called $A_3$-formality, 
and show that the differential graded $\F_2$-algebras of continuous cochains of all pro-$2$ Demushkin groups are $A_3$-formal. 
We prove this result by an explicit computation of the Benson--Krause--Schwede canonical class 
using the classification of pro-$2$ Demushkin groups by Demushkin, Serre, and Labute.  
Compared to the case of odd primes, the new idea is to interpret the data of the canonical class 
as defining systems of higher Massey products. 
\end{abstract}

\maketitle


\section{Introduction}

Let $F$ be a field and let $G_F$ denote its absolute Galois group. 
Let $\Cb(G_F,\F_2)$ denote the differential graded algebra of inhomogeneous continuous cochains 
of $G_F$ with coefficients in the constant discrete $G_F$-module $\F_2$. 
In \cite{HW}, Hopkins and Wickelgren showed that, for $F$ a local or global field, 
all triple Massey products of elements in $H^1(G_F,\F_2)$ vanish whenever they are defined. 
Since triple Massey products are the first obstruction to formality, 
Hopkins--Wickelgren therefore asked in \cite[Question 1.4]{HW} 
whether $\Cb(G_F,\F_2)$ is $A_{\infty}$-formal, 
i.e., the higher multiplication maps $m_n$ on its minimal $A_{\infty}$-model are trivial for all $n \ge 3$. 
However, Positselski showed in \cite[Section 9.11]{PoMoscow} and \cite[\S 6]{Po} 
that $\Cb(G_F,\F_2)$ is not $A_{\infty}$-formal in general for  local fields. 
%
Then Harpaz--Wittenberg, in the appendix to Guillot--Min\'a\v c--Topaz \cite[Example A.15]{GMT}, 
showed that the second obstruction to formality is not trivial for $F=\Q$, 
i.e., not all fourfold Massey products are defined when the neighbouring cup-products vanish. 
See also the work of Merkurjev--Scavia in \cite[Theorem 1.6]{MS1} and \cite[Theorem 1.3]{MS2023} 
for other examples. 
For the absolute Galois groups of local fields, or more generally, for Demushkin groups, 
we may then ask the weaker question whether the differential graded algebra $\Cb(G,\F_2)$ is $A_3$-formal, 
i.e., whether the first higher multiplication map $m_3$ on its minimal model is trivial. 
$A_3$-formality is much stronger than the vanishing of triple Massey products. 

\subsection{Demushkin groups and the main result}
The purpose of this paper is to study $A_3$-formality for 
$\Cb(G,\F_2)$ for pro-$2$ Demushkin groups.

\begin{defn}\label{def:Demushkin}
Let $p$ be a prime number and let $G$ be a pro-$p$ group. 
Then $G$ is called a \emph{Demushkin group} if 
$\dim_{\F_p}H^1(G,\F_p) < \infty$, 
$\dim_{\F_p}H^2(G,\F_p) = 1$, 
and the cup-product $H^1(G,\F_p) \times H^1(G,\F_p) \to H^2(G,\F_p)$ is a non-degenerate bilinear form. 
\end{defn}

Demushkin groups are Poincar\'e groups of dimension two, 
and arise, for example, as pro-$p$ completions of fundamental
groups of compact surfaces. 
They also play a key role in number theory since the maximal pro-$p$ quotients of absolute Galois groups of local fields that contain a primitive $p$-th root of unity are Demushkin groups or trivial (see \cite{Demushkin, Demushkin2}, \cite{SerreDem}, \cite{Labute}, and Section \ref{subsec:classification}). 
Demushkin groups are fundamental building blocks of the class of elementary type pro-$p$ groups in the sense of Efrat. 
%
The only finite Demushkin group is $\Z/2$ and $\Cb(\Z/2,\F_2)$ is known to be formal. 
We therefore consider from now on infinite Demushkin groups. 
Our main result in this paper is the following theorem, 
see Theorems \ref{thm:torsion_free_pro-2_Demushkin_groups_are_A_3-formal}, 
\ref{thm:type_II_pro-2_Demushkin_groups_are_A_3-formal}, 
and \ref{thm:type_III+IV_general_kappa_3_is_trivial}: 

\begin{theorem}\label{thm:Demushkin_groups_A3_formality_intro}
Let $G$ be a pro-$2$ Demushkin group.   
Then $\Cb(G,\F_2)$ is $A_3$-formal. 
\end{theorem}


The example of Harpaz--Wittenberg in \cite[Example A.15]{GMT} and Proposition \ref{prop:4tuple_products_are_defined} 
imply that $\Cb(G_{\Q},\F_2)$ is not $A_3$-formal. 
However, 
by the work of Demushkin, Serre, and Labute, see Examples \ref{example:realizable_Demushkin_group_two_generators_type_I_and_III}
and \ref{example:realizable_Demushkin_group_three_gen}, 
Theorem \ref{thm:Demushkin_groups_A3_formality_intro} 
has the following consequence. 
For a field $F$, let $G_F(2)$ denote the Galois group of the maximal pro-$2$ Galois extension of $F$. 
Since $G_{\R}(2) = \Z/2$, we get the following result: 

\begin{cor}\label{cor:local_global_does_not_hold_for_Q_and_A3_formality_intro}
For every place $\ell$ of $\Q$,  
$\Cb(G_{\Q_{\ell}}(2),\F_2)$ is $A_3$-formal, 
whereas $\Cb(G_{\Q},\F_2)$ is not $A_3$-formal. 
\end{cor}


Before we outline the proof of  Theorem \ref{thm:Demushkin_groups_A3_formality_intro} 
in Section \ref{subsec:intro_outline} and comment on the new ideas used compared to the case of odd primes in \cite{PQ3}, 
we first summarise the now complete characterisation of which Demushkin groups are $A_3$-formal 
and will then describe the relation of our work to the Massey vanishing conjecture.

Let $p$ be a prime number, and let $G$ be a pro-$p$ Demushkin group 
with $d \coloneqq \dim_{\F_p}H^1(G,\F_p)$.  
Let $G^{\mathrm{ab}}$ denote the abelianisation of $G$. 
Then, by for example \cite[Section 3.9 on page 232]{NSW}, 
either $G^{\mathrm{ab}} \cong \Z_p^d$ or 
$G^{\mathrm{ab}} \cong \Z_p/p^f\Z_p \times \Z_p^{d-1}$ with $f\ge 1$. 
In the former case $q \coloneqq 0$, and in the latter case we set $q \coloneqq p^f$. 
The numbers $d$ and $q$ are invariants of $G$. 
Theorem \ref{thm:Demushkin_groups_A3_formality_intro} 
and \cite[Theorem 1.2]{PQ3} then combine to the following result. 

\begin{theorem}\label{thm:Demushkin_groups_A3_formality_all_p_summary}
Let $p$ be a prime number, and let $G$ be a pro-$p$ Demushkin group with invariants $d\ge 2$ and $q$. 
For $q \ne 3$, $\Cb(G,\F_p)$ is $A_3$-formal.  
For $q=p=3$, $\Cb(G,\F_3)$ is not $A_3$-formal. 
\end{theorem}


\subsection{Relation to the Massey vanishing conjecture of Min\'a\v c--T\^an}
Min\'a\v c and T\^an conjectured in \cite[Conjecture 1.6]{MT2} that, for every field $F$ and prime $p$, 
$G_F$ {\em satisfies $n$-Massey vanishing} with respect to $p$, i.e., 
all $n$-fold Massey products of elements in $H^1(G_F,\F_p)$ vanish whenever they are defined. 
%
By the work of Matzri \cite{Matzri}, Efrat--Matzri \cite{EM17} and Min\'a\v c--T\^an \cite{MT1}, 
all fields satisfy triple Massey vanishing with respect to all primes. 
For a more recent, independent proof, see \cite[Section 6.2]{Lam_etal}. 
In \cite{HarpazWittenberg}, Harpaz--Wittenberg showed that number fields satisfy $n$-Massey vanishing with respect to all primes. 
More recently, Merkurjev--Scavia proved in \cite{MS2} that all fields satisfy fourfold Massey vanishing with respect to $p=2$. 
Other cases of the conjecture have been proven in \cite{PQ}, \cite{PSz} and \cite{Quadrelli}. 
The vanishing of Massey products has concrete consequences for the structure of the Zassenhaus filtration of an absolute Galois group 
and thereby led to new examples of profinite groups which are not absolute Galois groups of a field (see for example \cite{MT1, MT2}). 
%

In \cite[Definition 4.5]{MT0bis}, Min\'a\v c and T\^an also formulate the following related property. 
Let $G$ be a profinite group and $p$ be a prime number. %
Then $G$ is said to have the 
{\em cup-defining $n$-fold Massey product property} (with respect to $\F_p$) 
if for every $\chi_1, \ldots, \chi_n \in H^1(G,\F_p)$ with 
$0 = \chi_1 \cup \chi_2 = \chi_2 \cup \chi_3 = \cdots = \chi_{n-1} \cup \chi_n$ 
the $n$-fold Massey product $\langle \chi_1, \ldots, \chi_n \rangle$ is defined.  
For $n\ge 4$, this is a non-trivial condition, 
and, in \cite[Remark 4.4]{MT0bis}, Min\'a\v c--T\^an show that not all pro-$p$ groups 
have the cup-defining $n$-fold Massey product property. 
In \cite[Question 4.2]{MT0bis}, Min\'a\v c--T\^an ask whether every  
Galois group of a maximal $p$-extension of a field $F$ containing a primitive $p$-th root of unity 
has the cup-defining $n$-fold Massey product property with respect to $\F_p$ 
(see also \cite[Section 8]{MTE}). 
Moreover, in \cite[Proposition 4.1]{MT0bis}, Min\'a\v c--T\^an prove that pro-$p$ Demushkin groups 
have the cup-defining $n$-fold Massey product property with respect to $\F_p$. 
Together with their work in \cite{MT2}, this implies that pro-$p$ Demushkin groups 
have the following stronger property. 

We say that a profinite group $G$ satisfies {\em strong} $n$-Massey vanishing with respect to $p$ 
if for every $\chi_1, \ldots, \chi_n \in H^1(G,\F_p)$ with 
$0 = \chi_1 \cup \chi_2  
= \cdots = \chi_{n-1} \cup \chi_n$ 
the $n$-fold Massey product $\langle \chi_1, \ldots, \chi_n \rangle$ is {\em defined and  vanishes}. 
For $n\ge 4$, this is a strictly stronger condition than $n$-Massey vanishing. 
By the work of Min\'a\v c--T\^an in \cite[Proposition 4.1]{MT0bis} and \cite[Theorem 4.3]{MT1}, 
pro-$p$ Demushkin groups satisfy strong $n$-Massey vanishing with respect to $p$ and all $n \ge 3$.   
P\'al--Szab\'o gave an independent proof that pro-$p$ Demushkin groups satisfy strong $n$-Massey vanishing for all $n \ge 3$ 
in \cite[Theorem 3.5]{PSz}. 
Moreover, by \cite[Theorem 1]{MMRT}, the absolute Galois groups of number fields which do not contain a primitive $p$-th root of unity 
satisfy strong $n$-Massey vanishing with respect to $p$ for all $n\ge 3$.  
Strong vanishing of triple Massey products is a necessary condition for the $A_3$-formality of $\Cb(G,\F_p)$. 
In fact, $A_3$-formality implies triple Massey vanishing and the cup-defining fourfold Massey product property (see Proposition \ref{prop:4tuple_products_are_defined}). 
We note, however, that $A_3$-formality is a significant strengthening of the vanishing of the Massey product obstructions  
since $A_3$-formality requires that a specific element in the triple Massey product of elements in $\Hb(G, \F_p)$ vanishes 
and that defining systems of triple Massey products can be chosen compatibly. 
Moreover, 
\cite[Theorem 6.12]{PQ3} demonstrates that strong vanishing of Massey products is only a necessary but not a sufficient condition for $A_3$-formality.


\subsection{Outline of the proof}\label{subsec:intro_outline}
We now give a brief outline of the proof of  Theorem \ref{thm:Demushkin_groups_A3_formality_intro}. 
While $n$-Massey vanishing for Demushkin groups is a direct consequence of the non-degeneracy of the cup product, 
showing $A_3$-formality is much more involved. 
By the work of Kadeishvili, the differential graded algebra $\Cb(G,\F_2)$ is $A_3$-formal if and 
only if a certain {\em canonical class} $\gamma_G$ in the Hochschild cohomology group $\HH^{3,-1}(\Hb(G,\F_2))$ is zero. 
The canonical class of a differential graded algebra has also been studied by Benson--Krause--Schwede in \cite{BKS} 
in the context of the realizability of modules over Tate cohomology. 
The strategy for proving Theorem \ref{thm:Demushkin_groups_A3_formality_intro} is then 
to show that $\gamma_G$ vanishes in $\HH^{3,-1}(\Hb(G,\F_2))$ for all pro-$2$ Demushkin groups. 
The task of computing the class $\gamma_G$ in $\HH^{3,-1}(\Hb(G,\F_2))$ is simplified by 
the work of Min\'a\v c--Pasini--Quadrelli--T\^an in \cite[Theorem 5.2]{MPQT} who showed that 
the graded $\F_2$-algebra $\Hb(G,\F_2)$ for a pro-$2$ Demushkin group is Koszul. 
A second key tool is Dwyer's Theorem on unipotent representations and Massey products in group cohomology. 
A third key ingredient is the classification of pro-$2$ Demushkin groups by Demushkin, Serre, and Labute. 
A Demushkin group has a finite number of generators subject to a single relation. 
There are four different types of relations, each leading to a specific type of pro-$2$ Demushkin group. 
For each type, we compute the canonical class in a separate section. 
In each case, we first determine a basis for the relevant spaces $R$ and $K_3^3$ in the Koszul complex of $\Hb(G,\F_2)$. 
Then we compute a sufficiently large part of the differential in the complex which computes the Hochschild cohomology of $\Hb(G,\F_2)$. 
Next, we construct the map $\kappa_3 \colon K_3^3(\Hb(G,\F_2)) \to H^2(G,\F_2)$ which represents the canonical class. 
Finally, we compute the canonical class of $G$ for each case. 

The key new idea for $p=2$ compared to the argument for odd primes in \cite{PQ3} 
is that we interpret $\kappa_3$ as higher Massey products. 
In \cite{PQ3}, we computed triple Massey products in $H^1(G,\F_p)$ via Dwyer's Theorem.
Since there are enough relations in $K_3^3(\Hb(G,\F_p))$ for an odd prime $p$, 
we could then determine all values of $\kappa_3$. 
For $p=2$, however, there are not enough relations. 
Hence we use a different strategy and interpret the data in $K_3^3(\Hb(G,\F_2))$ as defining systems 
in certain $n$-fold Massey products for $n=3,\ldots,8$,  
and then we compute the specific values of the corresponding element in $H^2(G,\F_2)$. 
After adding signs, the new method also works for $p$ odd, and hence we can use the methods in Section \ref{sec:torsion-free_p=2_Demushkin} 
to provide a new computation of the canonical class of a pro-$p$ Demushkin group for odd primes. 
%

%
%

\subsection{Contents} 
From Section \ref{sec:HH_A3_MP} onward, unless stated otherwise, all vector spaces and algebras are over $\F_2$. 
In Section \ref{subsec:HH}, we recall the definition of graded Hochschild cohomology groups of a graded algebra. 
In Section \ref{subsec:Koszul_algebras}, we recall the definition of Koszul algebras. 
In Section \ref{subsec:A3_MP}, 
we define $A_3$-formality for differential graded algebras and recall Massey products. 
Then we construct the canonical class, 
first for a general differential graded algebra in  Section \ref{subsec:canonical_class}, 
and then for a differential graded algebra whose cohomology algebra is Koszul in Section \ref{sec:canonical_class_Koszul}. 
In Section \ref{subsec:group_coh}, 
we recall continuous group cohomology, 
and in Section \ref{subsec:Dwyer}, 
we recall Dwyer's Theorem. 
In Section \ref{subsec:classification}, 
we recall the classification of pro-$2$ Demushkin groups of Demushkin, Serre, and Labute, 
and we determine the algebra structure on $\Hb(G,\F_2)$ for pro-$2$ Demushkin groups. 
In Sections \ref{sec:torsion-free_p=2_Demushkin}, 
\ref{sec:type_II}, and \ref{sec:type_III},  
we compute the canonical class of pro-$2$ Demushkin groups for the three cases 
we need to consider according to the classification. 

\subsection{Acknowledgements} 
We thank Jan Min\'a\v c for helpful comments.


\section{Hochschild cohomology, $A_3$-formality, and Massey products}\label{sec:HH_A3_MP}

In this section, we recall the background for the computation of the canonical class of a differential graded algebra. 
%


\subsection{Graded Hochschild cohomology}\label{subsec:HH}

Let $A$ be a graded unital $\F_2$-algebra. 
We recall that the bar resolution
$B(A)$ of $A$ is the non-negative chain complex
of free graded $A$-bimodules
given by $B_n(A) \coloneqq A^{\otimes n+2}$ for $n \geq 0$.
The differential 
$d_n \colon B_n(A) \to B_{n-1}(A)$
is given by
\begin{align}\label{eq:diff_bar_complex}
a_0 \otimes \cdots \otimes a_{n+1} 
\mapsto \sum_{i = 0}^n 
a_0 \otimes \cdots \otimes 
a_i a_{i+1} \otimes \cdots \otimes a_{n+1}.
\end{align}
We write $A^e \coloneqq A \ot A^{\op}$, where $A^{\op}$ denotes the opposite algebra. 
Note that $A^{\otimes n+2} \cong A^e \otimes A^{\otimes n}$
as a graded $A$-bimodule, 
and hence $B(A)$ indeed consists of free modules. 

\begin{proposition}\label{bar-acyclic}
The bar resolution $B(A)$ is a free resolution
of $A$ as a graded $A$-bimodule.
\end{proposition}
\begin{proof}
It suffices to show that 
the extended complex $\widetilde{B}(A)$ is acyclic,  
where $\widetilde{B}(A)$  
is extended from $B(A)$ by adjoining 
$\widetilde{B}_{-1}(A) := A$ 
in degree $-1$ via the multiplication 
map $\mu \colon A \otimes A \to A$. 
The map
$h \colon \widetilde{B}(A) \to \widetilde{B}(A)$ of degree $1$ given by
\[
a_0 \otimes \cdots \otimes a_{n+1}
\mapsto 1 \otimes a_0 \otimes \cdots \otimes a_{n+1}
\]
is a contracting homotopy, i.e., $dh + hd = 1$, since 
\[
dh(a_0 \otimes \cdots \otimes a_{n+1})
= a_0 \otimes \cdots \otimes a_{n+1}
- hd(a_0 \otimes \cdots \otimes a_n). \qedhere
\]
\end{proof}

\begin{definition}
Let $M$ and $N$ be graded $A$-bimodules. 
We define $\gHom_A(M, N)$ as the graded $\F_2$-vector space
with degree $s$ component given by
$A$-linear graded maps $f \colon M \to N[s]$, where $N[s]$ is the graded $A$-module given by $N[s]^n = N^{s+n}$.
\end{definition}

\begin{definition}
Let $M$ be a graded $A$-bimodule. 
We define the Hochschild cohomology $\HH^{n,\bullet}(A,M)$
as the $n$th cohomology of the cochain complex
\[
\gHom_{A^e}(B(A),M)
\]
of graded $\F_2$-vector spaces with differential $\delta^n (f) = f \circ d_n$ where $d_n$ is the differential of $B(A)$ in \eqref{eq:diff_bar_complex}. 
When $M=A$ we will write
$\HH(A) \coloneqq \HH(A, A)$. 
\end{definition}

We note that the groups $\HH^{\bullet, \bullet}(A,M)$ are equipped with a cohomological grading, and an internal grading induced by the grading of $A$ and $M$.
We can describe $\HH^{n,s}(A,M)$ more concretely as follows. 
Using the natural contracting isomorphism
\[
\gHom_{A^e}(A^e \otimes A^{\otimes n}, M) \cong 
    \gHom_{\F_2}(A^{\otimes n}, M)
\]
we see that $\HH^{n,\bullet}(A, M)$ is isomorphic to the $n$th cohomology of the complex
\begin{align}
\label{eq:HH_reduced_complex}
\cdots \to \gHom_{\F_2}(A^{\otimes n-1}, M) \xto{\partial}
\gHom_{\F_2}(A^{\otimes n}, M) 
\xto{\partial} \gHom_{\F_2}(A^{\otimes n+1}, M) \to \cdots,
\end{align}
where the differentials are given by
\begin{align*}
\partial(f)(a_1 \otimes \cdots \otimes a_{n+1}) = 
    &  
   ~~  a_1f(a_2 \otimes \cdots \otimes a_{n+1})
     + \sum_{i=1}^{n} 
        f(a_1 \otimes \cdots \otimes a_i a_{i+1}
        \otimes \cdots \otimes a_{n+1})\\
    & ~~ + 
    f(a_1 \otimes \cdots \otimes a_{n})a_{n+1}.
\end{align*}


\begin{remark}\label{rem:HH_and_Ext}
By Proposition~\ref{bar-acyclic},
we see that
$\HH(A, M)$ computes the graded Ext
modules $\gExt_{A^e}(A, M)$. 
In particular,
we can compute $\HH(A, M)$
using any free resolution of $A$ as
a graded $A$-bimodule. 
\end{remark}



\subsection{Koszul algebras}\label{subsec:Koszul_algebras}

Let $V$ denote a finite-dimensional $\F_2$-vector space, 
and let $T(V)$ denote its graded tensor algebra over $\F_2$.  
For $R \subseteq V \otimes V$, 
let $(R)$ denote the two-sided ideal in $T(V)$ 
generated by $R$. 
We recall from \cite{ppqa} that a graded algebra 
of the form $T(V)/(R)$ is called a {\em quadratic algebra}. 
%
%
For any quadratic algebra $A$, 
we can define the following chain complex of free graded $A$-bimodules. 

\begin{definition}\label{koszulcomplex}
Let $A = T(V)/(R)$ be a quadratic algebra. 
Let $K^0_0 = \F_2$, $K^1_1 = V$, 
and, for $n\geq 2$, let  
\[
K^n_n \coloneqq \bigcap_{i=1}^{n-1} V^{\otimes i-1} \otimes R \otimes V^{\otimes n-i-1} \subseteq V^{\otimes n}.
\]
The Koszul complex $K(A^e, A)$ of $A$ is defined
as the non-negative chain complex of graded $A$-bimodules with
\[
K_n(A^e, A) \coloneqq A \otimes K^n_n \otimes A,
\]
and differential $d_n$ induced by the one in the bar resolution $B(A)$, 
i.e., 
\[
d_n \colon a \otimes v_1 \otimes \cdots \otimes v_n \otimes b
\mapsto
av_1 \otimes v_2 \otimes \cdots \otimes v_n \otimes b
+ a \otimes v_1 \otimes \cdots \otimes v_n b.
\]
\end{definition}

Note that for the differential $d_n$ in the Koszul complex, 
the middle terms of the bar construction differential 
in \eqref{eq:diff_bar_complex} vanish. 
This is because each product \(v_i v_{i+1}\)
in a middle term has its factors \(v_i\), \(v_{i+1}\) in the space of relations $R$, 
so the product \(v_i v_{i+1}\) vanishes in $A$, 
which is where the product in the expression 
of the differential is taking place. 


\begin{definition}
A quadratic algebra $A$ is called {\em Koszul}
if its Koszul complex $K(A^e, A)$ is
a resolution of $A$ as a graded $A$-bimodule,
i.e., if $H_n(K(A^e, A)) = 0$ for $n > 0$
and $H_0(A^e, A) = A$.
\end{definition}


Let $A$ be a Koszul algebra. 
Then the natural inclusion $K(A^e, A) \into B(A)$ 
is a quasi-isomorphism. 
Hence we can compute
$\HH(A)$ as the cohomology of
the complex $\gHom_{A^e}(K(A^e, A), A)$.
We first observe that we have
\[
K(A^e, A)_n = A \otimes K^n_n \otimes A \cong A^e \otimes K^n_n 
\]
as graded $A$-bimodules. 
Using the contracting isomorphism
\[
\gHom_{A^e} (A^e \otimes K^n_n, A) \cong \gHom_{\F_2}(K^n_n, A)
\]
of graded vector spaces 
we see that $\HH(A)$ can be computed as the cohomology
of the following complex of graded vector spaces:
\[
\cdots \to \gHom_{\F_2}(K^{n-1}_{n-1}, A) 
\xrightarrow{\partial^{n-1}} \gHom_{\F_2}(K^n_n, A)
\xrightarrow{\partial^{n}} \gHom_{\F_2}(K^{n+1}_{n+1}, A)
\to \cdots 
\]
where the differential $\partial^n$
is given by
\[
\partial^n(f)(v_1 \otimes \cdots \otimes v_{n+1}) 
    = v_1f(v_2 \otimes \cdots \otimes v_{n+1}) + 
    f(v_1 \otimes \cdots \otimes v_n)v_{n+1}.
\]
In particular, the group $\HH^{3,-1}(A)$ is thus isomorphic to the cohomology of the complex  
\begin{align}
\label{eq:Koszul_complex_3,-1}
\Hom_{\F_2}(K^2_2, A^1) 
    \xrightarrow{\partial^2} \Hom_{\F_2}(K^3_3, A^2)
    \xrightarrow{\partial^3} \Hom_{\F_2}(K^4_4, A^3). 
\end{align}


\subsection{$A_3$-formality and Massey products}\label{subsec:A3_MP}

We now recall the definition of an $A_{\infty}$-algebra.  
For an introduction to the theory of $A_{\infty}$-algebras  
we refer to \cite{Keller}. 

\begin{definition}\label{def:Ainfty_algebra}
Let $\Ah = \oplus_{i \ge 0} \Ah^i$ be a non-negatively graded $\F_2$-vector space with $\Ah^0=\F_2$. 
Then $\Ah$ is called an \emph{$A_{\infty}$-algebra over $\F_2$} if, for all $i \ge 1$, 
there are graded $\F_2$-linear maps  
$m_i \colon \Ah^{\otimes i} \to \Ah[2-i]$ 
such that, for all $n \ge 1$,  
\begin{align*}
\sum_{r+s+t=n} m_{n-s+1}(\One^{\ot r} \ot m_s \ot \One^{\ot t}) = 0  
\end{align*}
as maps $\Ah^{\otimes n} \to \Ah$, 
where $\One$ denotes the identity map of $\Ah$. 
\end{definition} 

\begin{remark}
For $n=1,2,3$, the maps $m_n$ satisfy the following identities: 
first, $m_1m_1=0$, i.e., $(\Ah,m_1)$ is a cochain complex;  
second, $m_1m_2  = m_2(m_1 \otimes \One + \One \otimes m_1)$ 
i.e., $m_1$ is a graded derivation with respect to the multiplication $m_2$; 
and 
\begin{equation*}
 m_2(\One \otimes m_2 + m_2 \otimes \One) 
=  m_1m_3 + m_3(m_1 \otimes \One \otimes \One + \One \otimes m_1 \otimes \One + \One \otimes \One \otimes m_1), 
\end{equation*}
i.e., $m_2$ is associative up to homotopy. 
\end{remark}

\begin{example}
Every graded $\F_2$-algebra is an $A_{\infty}$-algebra with trivial $m_1$ and $m_n$ for $n \ge 3$.    
Every differential graded algebra $(\Cb, \delta, \cup)$ over $\F_2$ is an $A_3$-algebra with $m_1= \delta$, $m_2 = \cup$, and $m_n=0$ for all $n \ge 3$. 
\end{example}


Let $(\Cb, \delta, \cup)$ be a differential graded algebra (DGA) over $\F_2$ with cohomology algebra $\Hb$. 
%
By the work of Kadeishvili \cite{Kadeishvili82, Kadeishvili88} (see also 
\cite{Keller}, and \cite{Merkulov}), 
one can equip $\Hb$ with the structure of an $A_{\infty}$-algebra $(\Hb,\{m_n\}_{n\ge 1})$ such that $m_1=0$ together with a quasi-isomorphism of $A_{\infty}$-algebras 
$(\Hb,\{m_n\}) \xto{\simeq} (\Cb, \delta, \cup)$. 
The $A_{\infty}$-algebra $(\Hb,\{m_n\})$ is called a {\em minimal model} of $(\Cb, \delta, \cup)$.  
Since any two minimal models of $\Cb$ are isomorphic as $A_{\infty}$-algebras, 
we speak of {\em the} minimal model from now on. 
A DGA is called {\em $A_{\infty}$-formal}, or {\em formal} as an $A_{\infty}$-algebra, if its minimal model can be chosen such that $m_n = 0$ for all $n \ge 3$. 
%
%
We now consider the following weaker notion.  
We refer to \cite{PQ3} for a more detailed discussion. 

\begin{definition}\label{def:A3formal}
Let $\Cb$ be a DGA. 
We say that $\Cb$ is {\em $A_3$-formal} if its minimal model can be chosen such that $m_3 = 0$.  
\end{definition}


For a simple but non-trivial example of an $A_3$-formal DGA we refer to \cite{DQ}. 
The notion of an $A_{\infty}$-algebra is closely related to Massey products. 
We recall the definition of Massey products for degree one classes next.   

\begin{defn}\label{def:Massey_product} 
Let $(\Cb, \delta, \cup)$  be a DGA over $\F_2$ with cohomology algebra $\Hb$.  
For an integer $n\geq2$, 
let $a_1,a_2,\ldots,a_n$ be cohomology classes in $H^1$. 
A \emph{defining system} for the $n$-fold Massey product of $a_1,a_2,\ldots,a_n$ is a set $\{a_{ij}\}$ of elements of $\Ch^1$ for $1\leq i<j\leq n+1$ and $(i,j)\neq(1,n+1)$ such that
\[
\delta(a_{i,j})=\sum_{k=i+1}^{j-1} a_{i,k}\cup a_{k,j}
\]
and, for each $i$,  the class $a_i$ is represented by $a_{i,i+1}$. 
We say that the $n$-fold Massey product of $a_1,a_2,\ldots,a_n$ is \emph{defined} if there exists a defining system. 
The $n$-fold Massey product $\langle a_1,a_2,\ldots,a_n\rangle_{\{a_{i,j}\}}$ of $a_1,a_2,\ldots,a_n$ with respect to the defining system $\{a_{i,j}\}$ is the cohomology class of
\[
\sum_{k=2}^n a_{1,k}\cup a_{k,n+1} 
\]
in $H^2$. 
Let $\langle a_1,a_2,\ldots,a_n\rangle$ denote the subset of $H^2$ consisting of the $n$-fold Massey products of $a_1,a_2,\ldots,a_n$ with respect to all defining systems. 
The set $\langle a_1,a_2,\ldots,a_n\rangle$ is called the \emph{$n$-fold Massey product} of $a_1,a_2,\ldots,a_n$. 
We say that the $n$-fold Massey product of $a_1,a_2,\ldots,a_n$ {\it vanishes} if the set $\langle a_1,a_2,\ldots,a_n\rangle \subseteq H^2$ contains zero.  
\end{defn}


\begin{example}\label{example:triple_Massey_product}
For $n=2$, the Massey product $\langle a_1, a_2\rangle$ is just the cup product $a_1 \cup a_2$. 
Let $a_1, a_2,a_3 \in H^1$ 
such that $a_1 \cup a_2 = 0$ and $a_2 \cup a_3 = 0$. 
Then the {\it triple Massey product} $\langle a_1, a_2, a_3 \rangle$ is defined.  
\end{example}


We recall the following special case from \cite[Theorem C]{BMM}. 
Let $a,b,c \in \Hb$ be cohomology classes such that $a \cup b = 0$ and $b \cup c = 0$. 
Then $m_3(a \ot b \ot c) \in \langle a,b,c \rangle$. 
This implies the following two well-known facts (see for example \cite[Corollary 2.15 and Proposition 2.16]{PQ3}).

\begin{proposition}\label{prop:A3formal_Massey}
Let $\Cb$ be a DGA with cohomology algebra $\Hb$.   
Let $a,b,c \in H^1$ be cohomology classes such that $a \cup b = 0$ and $b \cup c = 0$. 
Assume that $\Cb$ is $A_3$-formal. 
Then the triple Massey product $\langle a,b,c \rangle$ contains zero. 
\end{proposition}

\begin{proposition}\label{prop:4tuple_products_are_defined}
Let $\Cb$ be a DGA with cohomology algebra $\Hb$.   
Let $a_1, a_2, a_3, a_4 \in H^1$ be classes such that the neighbouring cup-products vanish, i.e., $a_1 \cup a_2 = 0$, $a_2 \cup a_3 = 0$, and $a_3 \cup a_4 = 0$.
If $\Cb$ is $A_3$-formal, then the fourfold Massey product $\langle a_1,a_2,a_3,a_4 \rangle$ is defined. 
\end{proposition}


\subsection{The canonical class}\label{subsec:canonical_class}

Let $(\Cb, \delta, \cup)$ be a DGA with cohomology algebra $\Hb$. 
Let $(\Hb,\{m_n\})$ be the minimal model of $\Cb$ as an $A_{\infty}$-algebra. 
We note that we can construct the map $m_3 \colon (\Hb)^{\ot 3} \to \Hb[-1]$ 
as follows (see for example \cite[Section 5]{BKS}, \cite{Merkulov}, \cite{PQ3}).  
We choose an $\F_2$-linear graded map $f_1 \colon \Hb \to \Ker \delta$ which induces the identity when taking cohomology.  
Since $f_1$ is multiplicative on cohomology, we can find a graded $\F_2$-linear map  
$f_2 \colon \Hb \otimes \Hb \to \Cb$ of degree $-1$ satisfying  
\begin{align*}
\delta(f_2(a \ot b)) = f_1(a \cup b) + f_1(a) \cup f_1(b). 
\end{align*}
Now we define a graded $\F_2$-linear map $\Phi_3 \colon (\Hb)^{\otimes 3} \to \Cb[-1]$ by 
\begin{align*}
\Phi_3(a \ot b \ot c) \coloneqq 
f_1(a) f_2(b \ot c) 
+ f_2(a \ot b) f_1(c) + f_2((a b) \ot c + a \ot (b c)) 
\end{align*}
for all homogeneous elements $a,b,c \in \Hb$ where we write $xy$ for the product $x \cup y$ to shorten the notation.  
We check that $\Phi_3$ has image in the cocycles of $\Cb$, 
and hence $\Phi_3$ induces a graded map $[\Phi_3] \colon (\Hb)^{\ot 3} \to \Hb[-1]$. 
We set $m_3 := [\Phi_3]$. 
By \cite[Proposition 5.4]{BKS}, $m_3$ is a cocycle in the complex \eqref{eq:HH_reduced_complex}.   
By \cite[Corollary 5.7]{BKS}, 
the corresponding Hochschild cohomology class  $[m_3] \in \HH^{3,-1}(\Hb)$ is independent of the choice of $f_1$ and $f_2$,   
and it is called the {\em canonical class} of $\Cb$ following Benson--Krause--Schwede who studied this class as an obstruction to the realizability of modules over Tate cohomology in \cite{BKS}.   
We note that the canonical class is also sometimes called the \emph{universal triple Massey product}, see for example \cite{Muro}. 
The following result is a modified version of Kadeishvili's theorem \cite{Kadeishvili88} (see also \cite[Theorem 3.9]{PQ3}, \cite[Theorem 2.7]{DQ}, and 
\cite[Theorem 4.7]{ST}). 

\begin{theorem}\label{thm:canonical_A3formal}
Let $\Cb$ be a DGA 
with canonical class $[m_3] \in \HH^{3,-1}(\Hb)$. 
Then $\Cb$ is $A_3$-formal if and only if $[m_3]=0$. 
\end{theorem}


\subsection{Canonical class for Koszul cohomology algebras}\label{sec:canonical_class_Koszul}

Let $(\Cb, \delta, \cup)$ be a differential graded algebra over $\F_2$ with cohomology algebra $\Hb$. 
We now assume that $\Hb$ is a Koszul algebra. 
In this case, the canonical class of $\Cb$ can be constructed in a simpler way as follows. 
Let $R \subset H^1 \ot H^1$ denote the relations such that $\Hb = T(H^1)/(R)$.  
Let $\Zh^1 = \ker \dd \subset \Ch^1$ denote the cocycles in degree one. 
Let $f_1 \colon K_1^1(\Hb) = H^1 \to \Zh^1$ be an $\F_2$-linear map which induces the identity when taking cohomology. 
Let $f_2 \colon K_2^2(\Hb) = R \to \Ch^1$ be an $\F_2$-linear map such that $\dd f_2(a,b) = f_1(a) \cup f_1(b)$. 
We define the $\F_2$-linear map 
\begin{align}\label{eq:def_of_Psi3_construction}
\Psi_3(a,b,c) \coloneqq f_1(a) \cup f_2(b,c) + f_2(a,b) \cup f_1(c) \colon K_3^3(\Hb) \to \Zh^2
\end{align}
where, by slight abuse of notation, we allow $(a,b)$, $(b,c)$, and $(a,b,c)$ to denote a sum of tensors in $R$ and $K_3^3(\Hb)$ respectively. 
By a direct computation, we check that $\Psi_3$ has image in the cocycles of $\Ch^2$. 
Hence $\Psi_3$ induces an $\F_2$-linear map 
\begin{align*}
\kappa_3(a,b,c)  = f_1(a) \cup f_2(b,c) + f_2(a,b) \cup f_1(c) \colon K_3^3(\Hb) \to H^2.
\end{align*}
In fact, the map $\kappa_3$ 
represents the canonical class of $\Cb$:  

\begin{proposition}\label{prop:kappa3_is_canonical_class}
The class of $\kappa_3$ in the cohomology of the complex \eqref{eq:Koszul_complex_3,-1}  
is the image of the canonical class of $\Cb$ under the automorphism of $\HH^{3,-1}(\Hb)$ 
induced by the inclusion $K^3_3(\Hb) \into (H^1)^{\ot 3}$. 
In particular, the differential graded algebra $\Cb$ is $A_3$-formal if and only if $[\kappa_3]=0$. 
\end{proposition}
\begin{proof}
This follows from \cite[Proposition 4.12]{PQ3} and Theorem \ref{thm:canonical_A3formal}. 
\end{proof}


\section{Dwyer's theorem and Demushkin groups}\label{sec:Dwyer_and_Demushkin}

We now specialise to differential graded algebras which arise from continuous group cohomology of profinite groups. 

\subsection{Continuous group cohomology and the canonical class}\label{subsec:group_coh}

Let $G$ be a profinite group, and let $G^n$ denote the $n$-fold direct product of $G$ with itself. 
Let $\Ch^n(G,\F_2)$ denote the $\F_2$-vector space of continuous functions $G^n \to \F_2$ 
with respect to the discrete topology on $\F_2$ and the profinite topology on $G$. 
Following \cite[\S 2.2]{Se} the differential $\delta \colon \Ch^n(G,\F_2) \to \Ch^{n+1}(G,\F_2)$, which is defined by 
\begin{align*}
(\delta \varphi)(g_1, \ldots, g_{n+1}) = & ~~ \varphi(g_2, \ldots,g_{n+1}) 
 + \sum_{i=1}^n \varphi(g_1, \ldots, g_ig_{i+1}, \ldots, g_{n+1}) \\
&  +  \varphi(g_1,\ldots,g_n), 
\end{align*}
turns $\Cb(G,\F_2)$ into a cochain complex whose cohomology $\Hb(G,\F_2)$ is the continuous cohomology of $G$ with coefficients in the trivial $G$-module $\F_2$.  
In particular, $H^1(G,\F_2)$ is the group of continuous group homomorphisms $G \to \F_2$. 
%
For every $\varphi \in \Ch^i(G,\F_2)$ and $\psi \in \Ch^j(G,\F_2)$, 
we define $\varphi \cup \psi \in \Ch^{i+j}(G,\F_2)$ by the formula: 
\begin{align*}
(\varphi \cup \psi)(g_1,\ldots,g_{i+j}) = 
\varphi(g_1,\ldots,g_i) \cdot \psi(g_{i+1},\ldots,g_{i+j}).  
\end{align*}
This induces the cup product on cohomology which turns $\Hb(G,\F_2)$ into a graded  
$\F_2$-algebra. 

\begin{defn}
Let $G$ be a pro-$2$ group.  
We say that $G$ is {\em $A_3$-formal}  
if the differential graded $\F_2$-algebra $\Cb(G,\F_2)$ given by continuous cochains is $A_3$-formal.  
We write $\gag$ for the canonical class $\gamma_{\Cb(G,\F_2)}$ in $\HH^{3,-1}(\Hb(G,\F_2))$ and call it the \emph{canonical class of $G$}.  
%
\end{defn}


\subsection{Dwyer's theorem on Massey products}\label{subsec:Dwyer}

We now recall from \cite[Theorem 2.4]{Dwyer} that the definition and vanishing of Massey products 
can be characterised as follows. 
Let $U_n(\F_2)$ denote the group of all upper triangular unipotent $(n\times n)$-matrices with coefficients in $\F_2$. 
Let $Z_n(\F_2)$ denote the center of $U_n(\F_2)$, 
i.e., the subgroup of all matrices in $U_n(\F_2)$ with all off-diagonal entries being $0$ except at position $(1,n)$. 
We write $\bU_n(\F_2) \coloneqq U_n(\F_2)/Z_n(\F_2)$. 

\begin{notn} 
We let $e_{ij} \colon U_n(\F_2) \to \F_2$ denote the projection to the $(i,j)$-coordinate. 
\end{notn}
The following result was proven by Dwyer in \cite[Theorem 2.4]{Dwyer}. 

\begin{theorem}[Dwyer]\label{thm:Dwyer}
Let $G$ be a profinite group. 
Let $a_1, \ldots, a_n \in H^1(G,\F_2)$. 
There is a one-to-one correspondence $M \leftrightarrow \brr_M$ between defining systems $M$ for $\langle a_1, \ldots, a_n \rangle$ 
and continuous group homomorphisms $\brr_M \colon G \to \bU_{n+1}(\F_2)$ such that 
$e_{i,i+1} \circ (\brr_M) = a_i$ for $i=1, \ldots, n$, 
where we identify $a_i$ with the corresponding continuous group homomorphism $G \to \F_2$. 
The correspondence is given by sending a defining system $M=\{a_{i,j}\}$ to the 
continuous group homomorphism $\brr_M \colon G \to \bU_{n+1}(\F_2)$ 
such that $e_{i,j} \circ \brr = a_{i,j}$ for $1 \le i < j \le n+1$. 
Moreover, the element $\langle a_1, \ldots, a_n \rangle_M \in H^2(G,\F_2)$ vanishes if and only if 
there is a continuous group homomorphism $\brr_M \colon G \to \bU_{n+1}(\F_2)$ making the following diagram commutative 
\begin{align*}
\xymatrix{
 & & & \ar@{.>}[dl]_-{\rr_M} G \ar[d]^{\brr_M} & \\
0 \ar[r] & Z_{n+1}(\F_2) \ar[r] & U_{n+1}(\F_2) \ar[r] & \bU_{n+1}(\F_2) \ar[r] & 0. 
}
\end{align*}
\end{theorem}


\begin{example}\label{example:unipotent3_and_cup_product}
Let $\chi_1, \chi_2 \in H^1(G,\F_2)$. 
Then we have $\chi_1 \cup \chi_2 = 0$ if and only if there is a continuous group homomorphism 
$\rr \colon G \to U_3(\F_2)$ such that 
$e_{1,2} \circ \rr =  \chi_1$ and $e_{2,3} \circ \rr =  \chi_2$.  
In particular, the continuous map $\eta \coloneqq e_{1,3} \circ \rr \colon G \to \F_2$ is a cochain in $\Ch^1(G,\F_2)$ 
such that $\delta \eta =  \chi_1 \cup \chi_2$. 
\end{example}


\subsection{The classification of pro-$2$ Demushkin groups following Demushkin, Serre, and Labute}\label{subsec:classification} 
By the work of Demushkin \cite{Demushkin, Demushkin2}, Serre \cite{SerreDem}, and Labute \cite{Labute}, 
Demushkin groups are completely classified and pro-$2$ Demushkin groups can be classified as follows. 
Let $G$ be a pro-$2$ Demushkin group. 
%
For elements $x,y \in G$, 
let $[x,y] = x^{-1}y^{-1}xy$ denote the commutator. 
In the following, we use the convention $2^{\infty} \coloneqq 0$. 
Let $d \coloneqq \dim_{\F_2}H^1(G,\F_2)$, which is a finite number by definition of a Demushkin group. 
Then $G$ is generated by elements $x_1,\ldots,x_d$ 
subject to a single relation $r = r(x_1,\ldots,x_d)$. 
The number $d$ and the relation $r$ satisfy one of the following characterisations: 
\begin{itemize}
\item[(I)] Demushkin\cite{Demushkin2}: 
$d \ge 2$ is even, $f \in \{2,3, \ldots\} \cup \{\infty\}$, and $r$ is of the form 
\begin{align}\label{eq:type_I_relation_general}
x_1^{2^f} [ x_1,x_2] [x_3,x_4] \cdots [x_{d-1}, x_d] = 1; 
\end{align}

\item[(II)] Serre \cite{SerreDem}: 
$d \ge 3$ is odd, $f \in \{2,3, \ldots\} \cup \{\infty\}$, and $r$ is of the form 
\begin{align}\label{eq:type_II_relation_general}
x_1^2 x_2^{2^f} [x_2,x_3]  \cdots [x_{d-1},x_d] = 1; 
\end{align}

\item[(III)] Labute \cite{Labute}: 
$d \ge 2$ is even, $f \in \{2,3, \ldots\} \cup \{\infty\}$, and $r$ is of the form 
\begin{align}\label{eq:type_III_relation_general}
x_1^{2 + 2^f} [x_1,x_2]  \cdots [x_{d-1},x_d] = 1; 
\end{align}

\item[(IV)] Labute \cite{Labute}:  
$d \ge 4$ is even, $f \in \{2,3, \ldots\}$, and $r$ is of the form  
    \begin{align}\label{eq:type_IV_relation_general}
x_1^{2} [x_1,x_2] x_3^{2^f} [x_3,x_4] \cdots [x_{d-1},x_d] = 1.
\end{align}  
\end{itemize}


\begin{example}\label{example:realizable_Demushkin_group_two_generators_type_I_and_III}
Let $\ell$ be an odd prime number. 
Let $G_{\Q_{\ell}}(2)$ be the maximal pro-$2$ quotient of the absolute Galois group of $\Q_{\ell}$, 
and let $q = 2^f$ be the maximum power of $2$ such that $\Q_{\ell}$ contains a $q$-th root of unity. 
By \cite[Proposition 7.5.9]{NSW}, 
$G_{\Q_{\ell}}(2)$ is a pro-$2$ group with generators $x_1$, $x_2$ subject to the relation $1=x_1^q[x_1,x_2]$. 
This group is of type I if $q \ge 4$, and it is of type III if $q = 2$.   
\end{example}


\begin{example}\label{example:realizable_Demushkin_group_three_gen}
%
Let $\Q_2$ denote the field of $2$-adic numbers, 
and let $K$ be a finite extension of $\Q_2$ of degree $n$. 
%
%
Let $q = 2^f$ denote the highest power of $2$ such that $K$ contains a primitive $q$-th root of unity. 
Assume $q \ne 1$. 
Let $G$ be the Galois group of the maximal pro-$2$ extension of $K$. 
Then, by \cite[Th\'eor\`eme 4.2]{SerreDem}, $G$ is a pro-$2$ Demushkin group with $d = n+2$. 
By \cite{Demushkin2} and \cite[Corollaire 4.3]{SerreDem}, if $q\ne 2$, then $G$ is defined by relation \eqref{eq:type_I_relation_general}. 
By \cite[Corollaire 4.4]{SerreDem}, if $q = 2$ and $n$ is odd, 
then $G$ is defined by relation \eqref{eq:type_II_relation_general}. 
In particular, for $K = \Q_2$, $G$ is the pro-$2$ Demushkin group $G$ with generators $x_1,x_2,x_3$ 
subject to the single relation $1 = x_1^2x_2^4 [ x_2,x_3]$. 
\end{example}


\begin{example}\label{example:realizable_Demushkin_group_four_gen}
We recall the following example provided by Labute in \cite[\S 5]{Labute}.  
Let $\Q_2$ denote the field of $2$-adic numbers, 
and let $K = \Q_2(\sqrt{-2})$.  
Let $G$ be the maximal pro-$2$ quotient of the absolute Galois group of $K$. 
By \cite[Example below Theorem 9 on page 132]{Labute}, 
$G$ is defined by relation \eqref{eq:type_III_relation_general} with $d=4$, 
i.e., $G$ is generated by $x_1,x_2,x_3,x_4$ subject to the relation 
$x_1^{2+4}[x_1,x_2][x_3,x_4] = 1$. 
For other examples of this type, we refer to \cite[Chapter 7.5, page 417]{NSW}. 
\end{example}


\begin{example}\label{example:non-realizable_Demushkin_group_three_gen}
Let $G$ be the pro-$2$ group with generators $x_1,x_2,x_3$ subject to the single relation $1 = x_1^2 [ x_2,x_3]$. 
According to \cite[Remark 5.5]{JacobWare}, it is not known whether this group arises as the maximal pro-$2$ quotient of an absolute Galois group.  
\end{example}


In the following proofs and constructions we will frequently use the following observation, often without explicitly mentioning it.  

\begin{notn}
For matrices $M$ and $N$ in $U_n(\F_2)$, we write $[M,N] \coloneqq M^{-1}N^{-1}MN$. 
\end{notn}


\begin{lemma}\label{lemma:matrix_relations}
Let $G$ be a pro-$2$ group with generators $x_1,\ldots,x_d$ subject to the single relation 
$r(x_1,\ldots,x_d)=1$. 
Let $M_1, \ldots, M_d \in U_n(\F_2)$ and let $I_n$ denote the $(n\times n)$-identity matrix. 
If 
\begin{align}\label{eq:matrix_commutator_relation}
r(M_1,\ldots, M_d) = I_n 
\end{align} 
in $\bU_n(\F_2)$, 
then the assignment $\brr \colon x_i \mapsto M_i$ for $x_1,\ldots,x_d$ 
defines a continuous group homomorphism $\brr \colon G \to \bU_n(\F_2)$. 
Moreover, $\brr$ can be lifted to a continuous group homomorphism $\rr \colon G \to U_n(\F_2)$ 
if and only if \eqref{eq:matrix_commutator_relation} holds in $U_n(\F_2)$.  
\end{lemma}
\begin{proof}
This follows directly from the defining relation for $G$ and the fact that $\bU_n(\F_2)$ and $U_n(\F_2)$ are finite $2$-groups. 
\end{proof}


\begin{lemma}\label{lemma:relations_in_H2_all_groups}
Let $G$ be a pro-$2$ Demushkin group with generators $x_1, \ldots, x_d$.  
We let $\chi_i \colon G \to \F_2$ denote the group homomorphism determined by $\chi_i(x_j) = \delta_{ij}$, where $\delta_{ij}$ denotes the Kronecker symbol. 
For each $i$, we consider $\chi_i$ as an element in $H^1(G,\F_2)$. 
In $H^2(G,\F_2)$, we have the following relations: 
\begin{itemize}
\item If $G$ is defined by relation \eqref{eq:type_I_relation_general}, then 
$\chi_1 \cup \chi_2 \ne 0$, $\chi_1 \cup \chi_2 = \chi_{2i-1} \cup \chi_{2i} = \chi_{2i} \cup \chi_{2i-1}$ for $1 \le i \le d/2$, 
and $\chi_i \cup \chi_j = 0$ for all other cases.   
We have $H^2 = \F_2 \langle \chi_1 \cup \chi_2 \rangle$. 
\item If $G$ is defined by relation \eqref{eq:type_II_relation_general}, then 
 $\chi_1 \cup \chi_1 \ne 0$, $\chi_1 \cup \chi_1 = \chi_{2i} \cup \chi_{2i+1} =  \chi_{2i+1} \cup \chi_{2i}$ for $1 \le i \le (d-1)/2$, 
and $\chi_j \cup \chi_k = 0$ for all other cases.   
We have $H^2 = \F_2 \langle \chi_1 \cup \chi_1 \rangle$. 
\item If $G$ is defined either by relation \eqref{eq:type_III_relation_general} or by relation \eqref{eq:type_IV_relation_general}, then 
$\chi_1 \cup \chi_1 \ne 0$, $\chi_1 \cup \chi_1 = \chi_{2i-1} \cup \chi_{2i} =  \chi_{2i} \cup \chi_{2i-1}$ for $1 \le i \le d/2$, 
and $\chi_j \cup \chi_k = 0$ for all other cases.  
We have $H^2 = \F_2 \langle \chi_1 \cup \chi_1 \rangle$. 
\end{itemize}
\end{lemma}
\begin{proof}
First, we assume that $G$ is defined by relation \eqref{eq:type_I_relation_general}. 
We write $q \coloneqq 2^f$. 
Let $A_{10}
= \begin{psmallmatrix} 
1 & 1 & 0  \\
0 & 1 & 0  \\
0 &  0 & 1  \\
\end{psmallmatrix}$
and 
$A_{01}
= \begin{psmallmatrix} 
1 & 0 & 0  \\
0 & 1 & 1  \\
0 &  0 & 1  \\
\end{psmallmatrix}$ 
be matrices in $U_3(\F_2)$. 
We have $A^{2^n}_{10} = A^{2^n}_{01} = I_3$ in $U_3(\F_2)$ for all $n\ge 1$.  
Since $[A_{10},A_{01}] = 
\begin{psmallmatrix} 
1 & 0 & 1  \\
0 & 1 & 0  \\
0 &  0 & 1  \\
\end{psmallmatrix}$,  
we have $A^{q}_{10}[A_{10},A_{01}] = I_3$ in $\bU_3(\F_2)$ 
but $A^q_{10}[A_{10},A_{01}] \ne I_3$ in $U_3(\F_2)$.  
By Lemma \ref{lemma:matrix_relations}, 
the assignment $x_{1} \mapsto A_{10}$, 
$x_2 \mapsto A_{01}$, 
and $x_j \mapsto I_3$ for $j \ne 1,2$ 
defines a continuous homomorphism $G \to \bU_3(\F_2)$ 
such that $e_{1,2} \circ \rr = \chi_1$, 
and  
$e_{2,3} \circ \rr = \chi_2$ 
which does not lift to a continuous homomorphism $G \to U_3(\F_2)$. 
By Dwyer's Theorem \ref{thm:Dwyer} and Example \ref{example:unipotent3_and_cup_product},  
this implies that $\chi_{1} \cup \chi_{2} \ne 0$ in $H^2(G,\F_2)$. 
Now assume that $i \ne j$ and $(i,j) \ne (2k-1,2k), (2k,2k-1)$. 
Since $[M,I_3] = [I_3,M] = I_3$ for every $M\in U_3(\F_2)$, 
the assignment $x_i \mapsto A_{10}$, 
$x_{j} \mapsto A_{01}$, 
and $x_s \mapsto I_3$ for $s \ne i, j$ 
defines, by Lemma \ref{lemma:matrix_relations},  
a continuous homomorphism $\rr_{i,j} \colon G \to U_3(\F_2)$ 
such that $e_{1,2} \circ \rr_{i,j}  = \chi_i$, 
and  
$e_{2,3} \circ \rr_{i,j}  = \chi_j$. 
By Dwyer's Theorem \ref{thm:Dwyer} and Example \ref{example:unipotent3_and_cup_product},  
this implies 
$\chi_i \cup \chi_j = 0$ in $H^2(G,\F_2)$. 
Now let $A_{11}$ denote the matrix 
$A_{11} = \begin{psmallmatrix} 
1 & 1 & 0  \\
0 & 1 & 1  \\
0 &  0 & 1  \\
\end{psmallmatrix}$ 
in  $U_3(\F_2)$.  
We have 
$A_{1,1}^n = 
\begin{psmallmatrix} 
1 & n & \binom{n}{2} \\
0 & 1 & n  \\
0 &  0 & 1  
\end{psmallmatrix}$ 
with $\binom{n}{2} = 0$ for $n=0,1$,   
and hence 
$A_{11}^{q} = 
\begin{psmallmatrix} 
1 & 0 & 0  \\
0 & 1 & 0  \\
0 &  0 & 1  \\
\end{psmallmatrix}$ 
in $U_3(\F_2)$ since either $q=0$ or $q =2^f$ with $f\ge 2$.  
Hence, by Lemma \ref{lemma:matrix_relations}, 
the assignment $x_i \mapsto A_{11}$, 
and $x_j \mapsto I_3$ for $j \ne i$ 
defines a continuous homomorphism $\rr_{i,i}  \colon G \to U_3(\F_2)$ 
such that $e_{1,2} \circ \rr_{i,i}  = \chi_i = e_{2,3} \circ \rr_{i,i}$. 
By Dwyer's Theorem \ref{thm:Dwyer} and Example \ref{example:unipotent3_and_cup_product},  
this implies 
$\chi_i \cup \chi_i = 0$ in $H^2(G,\F_2)$.
Since $A_{10}^q[A_{10},A_{01}][A_{01},A_{10}]=I_3$ in $U_3(\F_2)$, 
the assignment $x_1 \mapsto A_{10}$, 
$x_2 \mapsto A_{01}$, 
$x_{2i-1} \mapsto A_{01}$, 
$x_{2i} \mapsto A_{10}$,  
and $x_j \mapsto I_3$ else 
defines a continuous homomorphism $\rr \colon G \to U_3(\F_2)$ 
such that $e_{1,2} \circ \rr = \chi_1 + \chi_{2i}$, 
and  
$e_{2,3} \circ \rr = \chi_{2} + \chi_{2i-1}$. 
By Dwyer's Theorem \ref{thm:Dwyer} and Example \ref{example:unipotent3_and_cup_product},  
this implies 
$(\chi_1 + \chi_{2i}) \cup (\chi_2 + \chi_{2i-1}) = 0$ in $H^2(G,\F_2)$. 
Since we already know $\chi_1 \cup \chi_{2i-1} = 0$ 
$\chi_{2i} \cup \chi_2 = 0$, 
this implies 
$\chi_1 \cup \chi_2 = \chi_{2i} \cup \chi_{2i-1}$  in $H^2(G,\F_2)$. 
By switching the assignments for $x_1$, $x_2$, $x_{2i-1}$, and $x_{2i}$ 
accordingly we get the remaining relation. 
This finishes the proof of the first case. 

%
%

Now we assume that $G$ is defined by relation \eqref{eq:type_II_relation_general}. 
%
Since $A_{11}^2 = 
\begin{psmallmatrix} 
1 & 0 & 1  \\
0 & 1 & 0  \\
0 &  0 & 1  \\
\end{psmallmatrix}$,  
we have $A_{11}^2 = I_3$ in $\bU_3(\F_2)$ but $A_{11}^2 \ne I_3$ in $U_3(\F_2)$. 
By Lemma \ref{lemma:matrix_relations}, 
the assignment $x_1 \mapsto A_{11}$ and $x_j \mapsto I_3$ for $j \ne 1$ 
defines a continuous homomorphism $G \to \bU_3(\F_2)$ 
which does not lift to a continuous homomorphism $G \to U_3(\F_2)$. 
By Dwyer's Theorem \ref{thm:Dwyer} and Example \ref{example:unipotent3_and_cup_product}, 
this implies that $\chi_1 \cup \chi_1 \ne 0$ in $H^2(G,\F_2)$. 
We have $A^{2^n}_{10} = A^{2^n}_{01} = I_3$ for all $n\ge 1$.  
Assume $(j,k) \ne (1,1), (2i,2i+1), (2i+1,2i)$. 
By Lemma \ref{lemma:matrix_relations}, 
the assignment $x_j \mapsto A_{10}$, 
$x_{k} \mapsto A_{01}$, 
and $x_r \mapsto I_3$ for $r \ne j,k$ 
defines a continuous homomorphism $\rr \colon G \to U_3(\F_2)$ 
such that $e_{1,2} \circ \rr = \chi_{j}$, 
and  
$e_{2,3} \circ \rr = \chi_{k}$. 
By Dwyer's Theorem \ref{thm:Dwyer} and Example \ref{example:unipotent3_and_cup_product}, 
this implies 
$\chi_j \cup \chi_k = 0$ in $H^2(G,\F_2)$. 
Since $[A_{10},A_{01}] = 
\begin{psmallmatrix} 
1 & 0 & 1  \\
0 & 1 & 0  \\
0 &  0 & 1  \\
\end{psmallmatrix}$,  
we have $A^2_{11}[A_{10},A_{01}] = I_3$ in $U_3(\F_2)$.  
By Lemma \ref{lemma:matrix_relations}, 
the assignment $x_1 \mapsto A_{11}$, 
$x_{2i} \mapsto A_{10}$, 
$x_{2i+1} \mapsto A_{01}$, 
and $x_j \mapsto I_3$ for $j \ne 1, 2i,2i+1$ 
defines a continuous homomorphism $\rr \colon G \to U_3(\F_2)$ 
such that $e_{1,2} \circ \rr = \chi_1 + \chi_{2i}$, 
and  
$e_{2,3} \circ \rr = \chi_1 + \chi_{2i+1}$. 
By Dwyer's Theorem \ref{thm:Dwyer} and Example \ref{example:unipotent3_and_cup_product}, 
this implies 
$(\chi_1 + \chi_{2i}) \cup (\chi_1 + \chi_{2i+1}) = 0$ in $H^2(G,\F_2)$. 
Since we already know that $\chi_1 \cup \chi_{2i+1} = \chi_{2i} \cup \chi_1 = 0$ in $H^2(G,\F_2)$, 
this implies $\chi_1 \cup \chi_1 + \chi_{2i} \cup \chi_{2i+1} = 0$ in $H^2(G,\F_2)$. 
By switching the assignments for $x_{2i}$ and $x_{2i+1}$, 
we get the assertion $\chi_1 \cup \chi_1 + \chi_{2i+1} \cup \chi_{2i} = 0$ in $H^2(G,\F_2)$, too. 
This finishes the second case.  

%
%

The assertion for relations \eqref{eq:type_III_relation_general} and \eqref{eq:type_IV_relation_general},
follows as for relation \eqref{eq:type_II_relation_general} 
by replacing $x_{2i}$ and $x_{2i+1}$ 
with $x_{2i-1}$ and $x_{2i}$ respectively. 
See also Section \ref{subsec:construction_kappa_type_III}. 
Here we note that every matrix $M \in U_3(\F_2)$ we used  
satisfies $M^4 = I_3$ in $U_3(\F_2)$. 
Hence an additional factor $M^{2^f}$ in a matrix relation of the form  \eqref{eq:matrix_commutator_relation},  
corresponding to the additional factor $x_1^{2^f}$ or $x_3^{2^f}$ in the defining relation for $G$,  
does not impose an obstruction for defining a continuous group homomorphism $G \to U_3(\F_2)$. 
\end{proof}



\section{Pro-$2$ Demushkin groups of type I}\label{sec:torsion-free_p=2_Demushkin}

The goal of this section is to prove the following theorem. 

\begin{theorem}\label{thm:torsion_free_pro-2_Demushkin_groups_are_A_3-formal}
Let $d \ge 2$ be an even number, 
and let $q = 2^f$ with $f\ge 2$, or $q = 0$. 
%
Let $G$ be the pro-$2$ group with minimal set of generators $\{x_1,\ldots,x_d\}$ 
satisfying the single relation 
\begin{align*}
1 = x_1^q [ x_1,x_2] [x_3,x_4] \cdots [x_{d-1}, x_d].
\end{align*}
Then the canonical class of $G$ is trivial, and $G$ is $A_3$-formal. 
\end{theorem}

For the rest of this section, we let $G$ be the pro-$2$ group as in Theorem \ref{thm:torsion_free_pro-2_Demushkin_groups_are_A_3-formal}. 
%
%
Let $\Cb = \Cb(G,\F_2)$ denote the complex of continuous inhomogeneous cochains, let $\Zh^\bbb$ denote the cocycles, 
and let $\Hb = \Hb(G,\F_2)$ denote the corresponding cohomology algebra. 
We consider $\Ch^n$ and $\Zh^n$ as $\F_2$-vector spaces with addition and multiplication defined pointwise. 
%
We let $\chi_i \colon G \to \F_2$ denote the group homomorphism determined by $\chi_i(x_j) = \delta_{ij}$, where $\delta_{ij}$ denotes the Kronecker symbol. 
The collection $\{ \chi_1,\ldots, \chi_d \}$ forms a basis of $H^1$. 
%
We write $J \coloneqq \{1, \ldots, d\}$, $J/2 \coloneqq \{1, \ldots, d/2 \}$, 
and we set $\Jw/2 \coloneqq (J/2) \setminus \{1\} = \{2, \ldots, d/2\}$.  
%
In Section \ref{subsec:Koszul_complex_p=2_trosionfree}, 
we determine the complex which computes the Hochschild cohomology of $\Hb$. 
In Section \ref{subsec:Koszul_complex_p=2_trosionfree_differential}, 
we determine the effect of the corresponding differential. 
In Section \ref{subsec:construction_kappa}, we give a concrete construction of a map which represents the canonical class of $G$. 
In Section \ref{sec:computing_kappa3_torsionfree}, we finish the proof of Theorem \ref{thm:torsion_free_pro-2_Demushkin_groups_are_A_3-formal}. 
%



\subsection{The Koszul--Hochschild complex}\label{subsec:Koszul_complex_p=2_trosionfree}

By \cite[Theorem 5.2]{MPQT}, 
the graded $\F_2$-cohomology algebra of $G$ is a Koszul algebra. 
In this section, we determine a basis for the $\F_2$-vector spaces $R$ such that $\Hb = T(H^1)/R$
and of $K^3_3(\Hb)$. 


\begin{proposition}\label{prop:type_I_basis_2_for_Demushkin_even_generators}
The $\F_2$-vector subspace $R \subset H^1 \ot H^1$ such that $\Hb = T(H^1)/(R)$ 
admits the set $\BBB_2$ as a basis where  
\begin{align*}
\BBB_2 \coloneqq \{ & \chi_i \ot \chi_j ~ \text{for all} ~ (i,j) \notin \{(2k-1,2k), (2k,2k-1) :  k \in J/2 \}, \\
& \chi_1 \ot \chi_2 + \chi_{2i} \ot \chi_{2i-1} ~ \text{for} ~ i \in J/2, 
\chi_{2i-1} \ot \chi_{2i} + \chi_{2i} \ot \chi_{2i-1} ~ \text{for} ~ i \in \Jw/2  \}. 
\end{align*}
\end{proposition} 
\begin{proof}
This follows from the description of the cup product 
on $\Hb(G,\F_2)$ in Lemma \ref{lemma:relations_in_H2_all_groups}.  
Note that we have $\# \BBB_2 = d(d-1) + 2(d/2) - 1 = d^2 - 1$ as required for a basis 
by definition of a Demushkin group. 
\end{proof}



\begin{proposition}\label{prop:type_I_basis_for_Demushkin_even_generators}
The set 
$\BBB_3 \coloneqq \SSS \cup \DDD \cup \DDDw \cup \TTT$  
is a basis for the $\F_2$-vector space $K^3_3(\Hb)$ 
where the sets $\SSS$ (single terms), $\DDD$, $\DDDw$ (double sums), and $\TTT$ (triple sums) are 
\begin{align*}
\SSS & = \left\{ \chi_i \ot \chi_j \ot \chi_k 
~ \text{with}~ \begin{cases}
i, k \ne j+1 &  \text{if} ~ j ~ \text{is odd} \\
i, k \ne j-1 &  \text{if} ~ j ~ \text{is even}; 
\end{cases} 
\right\}, 
\end{align*}
\begin{align*}
\DDD = & ~ \{ \chi_k \ot (\chi_{2i-1} \ot \chi_{2i} + \chi_{2i} \ot \chi_{2i-1}), \\
& ~ ~(\chi_{2i-1} \ot \chi_{2i} + \chi_{2i} \ot \chi_{2i-1}) \ot \chi_k ~ \text{for} ~ i \in J/2, ~ k \ne 2i-1, 2i, \\
& ~ ~ \chi_k \ot (\chi_{1} \ot \chi_{2} + \chi_{2i} \ot \chi_{2i-1}), \\
& ~ ~(\chi_{1} \ot \chi_{2} + \chi_{2i} \ot \chi_{2i-1}) \ot \chi_k ~ \text{for} ~ i \in \Jw/2, ~ k \ne 1, 2, 2i-1, 2i
\} 
\end{align*}
\begin{align*}
\DDDw = & ~ \{ \chi_{1} \ot (\chi_{1} \ot \chi_{2} + \chi_{2i} \ot \chi_{2i-1}), 
\chi_{2i} \ot (\chi_{1} \ot \chi_{2} + \chi_{2i} \ot \chi_{2i-1}), \\
& ~ \chi_{2} \ot (\chi_{2} \ot \chi_{1} + \chi_{2i-1} \ot \chi_{2i}), 
\chi_{2i-1} \ot (\chi_{2} \ot \chi_{1} + \chi_{2i-1} \ot \chi_{2i}), \\
& ~(\chi_{2} \ot \chi_{1} + \chi_{2i-1} \ot \chi_{2i}) \ot \chi_1, 
(\chi_{2} \ot \chi_{1} + \chi_{2i-1} \ot \chi_{2i}) \ot \chi_{2i},  \\ 
& ~ (\chi_{1} \ot \chi_{2} + \chi_{2i} \ot \chi_{2i-1}) \ot \chi_2,  
(\chi_{1} \ot \chi_{2} + \chi_{2i} \ot \chi_{2i-1}) \ot \chi_{2i-1}  ~ \text{for} ~ i \in \Jw/2
\}.
\end{align*}
\begin{align*}
\TTT = & ~ \{ \chi_{2i} \ot \chi_{2i-1} \ot \chi_{2i-1} + \chi_{2i-1} \ot \chi_{2i} \ot \chi_{2i-1} + \chi_{2i-1} \ot \chi_{2i-1} \ot \chi_{2i}, \\
& ~ ~\chi_{2i-1} \ot \chi_{2i} \ot \chi_{2i} + \chi_{2i} \ot \chi_{2i-1} \ot \chi_{2i} + \chi_{2i} \ot \chi_{2i} \ot \chi_{2i-1} ~ \text{for} ~ i \in J/2 
\}. 
\end{align*}
\end{proposition}
\begin{proof}
Proposition \ref{prop:type_I_basis_2_for_Demushkin_even_generators} 
shows that $R$ has the same form as for pro-$p$ Demushkin groups for an odd prime $p$. 
Since $K^3_3(\Hb)$ only depends on the quadratic relations defining $\Hb$, 
the same arguments as in \cite[Lemma 7.2]{PQ3} 
then prove the assertion that $\BBB_3$ is a basis of $K^3_3(\Hb)$.  
(Note that we explain the argument in the other cases below.)  
\end{proof}


\subsection{The differential in the Koszul--Hochschild complex}\label{subsec:Koszul_complex_p=2_trosionfree_differential}

Now, we analyse the differential 
$\partial \colon \Hom_{\F_2}(K_2^2(\Hb),H^1)  \to \Hom_{\F_2}(K^3_3(\Hb),H^2)$ 
in the complex \eqref{eq:Koszul_complex_3,-1} 
which computes the Hochschild cohomology of $\Hb$.  
%
For any $\sigma \in \Hom_{\F_2}(K_2^2(\Hb), H^1)$, 
$\partial \sigma$ vanishes on elements in $\TTT$ 
and on elements in the set $\Delta_{\SSS} \coloneq \{\chi_i \ot \chi_i \ot \chi_i : i \in J\} \subseteq \SSS$. 
%
In particular, for $d=2$, the differential is trivial. 
For $d \ge 4$, however, $\partial$ is non-trivial as we now show. 
%
First, we note that it is easy to see that for every element $(\chi_a,\chi_b,\chi_c)$ in $\SSS \setminus \Delta_{\SSS} \cup \DDD$, 
there is a $\sigma \in \BBB^2_2$ such that $\partial \sigma (\chi_a,\chi_b,\chi_c) \ne 0$. 
The same $\partial \sigma$ may be non-trivial on several elements of $\BBB^3_3$ simultaneously. 
Since we will show in Section \ref{sec:computing_kappa3_torsionfree} 
that the map which represents the canonical class of $G$ vanishes on the sets $\SSS$ and $\DDD$, 
we only compute the effect on the set $\DDDw$.


\begin{notn}\label{notn:chi_ijk}
We will often write $\chi_{i,j}$ for $\chi_i \ot \chi_j \in H^1 \ot H^1$ 
and $\chi_{i,j,k}$ for $\chi_i \ot \chi_j \ot \chi_k \in H^1 \ot H^1 \ot H^1$ 
when it makes formulas easier to fit into the text. 
\end{notn}


\begin{notn}\label{notn:sigma}
Let $\chi_{i,j} \in \BBB_2$, 
i.e., $(i,j) \ne (2k-1,2k), (2k,2k-1)$. 
For $s \in J$, we write $\sigma^s_{i,j}$ for the unique $\F_2$-linear map $R \to H^1$ 
which sends $\chi_{i,j}$ to $\chi_s$ and 
every other element of $\BBB_2$ to $0$. 
For $\chi_1 \ot \chi_2 + \chi_{2i} \ot \chi_{2i-1} \in \BBB_2$, $i \in J/2$, and $s \in J$, 
we write  $\sigma^s_{1,2'+'2i,2i-1}$ for the unique $\F_2$-linear map $R \to H^1$ 
which sends $\chi_1 \ot \chi_2 + \chi_{2i} \ot \chi_{2i-1}$ to $\chi_s$ and 
every other element of $\BBB_2$ to $0$. 
Similarly, for $\chi_{2i-1} \ot \chi_{2i} + \chi_{2i} \ot \chi_{2i-1} \in \BBB_2$, 
we write  $\sigma^s_{2i-1,2i'+'2i,2i-1}$ for the unique $\F_2$-linear map $R \to H^1$ 
which sends $\chi_{2i-1} \ot \chi_{2i}  + \chi_{2i} \ot \chi_{2i-1}$ to $\chi_s$ and 
every other element of $\BBB_2$ to $0$. 
\end{notn} 


\begin{notn}\label{notn:dual_notation_type_I}
For an element $(\chi_a,\chi_b,\chi_c) \in \BBB_3$, 
we write $(\chi_a,\chi_b,\chi_c)^*$ 
for the unique 
$\F_2$-linear map in $\Hom_{\F_2}(K^3_3(\Hb),H^2)$ such that 
$(\chi_a,\chi_b,\chi_c)^*(\chi_a,\chi_b,\chi_c)\ne 0$ 
and 
$(\chi_a,\chi_b,\chi_c)^*$ sends all other elements in $\BBB_3$ to $0$. 
By slight abuse of terminology, using the isomorphism
$\Hom_{\F_2}(K^3_3(\Hb),H^2) \cong \Hom_{\F_2}(K^3_3(\Hb),\F_2)$ induced by the isomorphism $H^2 \cong \F_2$, 
we refer to $(\chi_a,\chi_b,\chi_c)^*$ as the {\em dual of $(\chi_a,\chi_b,\chi_c)$}. 
We write $\BBB_3^*$ for the set of $(\chi_a,\chi_b,\chi_c)^*$ for all elements in $\BBB_3$.  
\end{notn} 


\begin{lemma}\label{lemma:image_of_partial_sigmas_DDDw}
Assume $d \ge 4$ and let $i \in \Jw/2$. 
Then 
\begin{align*}
(\chi_{2,1,2i} + \chi_{2i-1,2i,2i})^* & =  \partial\sigma^{2i-1}_{2i-1,2i'+'2i,2i-1}, \\
(\chi_{2i-1,2,1} + \chi_{2i-1,2i-1,2i})^* & =  \partial\sigma^{2i}_{2i-1,2i'+'2i,2i-1}, \\
(\chi_{2i,1,2} + \chi_{2i,2i,2i-1})^* & =  \partial\sigma^{2i-1}_{2i-1,2i'+'2i,2i-1} + \partial\sigma^{2i-1}_{1,2'+'2i,2i-1}, \\
\text{and} ~ 
(\chi_{1,2,2i-1} + \chi_{2i,2i-1,2i-1})^* & =  \partial\sigma^{2i}_{2i-1,2i'+'2i,2i-1} + \partial\sigma^{2i}_{1,2'+'2i,2i-1}. 
\end{align*}
Moreover, we have  
\begin{align*}
\sum_{j \in \Jw/2} (\chi_{2,2,1} + \chi_{2,2j-1,2j})^*  & =  \partial\sigma^1_{1,2'+'2,1}, 
\sum_{j \in \Jw/2} (\chi_{2,1,1} + \chi_{2j-1,2j,1})^*  = \partial\sigma^2_{1,2'+'2,1},  \\
\sum_{j \in \Jw/2} (\chi_{1,2,2} + \chi_{2j,2j-1,2})^* & =  \partial\sigma^1_{1,2'+'2,1} + \sum_{j \in \Jw/2} \partial\sigma^1_{1,2'+'2j,2j-1}, 
\end{align*}
and 
\begin{align*}
\sum_{j \in \Jw/2} (\chi_{1,1,2} + \chi_{1,2j,2j-1})^* & =  \partial\sigma^2_{1,2'+'2,1} + \sum_{j \in \Jw/2} \partial\sigma^2_{1,2'+'2j,2j-1}. 
\end{align*}
\end{lemma}
\begin{proof}
%
First, we consider $\partial\sigma^{2i}_{2i-1,2i'+'2i,2i-1}$ 
and compute  
\begin{align*}
& \partial\sigma^{2i}_{2i-1,2i'+'2i,2i-1}(\chi_{2i-1,2,1} + \chi_{2i-1,2i-1,2i}) \\
= & ~~ \chi_{2i-1} \cup \sigma^{2i}_{2i-1,2i'+'2i,2i-1}(\chi_{2,1} + \chi_{2i-1,2i})\\ 
& ~~ + \sigma^{2i}_{2i-1,2i'+'2i,2i-1}(\chi_{2i-1,2}) \cup \chi_1  + \sigma^{2i}_{2i-1,2i'+'2i,2i-1}(\chi_{2i-1,2i-1}) \cup \chi_{2i}. 
\end{align*}
By definition of $\sigma^{2i}_{2i-1,2i+2i,2i-1}$,  
we know that  $\sigma^{2i}_{2i-1,2i+2i,2i-1}(\chi_{2i-1,2}) = 0$ 
and $\sigma^{2i}_{2i-1,2i+2i,2i-1}(\chi_{2i-1,2i-1}) = 0$. 
Moreover, we have 
\begin{align}\label{eq:relation_2,1+2i-1,2i}
\chi_{2,1} + \chi_{2i-1,2i} =  (\chi_{1,2} + \chi_{2,1}) + (\chi_{2i-1,2i} + \chi_{2i,2i-1}) + (\chi_{1,2} + \chi_{2i,2i-1}).  
\end{align}
By definition of $\sigma^{2i}_{2i-1,2i+2i,2i-1}$,   
we have 
$\sigma^{2i}_{2i-1,2i'+'2i,2i-1}(\chi_{1,2} + \chi_{2,1}) = 0$ 
and $\sigma^{2i}_{2i-1,2i'+'2i,2i-1}(\chi_{1,2} + \chi_{2i,2i-1}) = 0$, 
and hence 
\begin{align*}
 & \partial\sigma^{2i}_{2i-1,2i'+'2i,2i-1}(\chi_{2i-1,2,1} + \chi_{2i-1,2i-1,2i}) \\
= &~~  \chi_{2i-1} \cup \sigma^{2i}_{2i-1,2i'+'2i,2i-1}(\chi_{2i-1,2i} + \chi_{2i,2i-1}) 
= \chi_{2i-1} \cup \chi_{2i} \ne 0.
\end{align*} 
It remains to show that $\partial\sigma^{2i}_{2i-1,2i'+'2i,2i-1}$ vanishes on all other elements of $\BBB_3$. 
%
%
For $\chi_{a,b,c} \in \SSS$, we have $\sigma^{2i}_{2i-1,2i'+'2i,2i-1}(\chi_{a,b}) = 0 = \sigma^{2i}_{2i-1,2i'+'2i,2i-1}(\chi_{b,c})$ 
and hence $\partial\sigma^{2i}_{2i-1,2i'+'2i,2i-1}(\chi_{a,b,c}) = 0$. 
Now let $b \in J/2$ and $a \in J$ with $a \ne 2b-1, 2b$. 
Then 
\begin{align*}
 & \partial\sigma^{2i}_{2i-1,2i'+'2i,2i-1} (\chi_{a,2b-1,2b} + \chi_{a,2b,2b-1}) \\
= & ~~~  \chi_{a} \cup \sigma^{2i}_{2i-1,2i'+'2i,2i-1} (\chi_{2b-1,2b} + \chi_{2b,2b-1}) \\
& + \sigma^{2i}_{2i-1,2i'+'2i,2i-1} (\chi_{a,2b-1}) \cup \chi_{2b} + \sigma^{2i}_{2i-1,2i'+'2i,2i-1} (\chi_{a,2b}) \cup \chi_{2b-1}. 
\end{align*} 
By definition of $\sigma^{2i}_{2i-1,2i'+'2i,2i-1}$, 
the second and third summands are zero, 
and we have $\sigma^{2i}_{2i-1,2i'+'2i,2i-1}(\chi_{2b-1,2b} + \chi_{2b,2b-1}) \ne 0$ 
if and only if $b=i$ by Lemma \ref{lemma:relations_in_H2_all_groups}.  
However, 
 the cup-product $\chi_a \cup \chi_{2i}$ is non-zero 
if and only if $a=2i-1$. 
But we have $a \ne 2b-1$ 
by assumption. 
This shows $\partial\sigma^{2i}_{2i-1,2i'+'2i,2i-1} (\chi_{a,2b-1,2b} + \chi_{a,2b,2b-1}) = 0$. 
Similarly, we get $\partial\sigma^{2i}_{2i-1,2i'+'2i,2i-1} (\chi_{2b-1,2b,a} + \chi_{2b,2b-1,a}) = 0$. 
We can check that $\partial\sigma^{2i}_{2i-1,2i'+'2i,2i-1}$ vanishes on the remaining elements of $\DDD$. 
For another example of an element in $\DDDw$, we compute, for $b \in \Jw/2$, 
\begin{align*}
& \partial\sigma^{2i}_{2i-1,2i'+'2i,2i-1} (\chi_{2,1,1} + \chi_{2b-1,2b,1}) \\
= & ~~ \chi_2 \cup \sigma^{2i}_{2i-1,2i'+'2i,2i-1}(\chi_{1,1}) + \chi_{2b-1} \cup \sigma^{2i}_{2i-1,2i'+'2i,2i-1}(\chi_{2b,1}) \\
 & ~~ + \sigma^{2i}_{2i-1,2i'+'2i,2i-1}(\chi_{2,1} + \chi_{2b-1,2b}) \cup \chi_1 \\
= & ~~  \sigma^{2i}_{2i-1,2i'+'2i,2i-1}(\chi_{2b-1,2b} + \chi_{2b,2b-1})  \cup \chi_1
\end{align*} 
where we use again \eqref{eq:relation_2,1+2i-1,2i}. 
The last cup product is non-zero if and only if $b=i$ and $i=1$ by Lemma \ref{lemma:relations_in_H2_all_groups}. 
However, we have $b \ge 2$ by assumption. 
%
Similar computations show that $\partial\sigma^{2i}_{2i-1,2i'+'2i,2i-1}$ vanishes on the remaining elements in $\DDDw$ as well. 
Since $\partial \sigma$ vanishes on $\TTT$ for all $\sigma$, 
this proves the assertion for $\partial\sigma^{2i}_{2i-1,2i'+'2i,2i-1}$. 


Second, we consider $\partial\sigma^1_{1,2'+'2,1}$. 
Let $j \in \Jw/2$. 
By definition of $\sigma^1_{1,2'+'2,1}$, 
we have $\sigma^1_{1,2'+'2,1}(\chi_{2,2}) = 0$ and $\sigma^1_{1,2'+'2,1}(\chi_{2,2j-1}) = 0$. 
Hence we get 
\begin{align*}
& \partial\sigma^1_{1,2'+'2,1} (\chi_{2,2,1} + \chi_{2,2j-1,2j}) \\
= & ~~ \chi_2 \cup \sigma^1_{1,2'+'2,1}(\chi_{2,1} + \chi_{2j-1,2j}) 
 + \sigma^1_{1,2'+'2,1}(\chi_{2,2}) \cup \chi_{1} + \sigma^1_{1,2'+'2,1}(\chi_{2,2j-1}) \cup \chi_{2j} \\
= & ~~ \chi_2 \cup \sigma^1_{1,2'+'2,1}(\chi_{2,1} + \chi_{2j-1,2j}). 
\end{align*}
Now we use again  \eqref{eq:relation_2,1+2i-1,2i}. 
We have 
$\sigma^1_{1,2'+'2,1} (\chi_{2j-1,2j} + \chi_{2j,2j-1}) = 0$ 
and $\sigma^1_{1,2'+'2,1} (\chi_{1,2} + \chi_{2j,2j-1}) = 0$ 
since $j \ge 2$ by assumption. 
Thus, for every $j \in \Jw/2$, 
\begin{align*}
\partial\sigma^1_{1,2'+'2,1} (\chi_{2,2,1} + \chi_{2,2j-1,2j})  = \chi_2 \cup \sigma^1_{1,2'+'2,1}(\chi_{1,2} + \chi_{2,1}) = \chi_2 \cup \chi_1 \ne 0.
\end{align*} 
%
It remains to show that $\partial\sigma^1_{1,2'+'2,1}$ vanishes on all other elements of $\BBB^3_3$. 
For $\chi_{a,b,c} \in \SSS$, we have $\sigma^1_{1,2'+'2,1}(\chi_{a,b}) = 0 = \sigma^1_{1,2'+'2,1}(\chi_{b,c})$ 
and hence $\partial\sigma^1_{1,2'+'2,1}(\chi_{a,b,c}) = 0$. 
Now let $b \in J/2$ and $a \in J$ with $a \ne 2b-1, 2b$. 
Then 
\begin{align*}
 & \partial\sigma^1_{1,2'+'2,1} (\chi_{a,2b-1,2b} + \chi_{a,2b,2b-1}) \\
= & ~~~  \chi_{a} \cup \sigma^1_{1,2'+'2,1} (\chi_{2b-1,2b} + \chi_{2b,2b-1}) \\
& + \sigma^1_{1,2'+'2,1} (\chi_{a,2b-1}) \cup \chi_{2b} + \sigma^1_{1,2'+'2,1} (\chi_{a,2b}) \cup \chi_{2b-1}. 
\end{align*} 
By definition of $\sigma^1_{1,2'+'2,1}$, 
the second and third summands are zero, 
and we have $\sigma^1_{1,2'+'2,1} (\chi_{2b-1,2b} + \chi_{2b,2b-1}) \ne 0$ 
if and only if $b=1$. 
However, the cup-product $\chi_a \cup \chi_1$ is non-zero 
if and only if $a=2$ by Lemma \ref{lemma:relations_in_H2_all_groups}. 
But we have $a \ne 2b = 2$ by assumption. 
This shows $\partial\sigma^1_{1,2'+'2,1} (\chi_{a,2b-1,2b} + \chi_{a,2b,2b-1}) = 0$. 
Similarly, we get $\partial\sigma^1_{1,2'+'2,1} (\chi_{2b-1,2b,a} + \chi_{2b,2b-1,a}) = 0$. 
Now we assume $b \in \Jw/2$. 
Then $\partial\sigma^1_{1,2'+'2,1}$ sends both 
$\chi_{a,1,2} + \chi_{a,2b,2b-1}$ and $\chi_{1,2,a} + \chi_{2b,2b-1,a}$ 
to zero, 
since $\sigma^1_{1,2'+'2,1}(\chi_{1,2} + \chi_{2b,2b-1}) = 0$ 
for $b \ge 2$. 
It remains to show that $\partial\sigma^1_{1,2'+'2,1}$ vanishes on the other elements in $\DDDw$. 
%
%
By relation \eqref{eq:relation_2,1+2i-1,2i}, 
we get $\sigma^1_{1,2'+'2,1} (\chi_{2,1} + \chi_{2b-1,2b}) = \chi_1$. 
However, since $\chi_1 \cup \chi_1 = \chi_1 \cup \chi_{2b} = \chi_{2b-1} \cup \chi_1 = 0$, 
we conclude that $\partial\sigma^1_{1,2'+'2,1}$ vanishes on 
the elements 
$(\chi_{2,1,1} + \chi_{2b-1,2b,1})$, 
$(\chi_{2,1,2b} + \chi_{2b-1,2b,2b})$, 
and $(\chi_{2b-1,2,1} + \chi_{2b-1, 2b-1,2b})$ in $\DDDw$. 
Since $\sigma^1_{1,2'+'2,1} (\chi_{1,2} + \chi_{2b,2b-1}) = 0$ for $b\ge 2$, 
we get that  $\partial\sigma^1_{1,2'+'2,1}$ also sends  
the elements 
$\chi_{1,1,2} + \chi_{1, 2b,2b-1}$, 
$\chi_{1,2,2} + \chi_{2b,2b-1,2}$, 
$\chi_{2b,1,2} + \chi_{2b, 2b,2b-1}$, and
$\chi_{1,2,2b-1} + \chi_{2b,2b-1,2b-1}$  
in $\DDDw$ to zero. 
This proves that $\partial\sigma^1_{1,2'+'2,1}$ is non-zero exactly on all elements of the form 
$(\chi_{2,2,1} + \chi_{2,2j-1,2j})^*$ for all $j \in \Jw/2$. 
This proves the claim for $\partial\sigma^1_{1,2'+'2,1}$. 
The remaining assertions of the lemma follow from similar computations. 
\end{proof}


For the next proposition, we recall that every map $\kappa \in \Hom_{\F_2}(K^3_3(\Hb), H^2)$ is a cocycle in 
the complex $(\uHom_{\F_2}(K_{\bbb}^{\bbb}(\Hb), \Hb[-1]), \dee)$ which computes the Hochschild cohomology group 
$\HH^{3,-1}(\Hb)$ since $H^3=0$. 

\begin{proposition}\label{prop:image_of_partial_sigmas_DDDw}
Let $\DDDw^*_{\Sigma} \coloneqq \DDDw^*_1 \oplus \DDDw^*_2$ denote the subspace of $\Hom_{\F_2}(K^3_3(\Hb), H^2)$ 
given by the sum of 
\begin{align*}
\DDDw^*_1 \coloneqq \mathrm{Span}_{\F_2}\{ &  (\chi_{2,1,2i} + \chi_{2i-1,2i,2i})^*,  (\chi_{2i-1,2,1} + \chi_{2i-1,2i-1,2i})^*,  \\
& (\chi_{2i,1,2} + \chi_{2i,2i,2i-1})^*, (\chi_{1,2,2i-1} + \chi_{2i,2i-1,2i-1})^*: ~ i \in \Jw/2 \} 
\end{align*}
and
\begin{align*}
\DDDw^*_2 \coloneqq \mathrm{Span}_{\F_2} \{ &  \sum_{j \in \Jw/2} (\chi_{2,2,1} + \chi_{2,2j-1,2j})^*,  \sum_{j \in \Jw/2} (\chi_{2,1,1} + \chi_{2j-1,2j,1})^*,  \\
& \sum_{j \in \Jw/2} (\chi_{1,2,2} + \chi_{2j,2j-1,2})^*, \sum_{j \in \Jw/2} (\chi_{1,1,2} + \chi_{1,2j,2j-1})^* \}. 
\end{align*}
Then $\DDDw^*_{\Sigma} \subset \Imm(\partial)$. 
In particular, if $\kappa \in \Hom_{\F_2}(K^3_3(\Hb), H^2)$ sends $\SSS$, $\DDD$, and $\TTT$ to zero,  
and has the same values on $(\chi_{2,2,1} + \chi_{2,2i-1,2i})$, 
$(\chi_{2,1,1} + \chi_{2i-1,2i,1})$, 
$(\chi_{1,2,2} + \chi_{2i,2i-1,2})$, 
and 
$(\chi_{1,1,2} + \chi_{1,2i,2i-1})$, 
respectively, 
for all $i \in \Jw/2$, 
then $\kappa$ is a coboundary and $[\kappa] = 0$ in $\HH^{3,-1}(\Hb)$. 
\end{proposition}
\begin{proof} 
This follows directly from Lemma \ref{lemma:image_of_partial_sigmas_DDDw}. 
\end{proof}

\begin{corollary}\label{cor:image_of_partial_sigmas_DDDw}
Let $\kappa \in \Hom_{\F_2}(K^3_3(\Hb), H^2)$, 
and assume that $\kappa$ sends $\SSS$, $\DDD$, and $\TTT$ to zero,  
and has the same values on $(\chi_{2,2,1} + \chi_{2,2i-1,2i})$, 
$(\chi_{2,1,1} + \chi_{2i-1,2i,1})$, 
$(\chi_{1,2,2} + \chi_{2i,2i-1,2})$, 
and 
$(\chi_{1,1,2} + \chi_{1,2i,2i-1})$, 
respectively.  
Then $\kappa$ is a coboundary and $[\kappa] = 0$ in $\HH^{3,-1}(\Hb)$. 
\end{corollary}
\begin{proof} 
This follows directly from Proposition \ref{prop:image_of_partial_sigmas_DDDw}.
\end{proof}


\subsection{Construction of the canonical class}\label{subsec:construction_kappa}

We will now construct the canonical class of $G$.  
Let $\Hom(G,\F_2)$ denote the $\F_2$-vector space of continuous group homomorphisms. 
We choose $f_1$ to be the identity $\Hom(G,\F_2) = H^1 \to  \Zh^1 = \Hom(G,\F_2)$ and omit it from the notation. 
We construct an $\F_2$-linear map $f_2 \colon R \to \Ch^1$  by defining it on each element of $\BBB_2$ and then extend it $\F_2$-linearly.

Let $A_{10}
= \begin{psmallmatrix} 
1 & 1 & 0  \\
0 & 1 & 0  \\
0 &  0 & 1  \\
\end{psmallmatrix}$,  
$A_{01}
= \begin{psmallmatrix} 
1 & 0 & 0  \\
0 & 1 & 1  \\
0 &  0 & 1  \\
\end{psmallmatrix}$, 
and 
$A_{11}
= \begin{psmallmatrix} 
1 & 1 & 0  \\
0 & 1 & 1  \\
0 &  0 & 1  \\
\end{psmallmatrix}$  
denote the matrices used in the proof of Lemma \ref{lemma:relations_in_H2_all_groups}.  
Let $i \ne j$ and $(i,j) \ne (2k-1,2k), (2k,2k-1)$. 
As in the proof of Lemma \ref{lemma:relations_in_H2_all_groups},  
we define the continuous homomorphism $\rr_{i,j} \colon G \to U_3(\F_2)$ 
by the assignment $x_i \mapsto A_{10}$, 
$x_{j} \mapsto A_{01}$, 
and $x_s \mapsto I_3$ for $s \ne i, j$. 
We define the continuous map $f_2(\chi_i \ot \chi_j) \colon G \to \F_2$  
by setting $f_2(\chi_i \ot \chi_j) \coloneqq  e_{1,3} \circ \rr_{i,j}$. 
%
We define the continuous homomorphism $\rr_{i,i} \colon G \to U_3(\F_2)$ 
by the assignment $x_i \mapsto A_{11}$, 
and $x_j \mapsto I_3$ for $j \ne i$. 
We define the continuous map $f_2(\chi_i \ot \chi_i) \colon G \to \F_2$  
by setting $f_2(\chi_i \ot \chi_i) \coloneqq  e_{1,3} \circ \rr_{i,i}$. 
Now let $i\in J/2$. 
We will define a continuous map $\eta_{2i-1,2i'+'2i,2i-1} \colon G \to \F_2$ 
such that $\delta \eta_{2i-1,2i'+'2i,2i-1} = \chi_{2i-1} \cup \chi_{2i} + \chi_{2i} \cup \chi_{2i-1}$. 
To do so, we use Dwyer's Theorem \ref{thm:Dwyer} and the following notation. 

\begin{notn}\label{notn:matrix_group_hom}
By slight abuse of notation, we will often write a continuous homomorphism $\rr \colon G \to U_n(\F_2)$ 
as an $(n \times n)$-matrix with $(i,j)$-entry the map $e_{i,j} \circ \rr$. 
For example, for $n=4$, we write 
$\rr = \begin{psmallmatrix}
 1 & a_{1,2} & a_{1,3}  & a_{1,4}  \\
  & 1& a_{2,3}  & a_{2,4}    \\
  & & 1 & a_{3,4}  \\
  & & & 1 
\end{psmallmatrix}$ 
where each $a_{i,j} \colon G \to \F_2$ is a continuous map, 
and $a_{1,2}$, $a_{2,3}$, and $a_{3,4}$ are group homomorphisms.  
\end{notn}


\begin{notn}\label{notation:E_matrices}
Let $n\ge 3$. 
For $1 \le s < t \le n$, 
we let $E_{s,t}$ denote the matrix in $U_n(\F_2)$ with only non-zero entry above the diagonal in row $s$ and column $t$. 
We have $(E_{s,t})^q = I_n$ in $U_n(\F_2)$ for each pair $(s,t)$. 
We note that it will be clear from the context which $n$ we consider and we therefore omit it from notation for the matrices $E_{s,t}$. 
\end{notn}


To construct the cochain $\eta_{2i-1,2i'+'2i,2i-1}$ we use 
Dwyer's Theorem \ref{thm:Dwyer}. 
We will define a continuous homomorphism $\rr_{2i-1,2i'+'2i,2i-1} \colon G \to U_4(\F_2)$ 
which lifts the continuous homomorphism
$\brr_{2i-1,2i'+'2i,2i-1}  \colon G \to \bU_4(\F_2)$ given by 
$\brr_{2i-1,2i'+'2i,2i-1} = 
\begin{psmallmatrix}
 1 & \chi_{2i-1} & \chi_{2i} &  * \\
  & 1& 0 & \chi_{2i}  \\
  & & 1 &  \chi_{2i-1} \\
  & & & 1 
\end{psmallmatrix}$.  
To construct the desired lift, 
consider the matrices 
\begin{align*}
A \coloneqq  
\begin{psmallmatrix} 
1 & 1 & 0 & 0  \\
0 & 1 & 0 & 0  \\
0 & 0 & 1 & 1 \\
0 & 0 & 0 & 1 \\
\end{psmallmatrix}, ~ \text{and} ~ 
B \coloneqq
\begin{psmallmatrix} 
1 & 0 & 1 & 0  \\
0 & 1 & 0 & 1  \\
0 & 0 & 1 & 0 \\
0 & 0 & 0 & 1 \\
\end{psmallmatrix}.
\end{align*}
We have $A^q = I_4$ and $[A,B] = I_4$ in $U_4(\F_2)$. 
Thus, by Lemma \ref{lemma:matrix_relations}, 
we can define $\rr_{2i-1,2i'+'2i,2i-1}$  
by the assignment 
$x_{2i-1} \mapsto A$,   
$x_{2i} \mapsto B$,  
and $x_s \mapsto I_4$ for $s \ne 2i-1, 2i$. 
We define the desired continuous map $G \to \F_2$  
by setting 
\[
f_2(\chi_{2i-1} \ot \chi_{2i} + \chi_{2i} \ot \chi_{2i-1}) \coloneqq  e_{1,4} \circ \rr_{2i-1,2i'+'2i,2i-1}.
\] 
Similarly, for $i\in \Jw/2$, 
we construct a continuous homomorphism $\rr_{1,2'+'2i,2i-1} \colon G \to U_4(\F_2)$ 
which lifts the continuous homomorphism
$\brr_{1,2'+'2i,2i-1}  \colon G \to \bU_4(\F_2)$ given by 
$\brr_{1,2'+'2i,2i-1} = 
\begin{psmallmatrix}
 1 & \chi_{1} & \chi_{2i} &  * \\
  & 1& 0 & \chi_{2}  \\
  & & 1 &  \chi_{2i-1} \\
  & & & 1 
\end{psmallmatrix}$. 
We have $[E_{1,2}, E_{2,4}] = [E_{3,4}, E_{1,3}] = E_{1,4}$
and hence $[E_{1,2}, E_{2,4}] [E_{3,4}, E_{1,3}] = I_4$ in $U_4(\F_2)$. 
Thus, by Lemma \ref{lemma:matrix_relations}, 
we can define  $\rr_{1,2'+'2i,2i-1} \colon G \to U_4(\F_2)$ 
by the assignment 
$x_1 \mapsto E_{1,2}$, 
$x_2 \mapsto E_{2,4}$, 
$x_{2i-1} \mapsto E_{3,4}$, 
$x_{2i} \mapsto E_{1,3}$, 
and $x_s \mapsto I_4$ for $s \ne 1, 2, 2i-1, 2i$. 
We then define the desired continuous map $G \to \F_2$  
by  
\[
f_2(\chi_{1} \ot \chi_{2} + \chi_{2i} \ot \chi_{2i-1}) \coloneqq  e_{1,4} \circ \rr_{1,2'+'2i,2i-1}.
\] 
We now define the map $f_2 \colon R \to \Ch^1$ on all of $R$ by extending it $\F_2$-linearly from $\BBB_2$ to $R$. 
We then deduce from Proposition \ref{prop:kappa3_is_canonical_class}: 

\begin{prop}\label{prop:canonical_class_of_Demushkin_group_construction} 
We define the map $\Psi_3 \colon K_3^3(\Hb) \to \Zh^2$ by  
\begin{align}\label{eq:def_of_Psi3_Demushkin}
\Psi_3(\chi_a, \chi_b, \chi_c) =  \chi_a \cup f_2(\chi_b, \chi_c) + f_2(\chi_a, \chi_b) \cup \chi_c. 
\end{align}
Taking the cohomology class of $\Psi_3$ defines an $\F_2$-linear map $\kappa_3 \colon K^3_3(\Hb) \to H^2$. 
The map $\kappa_3$ is a cocycle in the complex $(\uHom_{\F_2}(K_{\bbb}^{\bbb}(\Hb), \Hb[-1]), \dee)$, 
and the class of $\kappa_3$ in $\HH^{3,-1}(\Hb)$ is the canonical class of $G$. \qed
\end{prop}


\subsection{Two compatibility results}\label{subsec:compatibility}

For the proof of Theorem \ref{thm:torsion_free_pro-2_Demushkin_groups_are_A_3-formal} 
in Section \ref{sec:computing_kappa3_torsionfree} we need the compatibility statements of 
Propositions \ref{prop:f2_comparison_3to4} and \ref{prop:f2_comparison_3to5} we prove below. 
To formulate and prove these results, we need the following lemmas in linear algebra. 

\begin{lemma}\label{lemma:matrix_subgroup_isos_3to4_left} 
Let $\Tlt$ be the subset of $U_4(\F_2)$ consisting of matrices of the form 
$\begin{psmallmatrix} 
1 & x & v & z  \\
0 & 1 & w & y  \\
0 &  0 & 1  & 0 \\
0 & 0 & 0 & 1
\end{psmallmatrix}$  
and let $\Wlt$ denote the subset of matrices of the form 
$\begin{psmallmatrix} 
1 & 0 & v & 0  \\
0 & 1 & w & 0  \\
0 &  0 & 1  & 0 \\
0 & 0 & 0 & 1
\end{psmallmatrix}$. 
Then $\Tlt$ 
is a subgroup of $U_4(\F_2)$, 
and $\Wlt$ is a normal subgroup of $\Tlt$. 
Moreover, 
the map $\taul \colon U_3(\F_2) \to \Tlt/\Wlt$ 
defined by 
$\taul \colon \begin{psmallmatrix} 
1 & x &  z  \\
0 & 1 & y  \\
0 &  0 & 1 
\end{psmallmatrix}
\mapsto 
\begin{psmallmatrix} 
1 & x & * & z  \\
0 & 1 & * & y  \\
0 &  0 & 1  & 0 \\
0 & 0 & 0 & 1
\end{psmallmatrix}$  
is an isomorphism of groups. 
\end{lemma}
\begin{proof}
The following computation shows that $\Tlt$ is closed under the group operation:  
\begin{align}\label{eq:computation_Tlt_is_closed}
\begin{psmallmatrix} 
1 & x_1 & v_1 & z_1  \\
0 & 1 & w_1 & y_1  \\
0 &  0 & 1  & 0 \\
0 & 0 & 0 & 1
\end{psmallmatrix}
\cdot 
\begin{psmallmatrix} 
1 & x_2 & v_2 & z_2  \\
0 & 1 & w_2 & y_2  \\
0 &  0 & 1  & 0 \\
0 & 0 & 0 & 1
\end{psmallmatrix}
=
\begin{psmallmatrix} 
1 & x_1+x_2 & v_1 + v_2 + x_1w_2 & z_1+z_2+x_1y_2  \\
0 & 1 & w_1 + w_2 & y_1+y_2  \\
0 &  0 & 1  & 0 \\
0 & 0 & 0 & 1
\end{psmallmatrix}.
\end{align}
Since $U_4(\F_2)$ is a finite group, this implies that $\Tlt$ is a subgroup. 
We note  that the inverse of 
$\begin{psmallmatrix} 
1 & x & v & z  \\
0 & 1 & w & y  \\
0 &  0 & 1  & 0 \\
0 & 0 & 0 & 1
\end{psmallmatrix}$ in $\Tlt$ 
is given by 
$\begin{psmallmatrix} 
1 & x & v+xw & z+xy  \\
0 & 1 & w & y  \\
0 &  0 & 1  & 0 \\
0 & 0 & 0 & 1
\end{psmallmatrix}$. 
We also see from \eqref{eq:computation_Tlt_is_closed} that $\Wlt$ is a subgroup of $\Tlt$. 
To show that $\Wlt$ is a normal subgroup of $\Tlt$, we observe 
\begin{align*}
\begin{psmallmatrix} 
1 & x & v & z  \\
0 & 1 & w & y  \\
0 &  0 & 1  & 0 \\
0 & 0 & 0 & 1
\end{psmallmatrix}
\cdot 
\begin{psmallmatrix} 
1 & 0 & v' & 0  \\
0 & 1 & w' & 0  \\
0 &  0 & 1  & 0 \\
0 & 0 & 0 & 1
\end{psmallmatrix}
\cdot 
\begin{psmallmatrix} 
1 & x & v+xw & z+xy  \\
0 & 1 & w & y  \\
0 &  0 & 1  & 0 \\
0 & 0 & 0 & 1
\end{psmallmatrix} 
& = 
\begin{psmallmatrix} 
1 & x & v + v' + xw' & z  \\
0 & 1 & w + w' & y \\
0 &  0 & 1  & 0 \\
0 & 0 & 0 & 1
\end{psmallmatrix}
\cdot
\begin{psmallmatrix} 
1 & x & v+xw & z+xy  \\
0 & 1 & w & y  \\
0 &  0 & 1  & 0 \\
0 & 0 & 0 & 1
\end{psmallmatrix} \\ 
 & =  
 \begin{psmallmatrix} 
1 & 0 & v' + xw' & 0  \\
0 & 1 & w' & 0 \\
0 &  0 & 1  & 0 \\
0 & 0 & 0 & 1
\end{psmallmatrix} \in \Wlt.
\end{align*}
Since $\#U_3(\F_2) = \#(\Tlt/\Wlt)$, to show that $\taul$ is an isomorphism, it suffices to show that it is a homomorphism which 
follows from \eqref{eq:computation_Tlt_is_closed} and 
\begin{align}\label{eq:computation_U3_is_like_Tlt}
\begin{psmallmatrix} 
1 & x_1 &  z_1  \\
0 & 1 & y_1 \\
0 &  0 & 1 
\end{psmallmatrix} 
\cdot 
 \begin{psmallmatrix} 
1 & x_2 &  z_2  \\
0 & 1 & y_2  \\
0 &  0 & 1 
\end{psmallmatrix} 
= 
 \begin{psmallmatrix} 
1 & x_1+x_2 &  z_1+z_2 +x_1y_2  \\
0 & 1 & y_1+y_2  \\
0 &  0 & 1 
\end{psmallmatrix}
\end{align}
in $U_3(\F_2)$. 
%
\end{proof}


\begin{lemma}\label{lemma:matrix_subgroup_isos_3to4_right} 
Let $\Trt$ be the subset of $U_4(\F_2)$ consisting of matrices of the form 
$\begin{psmallmatrix} 
1 & 0 & x & z  \\
0 & 1 & w & v  \\
0 &  0 & 1  & y \\
0 & 0 & 0 & 1
\end{psmallmatrix}$  
and let $\Wrt$ denote the subset of matrices of the form 
$\begin{psmallmatrix} 
1 & 0 & 0 & 0  \\
0 & 1 & w & v  \\
0 &  0 & 1  & 0 \\
0 & 0 & 0 & 1
\end{psmallmatrix}$. 
Then $\Trt$ 
is a subgroup of $U_4(\F_2)$, 
and $\Wrt$ is a normal subgroup of $\Trt$. 
Moreover, 
the map $\taur \colon U_3(\F_2) \to \Trt/\Wrt$ 
defined by 
$\taur \colon \begin{psmallmatrix} 
1 & x &  z  \\
0 & 1 & y  \\
0 &  0 & 1 
\end{psmallmatrix}
\mapsto 
\begin{psmallmatrix} 
1 & 0 & x & z  \\
0 & 1 & * & *  \\
0 &  0 & 1  & y \\
0 & 0 & 0 & 1
\end{psmallmatrix}$  
is an isomorphism of groups. 
\end{lemma}
\begin{proof}
This follows from an analogous but symmetric computation as in the proof of Lemma \ref{lemma:matrix_subgroup_isos_3to4_left}. 
\end{proof}


\begin{proposition}\label{prop:f2_comparison_3to4}
Assume $\chi_i \ot \chi_j \in R$.  
Let $f_2(\chi_i \ot \chi_j) \colon G \to \F_2$ be defined as in Section \ref{subsec:construction_kappa}. 
Let $\rr_4^l \colon G \to \Tlt$ be a continuous group homomorphism 
satisfying $\rr_4^l(x_i) = E_{1,2}$ and $\rr_4^l(x_j) = E_{2,4}$, and $\rr_4^l(x_k) = I_4$ for $k\ne i,j$. 
Let  $\rr_4^r \colon G \to \Trt$ be a continuous group homomorphism 
satisfying $\rr_4^r(x_i) = E_{1,3}$ and $\rr_4^r(x_j) = E_{3,4}$, and $\rr_4^r(x_k) = I_4$ for $k\ne i,j$. 
Then 
\begin{align*}
e_{1,4} \circ \rr_4^l = f_2(\chi_i \ot \chi_j) = e_{1,4} \circ \rr_4^r 
\end{align*} 
as continuous maps $G \to \F_2$. 
\end{proposition}
\begin{proof}
The assertion for $\rr_4^l$ follows from Lemma \ref{lemma:matrix_subgroup_isos_3to4_left}, 
and the assertion for $\rr_4^r$ follows from Lemma \ref{lemma:matrix_subgroup_isos_3to4_right}. 
\end{proof}


We also need to consider the case of matrices in $U_5(\F_2)$. 

\begin{lemma}\label{lemma:matrix_subgroup_isos_3to5_left} 
Let $\TTlt$ be the subset of $U_5(\F_2)$ consisting of matrices of the form 
$\begin{psmallmatrix} 
1 & t & x & u &  z  \\
0 & 1 & 0 & v & 0   \\
0 &  0 & 1 & w & y \\
0 &  0 & 0 & 1 & 0 \\
0 & 0 & 0 & 0 & 1
\end{psmallmatrix}$  
and let $\WWlt$ denote the subset of matrices of the form 
$\begin{psmallmatrix} 
1 & t & 0 & u &  0 \\
0 & 1 & 0 & v & 0   \\
0 &  0 & 1 & w & 0 \\
0 &  0 & 0 & 1 & 0 \\
0 & 0 & 0 & 0 & 1
\end{psmallmatrix}$. 
Then $\TTlt$ 
is a subgroup of $U_5(\F_2)$, 
and $\WWlt$ is a normal subgroup of $\TTlt$. 
Moreover, 
the map $\taul_5 \colon U_3(\F_2) \to \TTlt/\WWlt$ 
defined by 
$\tau \colon \begin{psmallmatrix} 
1 & x &  z  \\
0 & 1 & y  \\
0 &  0 & 1 
\end{psmallmatrix}
\mapsto 
\begin{psmallmatrix} 
1 & * & x & * &  z  \\
0 & 1 & 0 & * & 0   \\
0 &  0 & 1 & * & y \\
0 &  0 & 0 & 1 & 0 \\
0 & 0 & 0 & 0 & 1
\end{psmallmatrix}$   
is an isomorphism of groups. 
\end{lemma}
\begin{proof}
This follows from a direct computation as in the proof of Lemma \ref{lemma:matrix_subgroup_isos_3to4_left}. 
\end{proof}


\begin{lemma}\label{lemma:matrix_subgroup_isos_3to5_right} 
Let $\TTrt$ be the subset of $U_5(\F_2)$ consisting of matrices of the form 
$\begin{psmallmatrix} 
1 & 0 & x & 0 &  z  \\
0 & 1 & w & v & u   \\
0 &  0 & 1 & 0 & y \\
0 &  0 & 0 & 1 & t \\
0 & 0 & 0 & 0 & 1
\end{psmallmatrix}$  
and let $\WWrt$ denote the subset of matrices of the form 
$\begin{psmallmatrix} 
1 & 0 & 0 & 0 &  0  \\
0 & 1 & w & v & u   \\
0 &  0 & 1 & 0 & 0 \\
0 &  0 & 0 & 1 & t \\
0 & 0 & 0 & 0 & 1
\end{psmallmatrix}$. 
Then $\TTrt$ 
is a subgroup of $U_5(\F_2)$, 
and $\WWrt$ is a normal subgroup of $\TTrt$. 
Moreover, 
the map $\taur_5 \colon U_3(\F_2) \to \TTrt/\WWrt$ 
defined by 
$\tau \colon \begin{psmallmatrix} 
1 & x &  z  \\
0 & 1 & y  \\
0 &  0 & 1 
\end{psmallmatrix}
\mapsto 
\begin{psmallmatrix} 
1 & 0 & x & 0 &  z  \\
0 & 1 & * & * & *   \\
0 &  0 & 1 & 0 & y \\
0 &  0 & 0 & 1 & * \\
0 & 0 & 0 & 0 & 1
\end{psmallmatrix}$   
is an isomorphism of groups. 
\end{lemma}
\begin{proof}
This follows from a direct computation as in the proof of Lemma \ref{lemma:matrix_subgroup_isos_3to4_left}. 
\end{proof}


\begin{proposition}\label{prop:f2_comparison_3to5}
Assume $\chi_i \ot \chi_j \in R$.  
Let $f_2(\chi_i \ot \chi_j) \colon G \to \F_2$ be defined as in Section \ref{subsec:construction_kappa}. 
Let $\rr_5^l \colon G \to \TTlt$ be a continuous group homomorphism 
satisfying $\rr_5^l(x_i) = E_{1,3}$ and $\rr_5^l(x_j) = E_{3,5}$, and $\rr_5^l(x_k) = I_5$ for $k\ne i,j$. 
Let  $\rr_5^r \colon G \to \TTrt$ be a continuous group homomorphism 
satisfying $\rr_5^r(x_i) = E_{1,3}$ and $\rr_5^r(x_j) = E_{3,5}$, and $\rr_5^r(x_k) = I_5$ for $k\ne i,j$. 
Then 
\begin{align*}
e_{1,5} \circ \rr_5^l = f_2(\chi_i \ot \chi_j) = e_{1,5} \circ \rr_5^r 
\end{align*} 
as continuous maps $G \to \F_2$. 
\end{proposition}
\begin{proof}
The assertion for $\rr_5^l$ follows from Lemma \ref{lemma:matrix_subgroup_isos_3to5_left}, 
and the assertion for $\rr_5^r$ follows from Lemma \ref{lemma:matrix_subgroup_isos_3to5_right}. 
\end{proof}


\subsection{Computation of the canonical class}\label{sec:computing_kappa3_torsionfree}

We will now compute the values of $\kappa_3$. 
To do so, we will make frequent use of the following notation. 


%
\begin{notn}\label{notation:Bs} 
For $i=1,2,3$ and $\varepsilon_i \in \F_2$, 
we define the matrix  
\begin{align*}
B_{\varepsilon_1, \varepsilon_2, \varepsilon_3} = 
\begin{psmallmatrix} 
1 & \varepsilon_1 & 0 & 0 \\
0 & 1 & \varepsilon_2 & 0 \\
0 &  0 & 1 & \varepsilon_3\\
0 & 0 & 0 & 1 
\end{psmallmatrix} 
\in U_4(\F_2). 
\end{align*}
Note that a direct computation shows $B_{\varepsilon_1, \varepsilon_2, \varepsilon_3}^q = I_4$ in $U_4(\F_2)$. 
\end{notn}


\begin{convention}\label{convention:embedding}
We embed $U_n(\F_2)$ as a subgroup in $U_{n+1}(\F_2)$ either by 
sending a matrix $A \in U_n(\F_2)$ to 
$\begin{psmallmatrix} 
1 & 0  \\
0 & A 
\end{psmallmatrix} 
\in U_{n+1}(\F_2)$,  
or to 
 $\begin{psmallmatrix} 
A & 0  \\
0 & 1
\end{psmallmatrix} 
\in U_{n+1}(\F_2)$,  
where the $0$s denote rows and columns with entry $0$ respectively. 
\end{convention}



\begin{proposition}\label{prop:kappa_3_vanishes_on_S_D_T}
Let $\kappa_3 \colon K_3^3(\Hb) \to H^2$ be the map defined in Proposition \ref{prop:canonical_class_of_Demushkin_group_construction}. 
The map $\kappa_3$ vanishes on all elements in the sets $\SSS$, $\DDD$ and $\TTT$.  
\end{proposition}
\begin{proof}
Let $\Psi_3$ denote the cocycle 
defined in \eqref{eq:def_of_Psi3_Demushkin}. 
We prove the assertion by constructing cochains $\vartheta \in \Ch^1$ such that $\delta \vartheta = \Psi_3$ for the elements in each of the sets 
using Dwyer's Theorem \ref{thm:Dwyer}. 
We consider each set separately. 
We begin with elements in $\SSS$. 
By definition of $K^3_3(\Hb)$, if $\chi_a \ot \chi_b \ot \chi_c \in H^1 \ot H^1 \ot H^1$ is an element in $\SSS$, 
then the triple Massey product $\langle \chi_a, \chi_b, \chi_c \rangle$ is defined. 
Moreover, $\{\chi_a, \chi_b, \chi_c,\ f_2(\chi_a \ot \chi_b), f_2(\chi_b\ot \chi_c)\}$ is a defining system, 
and $\kappa_3(\chi_a,\chi_b,\chi_c)$ is the class of this defining system.   
Thus, by Dwyer's Theorem \ref{thm:Dwyer}, 
$\kappa_3(\chi_a,\chi_b,\chi_c) = 0$ 
if and only if the continuous group homomorphism 
$\brr \colon G \to \bU_4(\F_2)$ given by 
$\brr = \begin{psmallmatrix}
 1 & \chi_a & f_2(\chi_a \ot \chi_b) &  * \\
  & 1& \chi_b & f_2(\chi_b \ot \chi_c)   \\
  & & 1 &  \chi_c \\
  & & & 1 
\end{psmallmatrix}$ 
lifts to a continuous group homomorphism 
$\rr \colon G \to U_4(\F_2)$. 
If a lift $\rr$ exists, then the continuous map $\vartheta \coloneqq e_{1,4} \circ \rr \colon G \to \F_2$ 
is a cochain such that $\delta \vartheta = \Psi_3 (\chi_a \ot \chi_b \ot \chi_c)$. 
Thus, to prove $\kappa_3(\chi_a,\chi_b,\chi_c) = 0$ for an element in $\SSS$, 
it suffices to construct the corresponding continuous homomorphism $\rr$. 

For elements of the form $\chi_i \ot \chi_i \ot \chi_i$, 
we define the continuous group homomorphism $\rr^{\SSS}_{i,i,i} \colon G \to U_4(\F_2)$ 
by setting $\rr^{\SSS}_{i,i,i}(x_i) \coloneqq B_{1,1,1}$ and $\rr^{\SSS}_{i,i,i}(x_j) \coloneqq I_4$ for $j \ne i$.  
%
%
For elements of the form $\chi_i \ot \chi_i \ot \chi_j$ with $i\ne j$ and $j \ne i-1$ if $i$ is even, 
we define the continuous group homomorphism $\rr^{\SSS}_{i,i,j} \colon G \to U_4(\F_2)$ 
by setting 
\begin{align*}
\rr^{\SSS}_{i,i,j}(x_i) \coloneqq B_{1,1,0}, ~ \rr^{\SSS}_{i,i,j}(x_j) \coloneqq B_{0,0,1}, ~ 
\text{and} ~\rr^{\SSS}_{i,i,j}(x_k) \coloneqq I_4 ~ \text{for} ~ k \ne i,j.
\end{align*}
%
%
For elements of the form $\chi_i \ot \chi_j \ot \chi_j$ with $i\ne j$ and $i \ne j-1$ if $j$ is even, 
we define the continuous group homomorphism $\rr^{\SSS}_{i,j,j} \colon G \to U_4(\F_2)$ 
by setting 
\begin{align*}
\rr^{\SSS}_{i,j,j}(x_i) \coloneqq B_{1,0,0}, ~ \rr^{\SSS}_{i,j,j}(x_j) \coloneqq B_{0,1,1}, 
~ \text{and} ~ \rr^{\SSS}_{i,j,j}(x_k) \coloneqq I_4 ~ \text{for} ~ k \ne i,j. 
\end{align*}
%
%
For elements of the form $\chi_i \ot \chi_j \ot \chi_k$ with $i,k\ne j$, $i \ne j-1$ if $j$ is even and $k \ne j+1$ if $j$ is odd, 
we define a continuous group homomorphism $\rr^{\SSS}_{i,j,k} \colon G \to U_4(\F_2)$ 
by setting 
\begin{align*}
\rr^{\SSS}_{i,j,k}(x_i) \coloneqq B_{1,0,0}, ~ 
\rr^{\SSS}_{i,j,k}(x_j) \coloneqq B_{0,1,0}, ~ 
\rr^{\SSS}_{i,j,k}(x_k) \coloneqq B_{0,0,1},  
\end{align*}
and $\rr^{\SSS}_{i,j,k}(x_l) \coloneqq I_4$ for $l \ne i,j,k$.  
This finishes the proof for the set $\SSS$. 


Next we consider elements in $\DDD$. 
First, let $i\in J/2$ and $k \ne 2i-1, 2i$, 
and we consider an element of the form $\chi_k \ot (\chi_{2i-1} \ot \chi_{2i} + \chi_{2i} \ot \chi_{2i-1})$. 
%
%
While $\chi_k \ot (\chi_{2i-1} \ot \chi_{2i} + \chi_{2i} \ot \chi_{2i-1})$ does not correspond to a triple Massey product, 
the fourfold Massey product $\langle \chi_k, \chi_{2i-1}, 0, \chi_{2i-1} \rangle$ is defined, 
and by Dwyer's Theorem \ref{thm:Dwyer}, 
the continuous homomorphism 
$\brr^{\DDD}_{k,2i-1,2i} \colon G \to \bU_5(\F_2)$ given by 
\begin{align*}
\brr^{\DDD}_{k,2i-1,2i} = 
\begin{psmallmatrix}
 1 &  \chi_k & f_2(\chi_k \ot \chi_{2i-1})  & f_2(\chi_k \ot \chi_{2i}) &  * \\
  & 1& \chi_{2i-1} & \chi_{2i} &  f_2(\chi_{2i-1} \ot \chi_{2i} + \chi_{2i} \ot \chi_{2i-1}) \\
  & & 1 & 0 &  \chi_{2i} \\
  & &  & 1 &  \chi_{2i-1} \\   
  & & & & 1 
\end{psmallmatrix}
\end{align*} 
corresponds to a defining system. 
Moreover, we have $\kappa_3(\chi_k \ot (\chi_{2i-1} \ot \chi_{2i} + \chi_{2i} \ot \chi_{2i-1}))=0$ 
if and only if $\brr^{\DDD}_{k,2i-1,2i}$ 
lifts to a continuous homomorphism 
$\rr^{\DDD}_{k,2i-1,2i} \colon G \to U_5(\F_2)$. 
To construct $\rr^{\DDD}_{k,2i-1,2i}$, 
we consider the matrices 
\begin{align*}
A \coloneqq 
\begin{psmallmatrix}
 1 & 0 & 0  & 0 &  0 \\
  & 1& 1 & 0& 0 \\
  & & 1 & 0 & 0 \\
  & &  & 1 &  1 \\   
  & & & & 1 
\end{psmallmatrix}, ~ 
B \coloneqq 
\begin{psmallmatrix}
 1 & 0 & 0 & 0 & 0 \\
  & 1& 0 & 1 & 0 \\
  & & 1 & 0 &  1 \\
  & &  & 1 &  0 \\   
  & & & & 1 
\end{psmallmatrix} 
\end{align*} 
in $U_5(\F_2)$. 
We have $[A,B] = I_5$. 
Thus, by Lemma \ref{lemma:matrix_relations}, 
we can define a continuous group homomorphism $\rr^{\DDD}_{k,2i-1,2i} \colon G \to U_5(\F_2)$ 
by the assignment 
\begin{align*}
x_k \mapsto E_{1,2}, ~ 
 x_{2i-1}  \mapsto A, ~ 
x_{2i} \mapsto B, 
\end{align*}
and $\rr^{\DDD}_{k,2i-1,2i}(x_j) = I_5$ for $j \ne k, 2i-1, 2i$. 
We use Proposition \ref{prop:f2_comparison_3to4} 
to identify $e_{1,4} \circ \rr^{\DDD}_{k,2i-1,2i}$ with $f_2(\chi_k \ot \chi_{2i})$. 
We define the continuous map $\vartheta^{\DDD}_{k,2i-1,2i} \colon G \to \F_2$ 
by $e_{1,5} \circ \rr^{\DDD}_{k,2i-1,2i}$ 
and get 
\begin{align*}
\delta \vartheta^{\DDD}_{k,2i-1,2i} = \Psi_3 (\chi_k \ot (\chi_{2i-1} \ot \chi_{2i} + \chi_{2i} \ot \chi_{2i-1})). 
\end{align*}
%

We prove the other cases by specifying the corresponding homomorphisms $\brr$ and 
constructing the required lifts $\rr$. 
For elements of the form $(\chi_{2i-1} \ot \chi_{2i} + \chi_{2i} \ot \chi_{2i-1}) \ot \chi_k$, 
we need to show that the continuous homomorphism 
$\brr^{\DDD}_{2i-1,2i,k}\colon G \to \bU_5(\F_2)$ given by 
\begin{align*}
\brr^{\DDD}_{2i-1,2i,k} = 
\begin{psmallmatrix}
 1 &  \chi_{2i-1} & \chi_{2i}  &  f_2(\chi_{2i-1} \ot \chi_{2i} + \chi_{2i} \ot \chi_{2i-1}) &  * \\
  & 1& 0 & \chi_{2i} &  f_2(\chi_{2i} \ot \chi_k)  \\
  & & 1 &  \chi_{2i-1} & f_2(\chi_{2i-1} \ot \chi_k) \\
  & &  & 1 &  \chi_k \\   
  & & & & 1 
\end{psmallmatrix}
\end{align*} 
lifts to a continuous homomorphism 
$\rr^{\DDD}_{2i-1,2i,k} \colon G \to U_5(\F_2)$. 
Consider the matrices 
\begin{align*}
C \coloneqq 
\begin{psmallmatrix}
 1 & 1 & 0 & 0 & 0 \\
  & 1& 0 & 0 & 0 \\
  & & 1 & 1 &  0 \\
  & &  & 1 &  0 \\   
  & & & & 1 
\end{psmallmatrix}, ~ 
D \coloneqq 
\begin{psmallmatrix}
 1 & 0 & 1  & 0 &  0 \\
  & 1& 0 & 1& 0 \\
  & & 1 & 0 & 0 \\
  & &  & 1 &  0 \\   
  & & & & 1 
\end{psmallmatrix} 
\end{align*} 
in $U_5(\F_2)$. 
We have $[C,D] = I_5$. 
By Lemma \ref{lemma:matrix_relations}, 
we can define a continuous group homomorphism $\rr^{\DDD}_{2i-1,2i,k} \colon G \to U_5(\F_2)$ 
by the assignment 
\begin{align*}
x_k \mapsto E_{4,5}, ~ 
 x_{2i-1}  \mapsto C, 
x_{2i}\mapsto D, 
\end{align*}
and $\rr^{\DDD}_{2i-1,2i,k}(x_j) = I_5$ for $j \ne k, 2i-1, 2i$. 
We use Proposition \ref{prop:f2_comparison_3to4} 
to identify $e_{2,5} \circ \rr^{\DDD}_{2i-1,2i,k}$ with $f_2(\chi_{2i} \ot \chi_k)$. 
%

Second, let $i \in \Jw/2$ and $k \ne 1,2,2i-1,2i$.  
For elements of the form $\chi_k \ot (\chi_{1} \ot \chi_{2} + \chi_{2i} \ot \chi_{2i-1})$, 
we need to show that the continuous homomorphism 
$\brr^{\DDD}_{k,1,2+2i,2i-1} \colon G \to \bU_5(\F_2)$ given by 
\begin{align*}
\brr^{\DDD}_{k,1,2+2i,2i-1} = 
\begin{psmallmatrix}
 1 &  \chi_k & f_2(\chi_k \ot \chi_1)  & f_2(\chi_k \ot \chi_{2i}) &  * \\
  & 1& \chi_1 & \chi_{2i} &  f_2(\chi_1 \ot \chi_2 + \chi_{2i} \ot \chi_{2i-1}) \\
  & & 1 & 0 &  \chi_2 \\
  & &  & 1 &  \chi_{2i-1} \\   
  & & & & 1 
\end{psmallmatrix}
\end{align*} 
lifts to a continuous homomorphism 
$\rr^{\DDD}_{k,1,2+2i,2i-1} \colon G \to U_5(\F_2)$. 
We have the relation $[E_{2,3}, E_{3,5}] [E_{4,5},E_{2,4}] = I_5$. 
(Note, however, that $[E_{2,3}, E_{3,5}]  = [E_{4,5},E_{2,4}]= E_{2,5} \ne I_5$.) 
By Lemma \ref{lemma:matrix_relations}, 
we can define a continuous group homomorphism $\rr^{\DDD}_{k,1,2+2i,2i-1} \colon G \to U_5(\F_2)$ 
by the assignment 
\begin{align*}
x_k \mapsto E_{1,2}, ~
x_1 \mapsto E_{2,3}, ~ 
x_2 \mapsto E_{3,5}, ~ 
 x_{2i-1}  \mapsto E_{4,5}, ~  
x_{2i} \mapsto E_{2,4}, 
\end{align*}
and $\rr^{\DDD}_{k,1,2+2i,2i-1}(x_j) = I_5$ for $j \ne k, 1,2,2i-1, 2i$. 
We use Proposition \ref{prop:f2_comparison_3to4} 
to identify $e_{1,4} \circ \rr^{\DDD}_{k,1,2+2i,2i-1}$ with $f_2(\chi_k \ot \chi_{2i})$. 
%
%
For elements of the form $(\chi_{1} \ot \chi_{2} + \chi_{2i} \ot \chi_{2i-1}) \ot \chi_k$, 
we need to show that the continuous homomorphism 
$\brr^{\DDD}_{1,2+2i,2i-1,k}\colon G \to \bU_5(\F_2)$ given by 
\begin{align*}
\brr^{\DDD}_{1,2+2i,2i-1,k} = 
\begin{psmallmatrix}
 1 &  \chi_1 & \chi_{2i} &  f_2(\chi_1 \ot \chi_2 + \chi_{2i} \ot \chi_{2i-1}) &  * \\
  & 1& 0 & \chi_{2} & f_2(\chi_{2} \ot \chi_k) \\
  & & 1 & \chi_{2i-1} & f_2(\chi_{2i-1} \ot \chi_k)  \\
  & &  & 1 &  \chi_{k} \\   
  & & & & 1 
\end{psmallmatrix}
\end{align*} 
lifts to a continuous homomorphism 
$\rr^{\DDD}_{1,2+2i,2i-1,k} \colon G \to U_5(\F_2)$. 
We have the relation $[E_{1,2}, E_{2,4}] [E_{3,4},E_{1,3}] = I_5$. 
(Note, however, that $[E_{1,2}, E_{2,4}]  = [E_{3,4},E_{1,3}] = E_{1,4} \ne I_5$.) 
By Lemma \ref{lemma:matrix_relations}, 
we can define a continuous group homomorphism $\rr^{\DDD}_{1,2+2i,2i-1,k} \colon G \to U_5(\F_2)$ 
by the assignment 
\begin{align*}
x_k  \mapsto E_{4,5}, ~
x_1 \mapsto E_{1,2}, ~ 
x_2 \mapsto E_{2,4}, ~
 x_{2i-1} \mapsto E_{3,4}, ~ 
x_{2i}\mapsto E_{1,3}, 
\end{align*}
and $\rr^{\DDD}_{1,2+2i,2i-1,k}(x_j) = I_5$ for $j \ne k, 1,2,2i-1, 2i$. 
We use Proposition \ref{prop:f2_comparison_3to4} 
to identify $e_{2,5} \circ \rr^{\DDD}_{1,2+2i,2i-1,k}$ with $f_2(\chi_{2} \ot \chi_k)$. 
This proves the assertion for  $\DDD$. 


Finally, we compute $\kappa_3$ on elements in $\TTT$. 
For $i\in J/2$, we consider $\chi_{2i-1,2i-1,2i} +  \chi_{2i-1,2i,2i-1} + \chi_{2i,2i-1,2i-1}$. 
By Dwyer's Theorem \ref{thm:Dwyer}, 
the continuous group homomorphism 
$\brr^{\TTT}_{2i-1} \colon G \to \bU_6(\F_2)$ given by 
\begin{align*}
\brr^{\TTT}_{2i-1}  = \begin{psmallmatrix}
 1 &  \chi_{2i-1} & \chi_{2i}  & f_2(\chi_{2i-1} \ot \chi_{2i} + \chi_{2i} \ot \chi_{2i-1}) & f_2(\chi_{2i-1} \ot \chi_{2i-1}) & * \\
  & 1& 0 & \chi_{2i} &  \chi_{2i-1} & f_2(\chi_{2i-1} \ot \chi_{2i} + \chi_{2i} \ot \chi_{2i-1})  \\
  & & 1 & \chi_{2i-1} & 0 & f_2(\chi_{2i-1} \ot \chi_{2i-1})  \\
  & & & 1 & 0 & \chi_{2i-1} \\  
    & & &  & 1 &  \chi_{2i} \\   
  & & & & & 1 
\end{psmallmatrix}
\end{align*} 
corresponds to a defining system 
for the fivefold Massey product $\langle \chi_{2i-1}, 0, 0,0, \chi_{2i-1} \rangle$. 
Consider the matrices 
\begin{align*}
S \coloneqq 
\begin{psmallmatrix}
 1 &  1 & 0 & 0& 0 & 0 \\
  & 1& 0 & 0 & 1 & 0  \\
  & & 1 & 1 & 0 & 0 \\
  & & & 1 & 0 & 1 \\  
    & & &  & 1 & 0 \\   
  & & & & & 1 
\end{psmallmatrix}, ~ 
T \coloneqq 
\begin{psmallmatrix}
 1 &  0 & 1 & 0& 0 & 0 \\
  & 1& 0 & 1 & 0 & 0  \\
  & & 1 & 0 & 0 & 0 \\
  & & & 1 & 0 & 0 \\  
    & & &  & 1 & 1 \\   
  & & & & & 1 
\end{psmallmatrix}, 
\end{align*}
in $U_6(\F_2)$. 
We have $S^q=I_6$, $T^q= I_6$, and $[S,T] = I_6$ in $U_6(\F_2)$. 
Thus, by Lemma \ref{lemma:matrix_relations}, 
we can define continuous group homomorphisms $\rr^{\TTT}_{2i-1} \colon G \to U_6(\F_2)$ 
by the assignment  
%
$x_{2i-1} \mapsto S$,  
$x_{2i} \mapsto T$, 
and $\rr^{\TTT}_{2i-1}(x_j) = I_6$ for $j \ne 2i-1, 2i$. 
Then $\rr^{\TTT}_{2i-1}$ lifts $\brr^{\TTT}_{2i-1}$. 
This implies 
$\kappa_3 (\chi_{2i-1,2i-1,2i} +  \chi_{2i-1,2i,2i-1} + \chi_{2i,2i-1,2i-1}) = 0$. 
Here, we use Convention \ref{convention:embedding} and 
Proposition \ref{prop:f2_comparison_3to4} 
to identify $e_{3,6} \circ \rr^{\TTT}_{2i-1}$ 
and 
Proposition \ref{prop:f2_comparison_3to5} 
to identify $e_{1,5} \circ \rr^{\TTT}_{2i-1}$ 
with $f_2(\chi_{2i-1} \ot \chi_{2i-1})$. 
By switching the roles of $\chi_{2i-1}$ and $\chi_{2i}$, 
we define $\rr^{\TTT}_{2i} \colon G \to U_6(\F_2)$ 
by the assignment 
$x_{2i-1} \mapsto T$,  
$x_{2i} \mapsto S$, 
and $\rr^{\TTT}_{2i}(x_j) = I_6$ for $j \ne 2i-1, 2i$. 
The existence of $\rr^{\TTT}_{2i}$ implies 
$\kappa_3 (\chi_{2i,2i,2i-1} +  \chi_{2i,2i-1,2i} + \chi_{2i-1,2i,2i}) = 0$. 
%
%
This proves the assertion for the set $\TTT$ 
and  finishes the proof. 
\end{proof}


For $d \ge 4$, however, $\kappa_3$ is a non-trivial map 
as the proof of the following proposition shows. 
 
\begin{proposition}\label{prop:kappa_3_on_DDDw}
For all $i,j \in \Jw/2$, we have 
\begin{align*}
\kappa_3(\chi_{1,1,2} + \chi_{1,2i,2i-1}) & = \kappa_3(\chi_{1,1,2} + \chi_{1,2j,2j-1}), \\
\kappa_3(\chi_{1,2,2} + \chi_{2i,2i-1,2}) & = \kappa_3(\chi_{1,2,2} + \chi_{2j,2j-1,2}), \\
\kappa_3(\chi_{2,1,1} + \chi_{2i-1,2i,1}) & = \kappa_3(\chi_{2,1,1} + \chi_{2j-1,2j,1}), ~ \text{and} \\
\kappa_3(\chi_{2,2,1} + \chi_{2,2i-1,2i}) & = \kappa_3(\chi_{2,2,1} + \chi_{2,2j-1,2j}). 
\end{align*} 
\end{proposition}
\begin{proof}
The claimed relations follow from the symmetry of the roles of the $\chi_i$. 
However, we also give a more concrete proof by computing the respective values of $\kappa_3$ 
using Dwyer's Theorem \ref{thm:Dwyer} similar to the arguments in the proof of Proposition \ref{prop:kappa_3_vanishes_on_S_D_T}.  
For the first relation, 
we show that the continuous homomorphism 
$\brr_1 \colon G \to \bU_5(\F_2)$ given by 
\begin{align*}
\brr_1 = 
\begin{psmallmatrix}
 1 & \chi_1 & f_2(\chi_1 \ot \chi_{1})  & f_2(\chi_1 \ot \chi_{2i}) &  * \\
  & 1& \chi_{1} & \chi_{2i} &  f_2(\chi_1 \ot \chi_2 + \chi_{2i} \ot \chi_{2i-1}) \\
  & & 1 & 0 &  \chi_2 \\
  & &  & 1 &  \chi_{2i-1} \\   
  & & & & 1 
\end{psmallmatrix}
\end{align*} 
does not lift to a continuous homomorphism 
$G \to U_5(\F_2)$. 
Let 
$A \coloneqq 
\begin{psmallmatrix}
 1 &  1 & 0 & 0 &  0 \\
  & 1& 1 & 0& 0  \\
  & & 1 & 0 & 0\\
  & & & 1 &  0 \\  
    & & &  & 1 
\end{psmallmatrix}$  
in $U_5(\F_2)$. 
Then we have $A^q=I_5$ and $[A,E_{3,5}][E_{4,5}, E_{2,4}] = I_5$ in $\bU_5(\F_2)$, 
but $[A,E_{3,5}][E_{4,5}, E_{2,4}]  = E_{1,5} \ne I_5$ in $U_5(\F_2)$. 
Thus, we can define a continuous group homomorphism $\brr_1 \colon G \to \bU_5(\F_2)$ 
by the assignment 
\begin{align*}
x_1 \mapsto A, ~ 
x_2 \mapsto E_{3,5}, ~
 x_{2i-1} \mapsto E_{4,5}, ~ 
x_{2i}\mapsto E_{2,4}, 
\end{align*}
and $\brr_1(x_j) \coloneqq I_5$ for $j \ne 1,2,2i-1, 2i$,  
but $\brr_1$ does not lift to a 
continuous group homomorphism $G \to U_5(\F_2)$. 
Note that we use Proposition \ref{prop:f2_comparison_3to4} 
to identify $e_{1,4} \circ \brr_1$ with $f_2(\chi_1 \ot \chi_{2i})$. 
This proves $\kappa_3(\chi_{1,1,2} + \chi_{1,2i,2i-1}) \ne 0$ for all $i \in \Jw/2$. 
This implies the first relation 
since there is a unique non-zero element in $H^2 \cong \F_2$. 
%

For the second relation, 
we show that the continuous group homomorphism 
$\brr_2 \colon G \to \bU_5(\F_2)$ given by 
\begin{align*}
\brr_2 = 
\begin{psmallmatrix}
 1 & \chi_1 & \chi_{2i}  & f_2(\chi_1 \ot \chi_2 + \chi_{2i} \ot \chi_{2i-1}) &  * \\
  & 1& 0 & \chi_{2} &  f_2(\chi_2 \ot \chi_2) \\
  & & 1 & \chi_{2i-1} & f_2(\chi_{2i-1} \ot \chi_2) \\
  & &  & 1 &  \chi_2 \\   
  & & & & 1 
\end{psmallmatrix}
\end{align*} 
lifts to a continuous homomorphism 
$\rr_2 \colon G \to U_5(\F_2)$. 
Let 
$B \coloneqq 
\begin{psmallmatrix}
 1 &  0 & 0 & 0 & 0 \\
  & 1& 0 & 1 & 0  \\
  & & 1 & 0 & 0\\
  & & & 1 & 1 \\  
    & & &  & 1 
\end{psmallmatrix}$  
in $U_5(\F_2)$. 
Then we have $[E_{1,2},B][E_{3,4}, E_{1,3}] = I_5$ in $U_5(\F_2)$. 
Note, however, that $[E_{1,2},B] = [E_{3,4}, E_{1,3}] = E_{1,4} \ne I_5$. 
%
Thus, we can define a continuous group homomorphism $\rr_2 \colon G \to \bU_5(\F_2)$ 
by the assignment 
\begin{align*}
x_1 \mapsto E_{1,2}, ~ 
x_2 \mapsto B, ~
 x_{2i-1} \mapsto E_{3,4}, ~ 
x_{2i}\mapsto E_{1,3}, 
\end{align*}
and $\rr_2(x_j) \coloneqq I_5$ for $j \ne 1,2,2i-1, 2i$,  
which lifts $\brr_2$.  
Here we use Proposition \ref{prop:f2_comparison_3to4} 
to identify $e_{2,5} \circ \brr_2$ and $e_{2,5} \circ \rr_2$ with $f_2(\chi_2 \ot \chi_{2})$. 
This proves $\kappa_3(\chi_{1,2,2} + \chi_{2i,2i-1,2}) = 0$ for all $i \in \Jw/2$. 


For the third relation, we note that in $H^1 \ot H^1 \ot H^1$ we have 
\begin{align*}
& (\chi_{2,1} + \chi_{2i-1,2i}) \ot \chi_1 \\
= & ~~ [(\chi_{1,2} + \chi_{2,1}) + (\chi_{2i-1,2i} + \chi_{2i,2i-1}) + (\chi_{1,2} + \chi_{2i,2i-1})] \ot \chi_1. 
\end{align*}
Since $\chi_{1,2,1} + \chi_{2,1,1}$ and $\chi_{1,2,1} + \chi_{2i,2i-1,1}$
 are not in $K_3^3(\Hb)$, 
we add $\chi_{1,1,2}$ twice to the sum and get 
\begin{align*}
 (\chi_{2,1} + \chi_{2i-1,2i}) \ot \chi_1 & = (\chi_{1,1,2} + \chi_{1,2,1} + \chi_{2,1,1}) \\
&~~   + ((\chi_{2i-1,2i} + \chi_{2i,2i-1}) \ot \chi_1) + (\chi_{1,1,2} + \chi_{1,2,1} + \chi_{2i,2i-1,1}). 
\end{align*}
Here we note that the latter is not a linear combination of elements in $\BBB_3$, 
and that $\kappa_3$ is defined independently of our choice of basis for $K_3^3(\Hb)$. 
%
We know $\kappa_3(\chi_{1,1,2} + \chi_{1,2,1} + \chi_{2,1,1}) = 0$ by the computation for elements in $\TTT$, 
and $\kappa_3((\chi_{2i-1,2i} + \chi_{2i,2i-1}) \ot \chi_1) = 0$ by the computation for elements in $\DDD$. 
Thus, $\kappa_3(\chi_{2,1,1} + \chi_{2i-1,2i,1}) = \kappa_3(\chi_{1,1,2} + \chi_{1,2,1} + \chi_{2i,2i-1,1})$,  
and it suffices to compute the latter. 
%
%
To compute $\kappa_3(\chi_{1,1,2} + \chi_{1,2,1} + \chi_{2i,2i-1,1})$, 
we show that the continuous homomorphism 
$\brr_3 \colon G \to \bU_6(\F_2)$ given by 
\begin{align*}
\brr_3 = 
\begin{psmallmatrix}
 1 & \chi_{2i} & \chi_1 &  f_2(\chi_1 \ot \chi_2 + \chi_{2i} \ot \chi_{2i-1}) & f_2(\chi_1 \ot \chi_{1}) &  * \\
  & 1 & 0 & \chi_{2i-1} & 0 &  f_2(\chi_{2i-1} \ot \chi_1) \\
  & & 1 & \chi_2 &  \chi_1 &  f_2(\chi_1 \ot \chi_2 + \chi_{2} \ot \chi_{1}) \\
  & &  & 1 & 0 & \chi_{1} \\  
  & &  &  & 1 & \chi_{2} \\  
  & & & & & 1 
\end{psmallmatrix}
\end{align*} 
lifts to a continuous homomorphism 
$\rr_3 \colon G \to U_6(\F_2)$. 
Let 
$M \coloneqq 
\begin{psmallmatrix}
 1 & 0 & 1 & 0 &  0 & 0\\
  & 1 & 0 & 0 & 0 & 0 \\
  & & 1 & 1 & 1 & 0 \\
  & & & 1 &  0 & 1 \\ 
  &  & & & 1 &  0 \\  
   & & & &  & 1  
\end{psmallmatrix}$  
$N \coloneqq 
\begin{psmallmatrix}
 1 & 0 & 0 & 0 &  0 & 0\\
  & 1 & 0 & 0 & 0 & 0 \\
  & & 1 & 1 & 0 & 0 \\
  & & & 1 &  0 & 1 \\ 
  &  & & & 1 &  0 \\  
   & & & &  & 1  
\end{psmallmatrix}$  
in $U_6(\F_2)$. 
%
Then $M^q= I_6$ and $[M,N] [E_{2,4}, E_{1,2}] = I_6$ in $U_6(\F_2)$. 
(Note that $[M,N]  = [E_{2,4}, E_{1,2}] = E_{1,4}$ in $U_6(\F_2)$.) 
Thus, we can define a continuous group homomorphism $\rr_3 \colon G \to U_6(\F_2)$ 
by the assignment 
\begin{align*}
x_1 \mapsto M, ~ 
x_2 \mapsto N, ~
 x_{2i-1} \mapsto E_{2,4}, ~ 
x_{2i}\mapsto E_{1,2}, 
\end{align*}
and $\rr_3(x_j) \coloneqq I_6$ for $j \ne 1,2,2i-1, 2i$, 
which lifts $\brr_3$. 
Here we use Proposition \ref{prop:f2_comparison_3to5} 
to identify $e_{1,4} \circ \brr_3$ and $e_{1,4} \circ \rr_3$ with $f_2(\chi_1 \ot \chi_{1})$, 
and $e_{2,5} \circ \brr_3$ and $e_{2,5} \circ \rr_3$ with $f_2(\chi_{2i-1} \ot \chi_{1})$. 
This proves 
\[
\kappa_3(\chi_{2,1,1} + \chi_{2i-1,2i,1})  = \kappa_3(\chi_{1,1,2} + \chi_{1,2,1} + \chi_{2i,2i-1,1}) =  0 ~ \text{for all} ~ i \in \Jw/2,
\]
and thus the third relation. 
%


For the fourth relation, we note that in $H^1 \ot H^1 \ot H^1$ we have 
\begin{align*}
& \chi_2 \ot (\chi_{2,1} + \chi_{2i-1,2i}) \\
= & ~~ \chi_2 \ot [(\chi_{1,2} + \chi_{2,1}) + (\chi_{2i-1,2i} + \chi_{2i,2i-1}) + (\chi_{1,2} + \chi_{2i,2i-1})]. 
\end{align*}
Since $\chi_{2,1,2} + \chi_{2,2,1}$ and $\chi_{2,1,2} + \chi_{2,2i,2i-1}$ 
are not in $K_3^3(\Hb)$, 
we add $\chi_{1,2,2}$ twice to the sum and get 
\begin{align*}
& ~~ \chi_2 \ot (\chi_{2,1} + \chi_{2i-1,2i}) \\
= &  (\chi_{1,2,2} + \chi_{2,1,2} + \chi_{2,2,1})  + (\chi_2 \ot (\chi_{2i-1,2i} + \chi_{2i,2i-1})) + (\chi_{1,2,2} + \chi_{2,1,2} + \chi_{2,2i,2i-1}). 
\end{align*}
We know $\kappa_3(\chi_{1,2,2} + \chi_{2,1,2} + \chi_{2,2,1}) = 0$ by the computation for elements in $\TTT$, 
and $\kappa_3(\chi_2 \ot (\chi_{2i-1,2i} + \chi_{2i,2i-1})) = 0$ by the computation for elements in $\DDD$. 
Thus, $\kappa_3(\chi_{2,2,1} + \chi_{2,2i-1,2i}) = \kappa_3(\chi_{1,2,2} + \chi_{2,1,2} + \chi_{2,2i,2i-1})$,  
and it suffices to compute the latter. 
To compute $\kappa_3(\chi_{1,2,2} + \chi_{2,1,2} + \chi_{2,2i,2i-1})$, 
we show that the continuous homomorphism 
$\brr_4 \colon G \to \bU_6(\F_2)$ given by 
\begin{align*}
\brr_4 = 
\begin{psmallmatrix}
 1 & \chi_1 & \chi_2 &  f_2(\chi_1 \ot \chi_2 + \chi_{2} \ot \chi_{1}) & f_2(\chi_2 \ot \chi_{2i}) &  * \\
  & 1 & 0 & \chi_2 & 0 &  f_2(\chi_2 \ot \chi_2) \\
  & & 1 & \chi_1 &  \chi_{2i} &  f_2(\chi_1 \ot \chi_2 + \chi_{2i} \ot \chi_{2i-1}) \\
  & &  & 1 & 0 & \chi_{2} \\  
  & &  &  & 1 & \chi_{2i-1} \\  
  & & & & & 1 
\end{psmallmatrix}
\end{align*} 
does not lift to a continuous homomorphism 
$G \to U_6(\F_2)$. 
Let 
$P \coloneqq 
\begin{psmallmatrix}
 1 & 1 & 0 & 0 &  0 & 0\\
  & 1 & 0 & 0 & 0 & 0 \\
  & & 1 & 1 & 0 & 0 \\
  & & & 1 &  0 & 0 \\ 
  &  & & & 1 &  0 \\  
   & & & &  & 1  
\end{psmallmatrix}$  
$Q \coloneqq 
\begin{psmallmatrix}
 1 & 0 & 1 & 0 &  0 & 0\\
  & 1 & 0 & 1 & 0 & 0 \\
  & & 1 & 0 & 0 & 0 \\
  & & & 1 &  0 & 1 \\ 
  &  & & & 1 &  0 \\  
   & & & &  & 1  
\end{psmallmatrix}$  
in $U_6(\F_2)$. 
%
Then $P^q = I_6$ and $[P, Q] [E_{5,6}, E_{3,5}] = I_6$ in $\bU_6(\F_2)$, 
but $[P, Q] [E_{5,6}, E_{3,5}] = E_{1,6} \ne I_6$ in $U_6(\F_2)$. 
Thus, we can define a continuous group homomorphism $\brr_4 \colon G \to \bU_6(\F_2)$ 
by the assignment 
\begin{align*}
x_1 \mapsto P, ~ 
x_2 \mapsto Q, ~
 x_{2i-1} \mapsto E_{5,6}, ~ 
x_{2i}\mapsto E_{3,5}, 
\end{align*}
and $\brr_4(x_j) \coloneqq I_6$ for $j \ne 1,2,2i-1, 2i$,  
but $\brr_4$ does not lift to a 
continuous group homomorphism $G \to U_6(\F_2)$. 
Here we use Proposition \ref{prop:f2_comparison_3to5} 
to identify $e_{1,4} \circ \brr_4$ and $e_{1,4} \circ \rr_4$ with $f_2(\chi_2 \ot \chi_{2i})$, 
and $e_{2,5} \circ \brr_4$ and $e_{2,5} \circ \rr_4$ with $f_2(\chi_{2} \ot \chi_{2})$. 
This proves 
\[
\kappa_3(\chi_{2,2,1} + \chi_{2,2i-1,2i}) = \kappa_3(\chi_{1,2,2} + \chi_{2,1,2} + \chi_{2,2i,2i-1}) \ne  0 ~ \text{for all} ~ i \in \Jw/2,
\]
and thus the fourth relation, 
%
since there is a unique non-zero element in $H^2$. 
\end{proof}

\begin{proof}[Proof of Theorem \ref{thm:torsion_free_pro-2_Demushkin_groups_are_A_3-formal}] 
The theorem now follows from 
Corollary \ref{cor:image_of_partial_sigmas_DDDw},
and  
Propositions 
\ref{prop:kappa_3_vanishes_on_S_D_T},  
and \ref{prop:kappa_3_on_DDDw}. 
\end{proof}

%

\section{Demushkin groups of type II}\label{sec:type_II}

The goal of this section is to prove the following theorem. 

\begin{theorem}\label{thm:type_II_pro-2_Demushkin_groups_are_A_3-formal}
Let $d \ge 3$ be an odd number, and let $G$ be the pro-$2$ group generated by elements $x_1, \ldots, x_d$ with the single relation 
$x_1^2 x_2^{2^f} [x_2,x_3]  \cdots [x_{d-1},x_d] = 1$ with $f\ge 2$ or $f= \infty$ 
with the convention $2^{\infty} \coloneqq 0$.  
Then the canonical class of $G$ is trivial, and $G$ is $A_3$-formal. 
\end{theorem}

In this section, 
we let $G$ be a pro-$2$ group as in Theorem \ref{thm:type_II_pro-2_Demushkin_groups_are_A_3-formal}. 
We use the same notation for $\Cb$, $\Zh^{\bbb}$, $\Hb$, and $\chi_i$ as in Section \ref{sec:torsion-free_p=2_Demushkin}. 
We write again $J \coloneqq \{1, \ldots, d\}$ and now set $J/2 \coloneqq \{1, \ldots, (d-1)/2 \}$. 
%


\subsection{The Koszul--Hochschild complex -- Demushkin groups of type II}\label{subsec:Koszul_complex_p=2_type_II}

Again, by \cite[Theorem 5.2]{MPQT}, 
the graded $\F_2$-cohomology algebra of $G$ is a Koszul algebra. 
We now provide a basis for $R$ and $K_3^3(\Hb)$. 


\begin{proposition}\label{prop:type_II_basis_2_for_Demushkin_odd_generators}
The $\F_2$-vector subspace $R \subset H^1 \ot H^1$ such that $\Hb = T(H^1)/(R)$ 
admits the set $\BII_2$ as a basis where
\begin{align*}
\BII_2 \coloneqq \{ & \chi_i \ot \chi_j ~ \text{for all} ~ (i,j) \notin \{(1,1), (2r,2r+1), (2r+1,2r) : r  \in J/2 \}; \\ 
& \chi_1 \ot \chi_1 + \chi_{2i+1} \ot \chi_{2i}, \chi_{2i} \ot \chi_{2i+1} + \chi_{2i+1} \ot \chi_{2i} ~ \text{for all} ~ i \in J/2 \}.  
\end{align*}
\end{proposition}
\begin{proof}
This follows from the description of the cup product 
on $\Hb(G,\F_2)$ in Lemma \ref{lemma:relations_in_H2_all_groups}.  
\end{proof}


\begin{proposition}\label{prop:type_II_basis_for_Demushkin_odd_generators}
The $\F_2$-vector space $K^3_3(\Hb)$ admits  the set 
$\BII_3 = \SII \cup \DII  \cup \TII \cup \TIIw  \cup \EII$ 
as a basis where 
\begin{align*}
\SII & = \left\{ \chi_i \ot \chi_j \ot \chi_k 
~ \text{with}~ \begin{cases}
i, k \ne 1 &  \text{if} ~ j =1 \\
i, k \ne 2r+1 &  \text{if} ~ j = 2r \\
i, k \ne 2r &  \text{if} ~ j = 2r+1
\end{cases} 
\right\}, 
\end{align*}
\begin{align*}
\DII = & ~ \{ \chi_j \ot (\chi_{2i} \ot \chi_{2i+1} + \chi_{2i+1} \ot \chi_{2i}) , \\ 
& ~ ~ ~ (\chi_{2i} \ot \chi_{2i+1} + \chi_{2i+1} \ot \chi_{2i}) \ot \chi_j ~ \text{for} ~ i \in J/2, ~ j \ne 2i, 2i+1,   \\
& ~ ~ ~  \chi_j \ot (\chi_{1} \ot \chi_{1} + \chi_{2i+1} \ot \chi_{2i}), ~ \text{for} ~i \in J/2, ~ j \ne 1, 2i\\
& ~ ~ ~ (\chi_{1} \ot \chi_{1} + \chi_{2i+1} \ot \chi_{2i}) \ot \chi_j ~ \text{for} ~ i \in J/2, ~ j \ne 1, 2i+1
\}, 
\end{align*}
\begin{align*}
\TII =  \{ & \chi_{2i} \ot \chi_{2i+1} \ot \chi_{2i+1} + \chi_{2i+1} \ot \chi_{2i} \ot \chi_{2i+1} + \chi_{2i+1} \ot \chi_{2i+1} \ot \chi_{2i}, \\
& ~ \chi_{2i+1} \ot \chi_{2i} \ot \chi_{2i} + \chi_{2i} \ot \chi_{2i+1} \ot \chi_{2i} + \chi_{2i} \ot \chi_{2i} \ot \chi_{2i+1} ~ \text{for} ~ i \in J/2 
\}, 
\end{align*}
\begin{align*}
\TIIw =  \{ &  \chi_{1} \ot \chi_{1} \ot \chi_{1} + \chi_1 \ot \chi_{2i+1} \ot \chi_{2i} + \chi_{2i+1} \ot \chi_{2i} \ot \chi_{1}, \\
& \chi_{2i} \ot \chi_{1} \ot \chi_{1} + \chi_{2i} \ot \chi_{2i+1} \ot \chi_{2i} + \chi_{2i+1} \ot \chi_{2i} \ot \chi_{2i}, \\
& \chi_{1} \ot \chi_{1} \ot \chi_{2i+1} + \chi_{2i+1} \ot \chi_{2i} \ot \chi_{2i+1} + \chi_{2i+1} \ot \chi_{2i+1} \ot \chi_{2i}, ~  \text{for} ~ i\in J/2 
\},
\end{align*}
\begin{align*}
\EII = & ~ \{ \chi_{1} \ot \chi_{1} \ot \chi_{1} + \chi_{1} \ot \chi_{2i+1} \ot \chi_{2i} + \chi_{2i+3} \ot \chi_{2i+2} \ot \chi_1 ~ \text{for} ~ 1\le i \le (d-3)/2  \}.
\end{align*}
\end{proposition}
\begin{proof}
By \cite[Lemma 5.10]{PQ3}, we know $\dim_{\F_2} K^3_3(\Hb) = d^3-2d$. 
Hence it suffices to count the number of elements in $\BII_3$ and to show that the set is linearly independent. 
We have $\# \SII = (d-1) d (d-1)$ 
since we have $d$ choices for the middle index $j$ and for each choice exactly one value for $i$ and $k$ excluded. 
We have $\# \DII = 4(d-2) (d-1)/2 = 2(d-2)(d-1)$  
since we have four times $(d-1)/2$ choices for the index $i$ and for each choice exactly two values for $j$ excluded. 
%
%
We have $\# \EII = (d-3)/2$ 
since we have $(d-3)/2$ choices for the index $i$ each furnishing one element. 
We have $\# \TII  = d-1$ 
as there are $(d-1)/2$ choices for the index $i$ each furnishing two triple sums, 
and $\# \TIIw = 3(d-1)/2$ since each choice of $i$ furnishes three elements. 
%
%
Thus, in total we have 
\begin{align*}
\# \BII_3  & = \# \SII + \# \DII   + \# \TII + \#\TIIw + \#\EII \\
& = d(d-1)^2 + 2(d-2)(d-1) + (d-1) + 3(d-1)/2 + (d-3)/2 \\
& =  d^3 - 2d.
\end{align*}
It remains to show that $\BII_3$ is linearly independent. 
Assume we write the zero vector in $K_3^3(\Hb)$ as an $\F_2$-linear combination of elements in $\BII_3$. 
Since the terms in $\SII$ only occur in $\SII$, their coefficients must be zero. 
Similarly, for $i \in J/2$, 
terms of the form $\chi_{2i} \ot \chi_{1} \ot \chi_{1}$ 
and 
$\chi_{1} \ot \chi_{1} \ot \chi_{2i+1}$ 
each only occur once in sums in $\TIIw$. 
Hence the coefficients of 
$\chi_{2i} \ot \chi_{1} \ot \chi_{1} + \chi_{2i} \ot \chi_{2i+1} \ot \chi_{2i} + \chi_{2i+1} \ot \chi_{2i} \ot \chi_{2i}$ 
and $\chi_{1} \ot \chi_{1} \ot \chi_{2i+1} + \chi_{2i+1} \ot \chi_{2i} \ot \chi_{2i+1} + \chi_{2i+1} \ot \chi_{2i+1} \ot \chi_{2i}$ 
in $\TIIw$ must be zero. 
Terms of the form $\chi_{2i} \ot \chi_{2i+1} \ot \chi_{2i}$ 
and $\chi_{2i+1} \ot \chi_{2i} \ot \chi_{2i+1}$ 
only occur in sums in $\TII$ and $\TIIw$. 
We have shown that the coefficients of the corresponding sums in $\TIIw$ are zero. 
Hence the coefficients of all the terms in $\TII$ must be zero. 
For $i \in J/2$ and $j \ne 2i,2i+1$, 
terms of the form 
$\chi_{2i} \ot \chi_{2i+1} \ot \chi_{j}$ only occur once in sums in $\DII$. 
Hence the coefficient of 
$(\chi_{2i} \ot \chi_{2i+1} + \chi_{2i+1} \ot \chi_{2i}) \ot \chi_j$ 
must be zero.  
For $i \in J/2$ and $j \ne 2i,2i+1$, 
the only sum which contains the term 
$\chi_j \ot \chi_{2i} \ot \chi_{2i+1}$ 
is 
$\chi_j \ot (\chi_{2i} \ot \chi_{2i+1} + \chi_{2i+1} \ot \chi_{2i})$ 
in $\DII$. 
This implies that the coefficient of the latter must be zero. 
Similarly, for $i \in J/2$ and $j \ne 1,2i+1$, 
the only remaining sum which contains the term 
$\chi_{2i+1} \ot \chi_{2i} \ot \chi_j$ is 
$(\chi_{1} \ot \chi_{1} + \chi_{2i+1} \ot \chi_{2i}) \ot \chi_j$ in $\DII$. 
This implies that coefficient of the latter must be zero as well. 
For $i \in J/2$ and $j \ne 1,2i$, 
terms of the form 
$\chi_j \ot \chi_{2i+1} \ot \chi_{2i}$ only occur in the sum 
$\chi_j \ot (\chi_{1} \ot \chi_{1} + \chi_{2i+1} \ot \chi_{2i})$ in $\DII$. 
This implies that coefficient of the latter must be zero, too. 
This shows that the coefficients of all terms in $\DII$ must be zero. 
In particular, the coefficient of $(\chi_{2i} \ot \chi_{2i+1} + \chi_{2i+1} \ot \chi_{2i}) \ot \chi_1$ 
is zero. 
Thus, 
the term $\chi_{2i+1} \ot \chi_{2i} \ot \chi_1$ now only occurs in sums in $\TIIw$ and in $\EII$.  
Now we use induction to show that the remaining coefficients are zero as well. 
For $i=1$, 
the term $\chi_3 \ot \chi_2 \ot \chi_1$ only occurs in 
$\chi_{1} \ot \chi_{1} \ot \chi_{1} + \chi_{1} \ot \chi_{3} \ot \chi_{2} + \chi_{3} \ot \chi_{2} \ot \chi_1$ 
in $\TIIw$.  
Thus, the coefficient of this term is zero. 
Now the only remaining sum where the term 
$\chi_1 \ot \chi_3 \ot \chi_2$ occurs is the sum 
$\chi_{1} \ot \chi_{1} \ot \chi_{1} + \chi_{1} \ot \chi_{3} \ot \chi_{2} + \chi_{5} \ot \chi_{4} \ot \chi_1$ 
in $\EII$. 
Hence the coefficient of the latter is zero. 
Then the only remaining sum where the term 
$\chi_{5} \ot \chi_{4} \ot \chi_1$ occurs is the sum 
$\chi_{1} \ot \chi_{1} \ot \chi_{1} + \chi_{1} \ot \chi_{5} \ot \chi_{4} + \chi_{5} \ot \chi_{4} \ot \chi_1$ 
in $\TIIw$, 
and hence the coefficient of the latter must be zero. 
Now we proceed by induction to show that the coefficients of all terms in $\EII$ and of the remaining terms in $\TIIw$ are zero. 
This proves that all coefficients must be zero and $\BII_3$  
is linearly independent. 
\end{proof}


\subsection{The differential in the Koszul--Hochschild complex -- Demushkin groups of type II}\label{subsec:Koszul_complex_p=2_type_II_differential}

Let 
$\partial \colon \Hom_{\F_2}(K_2^2(\Hb),H^1)  \to \Hom_{\F_2}(K^3_3(\Hb),H^2)$ 
denote the differential 
in the complex \eqref{eq:Koszul_complex_3,-1} 
which computes the Hochschild cohomology of $\Hb$. 
We will now compute the effect of $\partial$. 
In the following, we use notation analogous to the one introduced in Notation \ref{notn:sigma} and \ref{notn:dual_notation_type_I} 
in Section \ref{subsec:Koszul_complex_p=2_trosionfree_differential}. 

\begin{lemma}\label{lemma:image_of_partial_sigmas_TIIw}
Let $i \in J/2$. 
Then we have the identities 
\begin{align*}
(\chi_{2i} \ot \chi_{1} \ot \chi_{1} + \chi_{2i} \ot \chi_{2i+1} \ot \chi_{2i} + \chi_{2i+1} \ot \chi_{2i} \ot \chi_{2i})^* 
=  \partial \sigma^{2i+1}_{2i,2i+1 '+' 2i+1,2i}, 
\end{align*}
and
\begin{align*}
 (\chi_{1} \ot \chi_{1} \ot \chi_{2i+1} + \chi_{2i+1} \ot \chi_{2i} \ot \chi_{2i+1} + \chi_{2i+1} \ot \chi_{2i+1} \ot \chi_{2i})^* 
=  \partial \sigma^{2i}_{2i,2i+1 '+' 2i+1,2i}
\end{align*}
in $\Hom_{\F_2}(K^3_3(\Hb),H^2)$. 
\end{lemma}
\begin{proof}
The proof is similar to the proof of Lemma \ref{lemma:image_of_partial_sigmas_DDDw}.  
To prove the first identity, 
let $i \in J/2$. 
To simplify the notation, we write $\sigma =  \sigma^{2i+1}_{2i,2i+1 '+' 2i+1,2i}$. 
First, we compute, using the definition of $\sigma$ in the final step, 
\begin{align*}
& \partial \sigma(\chi_{2i,1,1} + \chi_{2i,2i+1,2i} + \chi_{2i+1,2i,2i}) \\
= & ~ \chi_{2i} \cup \sigma(\chi_{1,1} + \chi_{2i+1,2i}) + \chi_{2i+1} \cup \sigma(\chi_{2i,2i}) \\
& ~ + \sigma(\chi_{2i,2i+1} + \chi_{2i+1,2i}) \cup \chi_{2i} + \sigma(\chi_{2i,1}) \cup \chi_1 
= \chi_{2i+1} \cup \chi_{2i} \ne 0.
\end{align*}
%
The value of $\sigma$ on other elements in $\BII_3$ can only be non-zero 
if they contain the factor $\chi_{2i,2i+1} + \chi_{2i+1,2i}$. 
This may be the case for elements in $\DII$, $\TII$, and $\TIIw$. 
As we have seen before, $\partial \sigma$ is trivial on elements in $\TII$ by symmetry. 
A similar computation as above yields 
$\partial \sigma(\chi_{1,1,2i+1} + \chi_{2i+1,2i,2i+1} + \chi_{2i+1,2i+1,2i}) = \chi_{2i+1} \cup \chi_{2i+1} = 0$. 
This proves the claim for $\TIIw$. 
The only terms we need to consider in $\DII$ are those of the form 
$\chi_j \ot (\chi_{2i} \ot \chi_{2i+1} + \chi_{2i+1} \ot \chi_{2i})$ 
and $(\chi_{2i} \ot \chi_{2i+1} + \chi_{2i+1} \ot \chi_{2i}) \ot \chi_j$. 
However, since $j \ne 2i$, 
we get 
$\partial \sigma(\chi_j \ot (\chi_{2i} \ot \chi_{2i+1} + \chi_{2i+1} \ot \chi_{2i})) = \chi_{j} \cup \chi_{2i+1} = 0$. 
Similarly, we get that $\partial \sigma$ vanishes on the other term as well. 
This completes the proof for the first identity. 
The second identity follows in a similar way. 
\end{proof}


\begin{proposition}\label{prop:image_of_partial_sigmas_TIIw}
Let  $\kappa \in \Hom_{\F_2}(K^3_3(\Hb), H^2)$. 
Assume that $\kappa$ vanishes on all elements in $\BII_3$ 
except possibly for elements of the form 
$\chi_{2i} \ot \chi_{1} \ot \chi_{1} + \chi_{2i} \ot \chi_{2i+1} \ot \chi_{2i} + \chi_{2i+1} \ot \chi_{2i} \ot \chi_{2i}$ 
and $\chi_{1} \ot \chi_{1} \ot \chi_{2i+1} + \chi_{2i+1} \ot \chi_{2i} \ot \chi_{2i+1} + \chi_{2i+1} \ot \chi_{2i+1} \ot \chi_{2i}$
in $\TIIw$. 
Then $[\kappa] = 0$ in $\HH^{3,-1}(\Hb)$. 
\end{proposition}
\begin{proof}
This follows from Lemma \ref{lemma:image_of_partial_sigmas_TIIw}.  
\end{proof}


\subsection{Construction of the canonical class -- Demushkin groups of type II}\label{subsec:construction_kappa_type_II}

We will now construct the canonical class of $G$.    
Let $\Hom(G,\F_2)$ denote the $\F_2$-vector space of continuous group homomorphisms. 
We choose $f_1$ to be the identity $H^1 \to  \Zh^1 = \Hom(G,\F_2)$ and omit it from the notation. 
%
We construct an $\F_2$-linear map $f_2 \colon R \to \Ch^1$  by defining it on each element of $\BII_2$ and then extend it $\F_2$-linearly. 

Let $A_{10}
= \begin{psmallmatrix} 
1 & 1 & 0  \\
0 & 1 & 0  \\
0 &  0 & 1  \\
\end{psmallmatrix}$,  
$A_{01}
= \begin{psmallmatrix} 
1 & 0 & 0  \\
0 & 1 & 1  \\
0 &  0 & 1  \\
\end{psmallmatrix}$, 
and 
$A_{11}
= \begin{psmallmatrix} 
1 & 1 & 0  \\
0 & 1 & 1  \\
0 &  0 & 1  \\
\end{psmallmatrix}$  
denote the matrices used in the proof of Lemma \ref{lemma:relations_in_H2_all_groups}.  
Let $i \ne j$ and $(i,j) \ne (2k,2k+1), (2k+1,2k)$. 
We define the continuous homomorphism $\rrII_{i,j} \colon G \to U_3(\F_2)$ 
by the assignment $x_i \mapsto A_{10} = E_{1,2}$, 
$x_{j} \mapsto A_{01} = E_{2,3}$, 
and $x_s \mapsto I_3$ for $s \ne i, j$. 
We define the continuous map $f_2(\chi_i \ot \chi_j) \colon G \to \F_2$  
by setting $f_2(\chi_i \ot \chi_j) \coloneqq  e_{1,3} \circ \rrII_{i,j}$. 
%
For $i \ne 1$, 
we define the continuous homomorphism $\rrII_{i,i} \colon G \to U_3(\F_2)$ 
by the assignment $x_i \mapsto A_{11}$,  
and $x_j \mapsto I_3$ for $j \ne i$. 
We define the continuous map $f_2(\chi_i \ot \chi_i) \colon G \to \F_2$  
by setting $f_2(\chi_i \ot \chi_i) \coloneqq  e_{1,3} \circ \rrII_{i,i}$. 
Now let $1 \le i \le (d-1)/2$. 
We define a continuous homomorphism $\rrII_{1,1'+'2i+1,2i} \colon G \to U_4(\F_2)$ 
which lifts the continuous homomorphism
$\brrII_{1,1'+'2i+1,2i}   \colon G \to \bU_4(\F_2)$ given by 
$\brrII_{1,1'+'2i+1,2i}  = 
\begin{psmallmatrix}
 1 & \chi_{1} & \chi_{2i+1} &  * \\
  & 1& 0 & \chi_{1}  \\
  & & 1 &  \chi_{2i} \\
  & & & 1 
\end{psmallmatrix}$.  
To construct the desired lift, 
consider the matrices 
\begin{align*}
A \coloneqq  
\begin{psmallmatrix} 
1 & 1 & 0 & 0  \\
0 & 1 & 0 & 1 \\
0 & 0 & 1 & 0 \\
0 & 0 & 0 & 1 \\
\end{psmallmatrix}, ~ 
B \coloneqq
\begin{psmallmatrix} 
1 & 0 & 0 & 0  \\
0 & 1 & 0 & 0  \\
0 & 0 & 1 & 1 \\
0 & 0 & 0 & 1 \\
\end{psmallmatrix}, 
~ \text{and} ~ 
C \coloneqq
\begin{psmallmatrix} 
1 & 0 & 1 & 0  \\
0 & 1 & 0 & 0  \\
0 & 0 & 1 & 0 \\
0 & 0 & 0 & 1 \\
\end{psmallmatrix}. 
\end{align*}
We have $A^2 = [B,C] = E_{1,4}$ and hence $A^2[B,C] = I_4$ in $U_4(\F_2)$. 
Thus, by Lemma \ref{lemma:matrix_relations}, 
we can define  $\rrII_{1,1'+'2i+1,2i} \colon G \to U_4(\F_2)$   
by the assignment 
$x_{1} \mapsto A$,   
$x_{2i} \mapsto B$,  
$x_{2i+1} \mapsto C$,  
and $x_s \mapsto I_4$ for $s \ne 1, 2i, 2i+1$. 
We define the desired continuous map $G \to \F_2$  
by setting 
\[
f_2(\chi_{1,1} + \chi_{2i+1,2i}) \coloneqq  e_{1,4} \circ \rrII_{1,1'+'2i+1,2i}. 
\] 
Next, we define a continuous homomorphism $\rrII_{2i,2i+1'+'2i+1,2i} \colon G \to U_4(\F_2)$ 
which lifts the continuous homomorphism
$\brrII_{2i,2i+1'+'2i+1,2i}  \colon G \to \bU_4(\F_2)$ given by 
$\brrII_{2i,2i+1'+'2i+1,2i} = 
\begin{psmallmatrix}
 1 & \chi_{2i} & \chi_{2i+1} &  * \\
  & 1& 0 & \chi_{2i+1}  \\
  & & 1 &  \chi_{2i} \\
  & & & 1 
\end{psmallmatrix}$.  
To construct the desired lift, 
consider the matrices 
\begin{align*}
D \coloneqq  
\begin{psmallmatrix} 
1 & 1 & 0 & 0  \\
0 & 1 & 0 & 0  \\
0 & 0 & 1 & 1 \\
0 & 0 & 0 & 1 \\
\end{psmallmatrix}, ~ \text{and} ~ 
E \coloneqq
\begin{psmallmatrix} 
1 & 0 & 1 & 0  \\
0 & 1 & 0 & 1  \\
0 & 0 & 1 & 0 \\
0 & 0 & 0 & 1 \\
\end{psmallmatrix}. 
\end{align*}
We have $[D,E] = I_4$ in $U_4(\F_2)$. 
Thus, by Lemma \ref{lemma:matrix_relations}, 
we can define $\rrII_{2i,2i+1'+'2i+1,2i}$  
by the assignment 
$x_{2i} \mapsto D$,   
$x_{2i+1} \mapsto E$,  
and $x_s \mapsto I_4$ for $s \ne 2i, 2i+1$. 
We define the desired continuous map $G \to \F_2$  
by setting 
\[
f_2(\chi_{2i} \ot \chi_{2i+1} + \chi_{2i+1} \ot \chi_{2i}) \coloneqq  e_{1,4} \circ \rrII_{2i,2i+1'+'2i+1,2i}. 
\] 
We now define the map $f_2 \colon R \to \Ch^1$ on all of $R$ by extending it $\F_2$-linearly from $\BII_2$ to $R$. 
We define the map $\PsiII_3 \colon K_3^3(\Hb) \to \Zh^2$ 
as in \eqref{eq:def_of_Psi3_construction}. 
Taking the cohomology class of $\PsiII_3$ defines an $\F_2$-linear map $\kII_3 \colon K^3_3(\Hb) \to H^2$ 
which represents the canonical class of $G$ in $\HH^{3,-1}(\Hb)$. 
%


\subsection{One more compatibility result}\label{subsec:compatibility_type_II}

For the computation of the canonical class we need another compatibility result. 

\begin{lemma}\label{lemma:matrix_subgroup_isos_4to5_type_II} 
Let $\TIIt$ be the subset of $U_5(\F_2)$ consisting of matrices of the form 
$\begin{psmallmatrix} 
1 & 0 & x & y &  t  \\
0 & 1 & v & 0 & w   \\
0 &  0 & 1 & 0 & x \\
0 &  0 & 0 & 1 & z \\
0 & 0 & 0 & 0 & 1
\end{psmallmatrix}$  
and let $\WIIt$ denote the subset of matrices of the form 
$\begin{psmallmatrix} 
1 & 0 & 0 & 0 &  0  \\
0 & 1 & v & 0 & w   \\
0 &  0 & 1 & 0 & 0 \\
0 &  0 & 0 & 1 & 0 \\
0 & 0 & 0 & 0 & 1
\end{psmallmatrix}$. 
Then $\TIIt$ 
is a subgroup of $U_5(\F_2)$, 
and $\WIIt$ is a normal subgroup of $\TIIt$. 
Moreover, 
the map $\tauII_5 \colon U_4(\F_2) \to \TIIt/\WIIt$ 
defined by 
$\tauII_5 \colon \begin{psmallmatrix} 
1 & x & y & t  \\
0 & 1 & 0 & x  \\
0 &  0 & 1 & z \\
0 &  0 & 0 & 1 \\
\end{psmallmatrix}
\mapsto 
\begin{psmallmatrix} 
1 & 0 & x & y &  t  \\
0 & 1 & * & 0 & *   \\
0 &  0 & 1 & 0 & x \\
0 &  0 & 0 & 1 & z \\
0 & 0 & 0 & 0 & 1
\end{psmallmatrix}$   
is an isomorphism of groups. 
\end{lemma}
\begin{proof}
This follows from a direct computation as in the proof of Lemma \ref{lemma:matrix_subgroup_isos_3to4_left}. 
\end{proof}


\begin{notn}\label{notation:E_matrices_double}
Let $n\ge 3$. 
For $1 \le s_i < t_i \le n$ for $i=1,2$, 
we let $E^{s_2,t_2}_{s_1,t_1}$ denote the matrix in $U_n(\F_2)$ with only non-zero entries above the diagonal in positions $(s_1,t_1)$ and $(s_2,t_2)$.  
%
%
We note that it will be clear from the context which $n$ we consider and we therefore omit it from notation for the matrices $E^{s_2,t_2}_{s_1,t_1}$. 
\end{notn}


\begin{proposition}\label{prop:f2_comparison_4to5_type_II}
Let $\rrII_5 \colon G \to \TIIt$ be a continuous group homomorphism 
satisfying $\rrII_5(x_1) = E^{1,3}_{3,5}$, 
$\rrII_5(x_{2i}) = E_{1,4}$, 
$\rrII_5(x_{2i+1}) = E_{4,5}$, 
and $\rrII_5(x_k) = I_5$ for $k\ne 1,2i,2i+1$. 
Then 
\begin{align*}
e_{1,5} \circ \rrII_5 = f_2(\chi_{1,1} +  \chi_{2i+1,2i})   
\end{align*} 
as continuous maps $G \to \F_2$. 
\end{proposition}
\begin{proof}
This follows from the definition of $f_2$ and Lemma \ref{lemma:matrix_subgroup_isos_4to5_type_II}. 
\end{proof}


\subsection{Computation of the canonical class -- Demushkin groups of type II}\label{sec:computing_kappa3_type_II}

Now we compute the values of the map representing the canonical class of $G$. 

\begin{proposition}\label{prop:kappa_3_vanishes_type_II}
The map $\kII_3 \colon K_3^3(\Hb)\to H^2$ vanishes on all elements in $\BII_3$ 
except, possibly, for the elements of the form 
$\chi_{2i} \ot \chi_{1} \ot \chi_{1} + \chi_{2i} \ot \chi_{2i+1} \ot \chi_{2i} + \chi_{2i+1} \ot \chi_{2i} \ot \chi_{2i}$ 
and $\chi_{1} \ot \chi_{1} \ot \chi_{2i+1} + \chi_{2i+1} \ot \chi_{2i} \ot \chi_{2i+1} + \chi_{2i+1} \ot \chi_{2i+1} \ot \chi_{2i}$
in $\TIIw$. 
\end{proposition}
\begin{proof}
We use the same strategy and notation as in the proof of Proposition \ref{prop:kappa_3_vanishes_on_S_D_T}. 
In fact, for elements in $\SII$ and $\TII$ we can use the same homomorphisms as in the proof of Proposition \ref{prop:kappa_3_vanishes_on_S_D_T} 
since the required matrix relations hold 
where we note that $\chi_{1,1,1} \notin \SII$ and $2i, 2i+1 \ge 2$ for indices in $\TII$. 
This shows that $\kII_3$ vanishes on $\SII$ and $\TII$. 

Now we consider the set $\DII$. 
First, let $i \in J/2$ and $j \ne 2i, 2i+1$. 
We consider elements of the form 
$\chi_j \ot (\chi_{2i} \ot \chi_{2i+1} + \chi_{2i+1} \ot \chi_{2i})$. 
We need to show that 
the continuous homomorphism 
$\brr \colon G \to \bU_5(\F_2)$ given by 
\begin{align*}
\brr = 
\begin{psmallmatrix}
 1 &  \chi_j & f_2(\chi_{j,2i})  & f_2(\chi_{j,2i+1}) &  * \\
  & 1& \chi_{2i} & \chi_{2i+1} &  f_2(\chi_{2i,2i+1} + \chi_{2i+1,2i}) \\
  & & 1 & 0 &  \chi_{2i+1} \\
  & &  & 1 &  \chi_{2i} \\   
  & & & & 1 
\end{psmallmatrix}
\end{align*} 
lifts to a homomorphism 
$\rr \colon G \to U_5(\F_2)$. 
We have $(E_{1,2})^2 [E^{2,3}_{4,5}, E^{2,4}_{3,5}] = I_5$. 
Thus, by Lemma \ref{lemma:matrix_relations}, 
we can define a continuous group homomorphism $\rr\colon G \to U_5(\F_2)$ 
by the assignment 
\begin{align*}
x_j \mapsto E_{1,2}, ~ 
x_{2i}  \mapsto E^{2,3}_{4,5}, ~ 
x_{2i+1} \mapsto E^{2,4}_{3,5}, 
\end{align*}
and $\rr(x_k) = I_5$ for $k \ne j, 2i, 2i+1$. 
We use Proposition \ref{prop:f2_comparison_3to4} 
to identify $e_{1,4} \circ \rr$ with $f_2(\chi_{j,2i+1})$.  
By Dwyer's Theorem \ref{thm:Dwyer}, the existence of $\rr$ 
implies $\kII_3(\chi_j \ot (\chi_{2i} \ot \chi_{2i+1} + \chi_{2i+1} \ot \chi_{2i})) = 0$. 

Next, we consider elements of the form 
$(\chi_{2i} \ot \chi_{2i+1} + \chi_{2i+1} \ot \chi_{2i}) \ot \chi_j$. 
We need to show that 
the continuous homomorphism 
$\brr \colon G \to \bU_5(\F_2)$ given by 
\begin{align*}
\brr = 
\begin{psmallmatrix}
 1 &  \chi_{2i} & \chi_{2i+1}  & f_2(\chi_{2i,2i+1} + \chi_{2i+1,2i})&  * \\
  & 1& 0 & \chi_{2i+1} &  f_2(\chi_{2i+1,j})\\
  & & 1 & \chi_{2i} &  f_2(\chi_{2i,j}) \\
  & &  & 1 &  \chi_{j} \\   
  & & & & 1 
\end{psmallmatrix}
\end{align*} 
lifts to a homomorphism 
$\rr \colon G \to U_5(\F_2)$. 
We have $(E_{4,5})^2 [E^{1,2}_{3,4},E^{1,3}_{2,4}] = I_5$. 
Thus, by Lemma \ref{lemma:matrix_relations}, 
we can define a continuous group homomorphism $\rr \colon G \to U_5(\F_2)$ 
by the assignment 
\begin{align*}
x_j \mapsto E_{4,5}, ~ 
x_{2i}  \mapsto E^{1,2}_{3,4}, ~ 
x_{2i+1} \mapsto E^{1,3}_{2,4}, 
\end{align*}
and $\rr(x_k) = I_5$ for $k \ne j, 2i, 2i+1$. 
We use Proposition \ref{prop:f2_comparison_3to4} 
to identify $e_{1,4} \circ \rr$ with $f_2(\chi_{2i+1,j})$.  
By Dwyer's Theorem \ref{thm:Dwyer}, the existence of $\rr$  
implies $\kII_3((\chi_{2i} \ot \chi_{2i+1} + \chi_{2i+1} \ot \chi_{2i})\ot \chi_j) = 0$.

Now, let $i \in J/2$ and $j \ne 1,2i$. 
We consider elements of the form 
$\chi_j \ot (\chi_{1} \ot \chi_{1} + \chi_{2i+1} \ot \chi_{2i})$. 
We want to show that  
the continuous homomorphism 
$\brr \colon G \to \bU_5(\F_2)$ given by 
\begin{align*}
\brr = 
\begin{psmallmatrix}
 1 &  \chi_j & f_2(\chi_{j,1})  & f_2(\chi_{j,2i+1}) &  * \\
  & 1& \chi_{1} & \chi_{2i+1} &  f_2(\chi_{1,1} + \chi_{2i+1,2i}) \\
  & & 1 & 0 &  \chi_1 \\
  & &  & 1 &  \chi_{2i} \\   
  & & & & 1 
\end{psmallmatrix}
\end{align*} 
lifts to a homomorphism 
$\rr \colon G \to U_5(\F_2)$. 
We have $(E_{2,3}^{3,5})^2 [E_{4,5},E_{2,3}] = I_5$. 
Thus, by Lemma \ref{lemma:matrix_relations}, 
we can define a continuous group homomorphism $\rr \colon G \to U_5(\F_2)$ 
by the assignment 
\begin{align*}
x_1 \mapsto E_{2,3}^{3,5}, ~ 
 x_{2i}  \mapsto E_{4,5}, ~ 
x_{2i+1} \mapsto E_{2,3}, ~
x_j \mapsto E_{1,2}, 
\end{align*}
and $\rr(x_k) = I_5$ for $k \ne 1,2i, 2i+1,j$. 
We use Proposition \ref{prop:f2_comparison_3to4} 
to identify $e_{1,4} \circ \rr$ with $f_2(\chi_{j,2i+1})$.  
By Dwyer's Theorem \ref{thm:Dwyer}, the existence of $\rr$ 
implies $\kII_3(\chi_j \ot (\chi_{1} \ot \chi_{1} + \chi_{2i+1} \ot \chi_{2i})) = 0$. 

%
%
Now, let $i \in J/2$ and $j \ne 1,2i+1$. 
We consider elements of the form 
$(\chi_{1} \ot \chi_{1} + \chi_{2i+1} \ot \chi_{2i}) \ot \chi_j$. 
We want to show that 
the continuous homomorphism 
$\brr \colon G \to \bU_5(\F_2)$ given by 
\begin{align*}
\brr = 
\begin{psmallmatrix}
 1 &  \chi_1 & \chi_{2i+1}  & f_2(\chi_{1,1} + \chi_{2i+1,2i})&  * \\
  & 1& 0 & \chi_{1} &  f_2(\chi_{1,j})\\
  & & 1 & \chi_{2i} &  f_2(\chi_{2i,j}) \\
  & &  & 1 &  \chi_{j} \\   
  & & & & 1 
\end{psmallmatrix}
\end{align*} 
lifts to a homomorphism 
$\rr \colon G \to U_5(\F_2)$. 
We have $(E_{1,2}^{2,4})^2 [E_{3,4},E_{1,3}] = I_5$. 
Thus, by Lemma \ref{lemma:matrix_relations}, 
we can define the continuous homomorphism $\rr \colon G \to U_5(\F_2)$ 
by the assignment 
\begin{align*}
x_1 \mapsto E_{1,2}^{2,4}, ~ 
 x_{2i}  \mapsto E_{3,4}, ~ 
x_{2i+1} \mapsto E_{1,3}, ~
x_j \mapsto E_{4,5}, 
\end{align*}
and $\rr(x_k) = I_5$ for $k \ne 1,2i, 2i+1,j$. 
We use Proposition \ref{prop:f2_comparison_3to4} 
to identify $e_{2,5} \circ \rr$ with $f_2(\chi_{1,j})$.  
By Dwyer's Theorem \ref{thm:Dwyer}, the existence of $\rr$ 
implies $\kII_3((\chi_{1} \ot \chi_{1} + \chi_{2i+1} \ot \chi_{2i}) \ot \chi_j) = 0$. 
This finishes the proof for the set $\DII$. 
%
%

Next, we consider the set $\TIIw$. 
Let $1 \le i \le (d-1)/2$. 
We only need to consider elements of the form 
$\chi_{1} \ot \chi_{1} \ot \chi_{1} + \chi_1 \ot \chi_{2i+1} \ot \chi_{2i} + \chi_{2i+1} \ot \chi_{2i} \ot \chi_{1}$. 
We want to show that 
the continuous homomorphism 
$\brr \colon G \to \bU_6(\F_2)$ given by 
\begin{align*}
\brr = 
\begin{psmallmatrix}
 1 &  \chi_1 & \chi_{2i+1} & f_2(\chi_{1,1} + \chi_{2i+1,2i}) & f_2(\chi_{1,2i+1})  &  * \\
  & 1& 0 & \chi_1 & \chi_{2i+1} &  f_2(\chi_{1,1} + \chi_{2i+1,2i}) \\
  & & 1 & \chi_{2i} & 0 & f_2(\chi_{2i,1}) \\
  & &  & 1 & 0 & \chi_{1} \\   
    & &  &  & 1 & \chi_{2i} \\   
  & & & & & 1 
\end{psmallmatrix}
\end{align*} 
lifts to a homomorphism 
$\rr \colon G \to U_6(\F_2)$. 
Consider the matrix  
$T \coloneqq 
\begin{psmallmatrix}
 1 &  1 & 0 & 0& 0 & 0 \\
  & 1& 0 & 1 &  0 & 0  \\
  & & 1 & 0 & 0 & 0 \\
  & & & 1 & 0 & 1 \\  
    & & &  & 1 & 0 \\   
  & & & & & 1 
\end{psmallmatrix}$ 
in $U_6(\F_2)$. 
We have  $T^2 [E^{3,4}_{5,6},E^{1,3}_{2,5}] = I_6$ in $U_6(\F_2)$. 
(Note, however, that $T^2 = [E^{3,4}_{5,6},E^{1,3}_{2,5}]  = E^{1,4}_{2,6} \ne I_6$ in $U_6(\F_2)$.)
Thus, by Lemma \ref{lemma:matrix_relations}, 
we can define a continuous group homomorphism $\rr \colon G \to U_6(\F_2)$ 
by the assignment 
\begin{align*}
x_1 \mapsto T, ~ 
x_{2i}  \mapsto E^{3,4}_{5,6}, ~ 
x_{2i+1} \mapsto E^{1,3}_{2,5}, 
\end{align*}
and $\rr(x_k) = I_6$ for $k \ne 1,2i, 2i+1$. 
We use Convention \ref{convention:embedding} and 
Proposition \ref{prop:f2_comparison_3to5}
to identify $e_{1,5} \circ \rr$ with $f_2(\chi_{1,2i+1})$,  
Proposition \ref{prop:f2_comparison_3to4} to identify 
$e_{3,6} \circ \rr$  
with $f_2(\chi_{2i,1})$, 
and Proposition \ref{prop:f2_comparison_4to5_type_II}
to identify $e_{2,6} \circ \rr$ with $f_2(\chi_{1,1} + \chi_{2i+1,2i})$. 
By Dwyer's Theorem \ref{thm:Dwyer}, the existence of $\rr$ 
implies $\kII_3(\chi_{1,1,1} + \chi_{1,2i+1,2i} + \chi_{2i+1,2i,1}) = 0$.  
This proves the assertion for the set $\TIIw$. 
%
%

Finally, we consider the set $\EII$. 
Let $1 \le i \le (d-3)/2$. 
We need  to show that 
the continuous homomorphism 
$\brr \colon G \to \bU_6(\F_2)$ given by 
\begin{align*}
\brr = 
\begin{psmallmatrix}
 1 &  \chi_1 & \chi_{2i+3} & f_2(\chi_{1,1} + \chi_{2i+3,2i+2}) & f_2(\chi_{1,2i+1})  &  * \\
  & 1& 0 & \chi_1 & \chi_{2i+1} &  f_2(\chi_{1,1} + \chi_{2i+1,2i}) \\
  & & 1 & \chi_{2i+2} & 0 & f_2(\chi_{2i+2,1}) \\
  & &  & 1 & 0 & \chi_{1} \\   
    & &  &  & 1 & \chi_{2i} \\   
  & & & & & 1 
\end{psmallmatrix}
\end{align*} 
lifts to a homomorphism 
$\rr \colon G \to U_6(\F_2)$. 
We consider again the matrix  
$T =
\begin{psmallmatrix}
 1 &  1 & 0 & 0& 0 & 0 \\
  & 1& 0 & 1 &  0 & 0  \\
  & & 1 & 0 & 0 & 0 \\
  & & & 1 & 0 & 1 \\  
    & & &  & 1 & 0 \\   
  & & & & & 1 
\end{psmallmatrix}$ 
in $U_6(\F_2)$. 
We have $T^2 [E_{5,6},E_{2,5}] [E_{3,4}, E_{1,3}] = I_6$ in $U_6(\F_2)$. 
(Note that $T^2 = E^{1,4}_{2,6}$, $[E_{5,6},E_{2,5}] =E_{2,6}$, and $[E_{3,4}, E_{1,3}] = E_{1,4}$ in $U_6(\F_2)$.)
Thus, by Lemma \ref{lemma:matrix_relations},  
we can define a continuous group homomorphism $\rr \colon G \to U_6(\F_2)$ 
by the assignment 
\begin{align*}
x_1 \mapsto T, ~ 
x_{2i}  \mapsto E_{5,6}, ~ 
x_{2i+1} \mapsto E_{2,5}, ~
x_{2i+2}  \mapsto E_{3,4}, ~ 
x_{2i+3} \mapsto E_{1,3}, 
\end{align*}
and $\rr(x_k) = I_6$ for $k \ne 1,2i, 2i+1, 2i+2, 2i+3$. 
We use Convention \ref{convention:embedding} and 
Proposition \ref{prop:f2_comparison_3to5}
to identify $e_{1,5} \circ \rr$ with $f_2(\chi_{1,2i+1})$,  
Proposition \ref{prop:f2_comparison_3to4} to identify 
$e_{3,6} \circ \rr$  
with $f_2(\chi_{2i+2,1})$, 
and Proposition \ref{prop:f2_comparison_4to5_type_II}
to identify $e_{2,6} \circ \rr$ with $f_2(\chi_{1,1} + \chi_{2i+1,2i})$. 
By Dwyer's Theorem \ref{thm:Dwyer}, the existence of $\rr$ 
implies 
\begin{align*}
\kII_3(\chi_{1} \ot \chi_{1} \ot \chi_{1} + \chi_{1} \ot \chi_{2i+1} \ot \chi_{2i} + \chi_{2i+3} \ot \chi_{2i+2} \ot \chi_1) = 0
\end{align*}
and proves the assertion for the set $\EII$. 
%
%
This proves the proposition. 
\end{proof}


\begin{remark}
In fact, the map $\kII_3$ also vanishes on elements of the form 
$\chi_{1} \ot \chi_{1} \ot \chi_{2i+1} + \chi_{2i+1} \ot \chi_{2i} \ot \chi_{2i+1} + \chi_{2i+1} \ot \chi_{2i+1} \ot \chi_{2i}$ in $\TIIw$. 
However, $\kII_3$ is non-trivial on elements of the form 
$\chi_{2i} \ot \chi_{1} \ot \chi_{1} + \chi_{2i} \ot \chi_{2i+1} \ot \chi_{2i} + \chi_{2i+1} \ot \chi_{2i} \ot \chi_{2i}$ 
in $\TIIw$. 
This follows from the fact that 
the continuous homomorphism 
$\brr \colon G \to \bU_6(\F_2)$ given by 
\begin{align*}
\brr = 
\begin{psmallmatrix}
 1 &  \chi_{2i} & \chi_{2i+1} & f_2(\chi_{2i,2i+1} + \chi_{2i+1,2i}) & f_2(\chi_{2i,1})  &  * \\
  & 1& 0 & \chi_{2i+1} & \chi_{1} &  f_2(\chi_{1,1} + \chi_{2i+1,2i}) \\
  & & 1 & \chi_{2i} & 0 & f_2(\chi_{2i,2i}) \\
  & &  & 1 & 0 & \chi_{2i} \\   
    & &  &  & 1 & \chi_{1} \\   
  & & & & & 1 
\end{psmallmatrix}
\end{align*} 
does not lift to a continuous homomorphism 
$G \to U_6(\F_2)$. 
For, let  
$T' \coloneqq
\begin{psmallmatrix}
 1 &  1 & 0 & 0& 0 & 0 \\
  & 1 & 0 & 0 &  0 & 0  \\
  & & 1 & 1 & 0 & 0 \\
  & & & 1 & 0 & 1 \\  
    & & &  & 1 & 0 \\   
  & & & & & 1 
\end{psmallmatrix}$ 
in $U_6(\F_2)$. 
We have $(E^{2,5}_{5,6})^2 = E_{2,6}$ and $[T',E^{1,3}_{2,4}] = E^{1,6}_{2,6}$ 
and hence $(E^{2,5}_{5,6})^2  [T',E^{1,3}_{2,4}] = I_6$ in $\bU_6(\F_2)$, 
but $(E^{2,5}_{5,6})^2  [T',E^{1,3}_{2,4}] = E_{1,6} \ne I_6$ in $U_6(\F_2)$. 
Thus, by Lemma \ref{lemma:matrix_relations}, 
we can define a continuous homomorphism $\brr \colon G \to \bU_6(\F_2)$ 
by the assignment 
\begin{align*}
x_1 \mapsto E^{2,5}_{5,6} , ~ 
x_{2i}  \mapsto T', ~ 
x_{2i+1} \mapsto E^{1,3}_{2,4}, 
\end{align*}
and $\brr(x_k) = I_6$ for $k \ne 1,2i, 2i+1$, 
but $\brr$ does not lift to a continuous homomorphism $G \to U_6(\F_2)$. 
By Dwyer's Theorem \ref{thm:Dwyer}, this implies that, for all $i \in J/2$, 
we have 
$\kII_3(\chi_{2i,1,1} + \chi_{2i,2i+1,2i} + \chi_{2i+1,2i,2i}) \ne 0$. 
\end{remark}


\begin{proof}[Proof of Theorem \ref{thm:type_II_pro-2_Demushkin_groups_are_A_3-formal}] 
The assertion now follows from 
Propositions \ref{prop:image_of_partial_sigmas_TIIw} 
and \ref{prop:kappa_3_vanishes_type_II}. 
\end{proof}


\section{Demushkin groups of type III and IV}\label{sec:type_III}

In this section, we consider the remaining cases of pro-$2$ Demushkin groups. 
In fact, the cases of defining relation \eqref{eq:type_III_relation_general} or \eqref{eq:type_IV_relation_general} 
can be treated at the same time. 
Our goal is to prove the following theorem  
which, by the classification of pro-$2$ Demushkin groups recalled in Section \ref{subsec:classification}, 
concludes the proof of  Theorem \ref{thm:Demushkin_groups_A3_formality_intro}. 

\begin{theorem}\label{thm:type_III+IV_general_kappa_3_is_trivial}
Let $f \ge 2$ and $d \ge 2$ be an even number. 
Assume that $G$ is the pro-$2$ group generated by $x_1, \ldots, x_d$ 
subject to either the single relation \eqref{eq:type_III_relation_general} or the single relation \eqref{eq:type_IV_relation_general}, 
where we allow $f =\infty$ for relation \eqref{eq:type_III_relation_general} and assume $d\ge 4$ for relation \eqref{eq:type_IV_relation_general}. 
Then the canonical class of $G$ is trivial, and $G$ is $A_3$-formal. 
\end{theorem}


In this section, 
we let $G$ be a pro-$2$ group as in Theorem \ref{thm:type_III+IV_general_kappa_3_is_trivial}. 
We use the same notation for $\Cb$, $\Zh^{\bbb}$, $\Hb$, and $\chi_i$ as in Section \ref{sec:torsion-free_p=2_Demushkin}. 
We write again $J \coloneqq \{1, \ldots, d\}$, $J/2 \coloneqq \{1, \ldots, d/2 \}$, 
and $\Jw \coloneqq (J/2) \setminus \{1\} = \{2, \ldots, d/2\}$.


\subsection{The Koszul--Hochschild complex -- type III and IV}\label{subsec:Koszul_complex_p=2_type_III}

Again, by \cite[Theorem 5.2]{MPQT}, 
the graded $\F_2$-cohomology algebra of $G$ is a Koszul algebra. 
We now provide a basis for $R$ and $K_3^3(\Hb)$. 


\begin{proposition}\label{prop:type_III_R_basis_for_Demushkin_odd_generators}
The $\F_2$-vector subspace $R \subset H^1 \ot H^1$ such that $\Hb = T(H^1)/(R)$ 
admits the set $\BIII_2$ as a basis where  
\begin{align*}
\BIII_2 \coloneqq \{ & \chi_i \ot \chi_j ~ \text{for all} ~ (i,j) \notin \{(1,1), (2r-1,2r), (2r,2r-1) : r \in J/2 \}; \\ 
& \chi_1 \ot \chi_1 + \chi_{2i} \ot \chi_{2i-1}, \chi_{2i-1} \ot \chi_{2i} + \chi_{2i} \ot \chi_{2i-1} ~ \text{for all} ~ i \in J/2 \}.   
\end{align*}
\end{proposition}
\begin{proof}
This follows from the description of the cup product on $\Hb(G,\F_2)$ in 
Lemma \ref{lemma:relations_in_H2_all_groups}.  
\end{proof}


Next, we need to find a basis for $K^3_3(\Hb)$. 

\begin{proposition}\label{prop:type_III_K3_basis_for_Demushkin_odd_generators_d>2}
For $d =2$, 
the $\F_2$-vector space $K^3_3(\Hb)$ admits the set   
\begin{align*}
\BIII_3 \coloneqq  
\{ & \chi_2 \ot \chi_2 \ot \chi_2, 
\chi_1 \ot \chi_1 \ot \chi_1 + \chi_1 \ot \chi_2 \ot \chi_1, \\
 & ~  \chi_2 \ot \chi_2 \ot \chi_1 + \chi_2 \ot \chi_1 \ot \chi_2 + \chi_1 \ot \chi_2 \ot \chi_2, \\
 & ~  \chi_1 \ot \chi_1 \ot \chi_2 + \chi_1 \ot \chi_2 \ot \chi_1 + \chi_2 \ot \chi_1 \ot \chi_1 + \chi_2 \ot \chi_1 \ot \chi_2 \}
\end{align*}
as a basis. 
For $d \ge 4$, 
$K^3_3(\Hb)$ admits the set 
$\BIII_3 \coloneqq \SIII \cup \DIII \cup \TIII \cup \TIIIw \cup \EIII \cup \XIII$ 
as a basis 
where 
\begin{align*}
\SIII & = \left\{ \chi_i \ot \chi_j \ot \chi_k 
~ \text{with}~ \begin{cases}
i, k \ne 1, 2 &  \text{if} ~ j =1 \\
i, k \ne 2r &  \text{if} ~ j = 2r-1 \\
i, k \ne 2r-1 &  \text{if} ~ j = 2r 
\end{cases} 
\right\}, 
\end{align*}
\begin{align*}
\DIII =  \{ & \chi_j \ot (\chi_{2i-1} \ot \chi_{2i} + \chi_{2i} \ot \chi_{2i-1}), \\
& ~ ~ ~ (\chi_{2i-1} \ot \chi_{2i} + \chi_{2i} \ot \chi_{2i-1}) \ot \chi_j ~ \text{for} ~ i \in J/2, ~ j \ne 2i-1, 2i;  \\
& ~ ~ ~  \chi_j \ot (\chi_{1} \ot \chi_{1} + \chi_{2i} \ot \chi_{2i-1}), ~ \text{for} ~ i \in J/2, ~ j \ne 1, 2, 2i-1 \\
& ~ ~ ~ (\chi_{1} \ot \chi_{1} + \chi_{2i} \ot \chi_{2i-1}) \ot \chi_j ~ \text{for} ~ i \in J/2, ~ j \ne 1, 2, 2i 
\}, 
\end{align*}
\begin{align*}
\TIII =  \{ & \chi_{2i} \ot \chi_{2i-1} \ot \chi_{2i-1} + \chi_{2i-1} \ot \chi_{2i} \ot \chi_{2i-1} + \chi_{2i-1} \ot \chi_{2i-1} \ot \chi_{2i} ~ \text{for} ~ i \in  \Jw/2,  \\ 
& ~ ~\chi_{2i-1} \ot \chi_{2i} \ot \chi_{2i} + \chi_{2i} \ot \chi_{2i-1} \ot \chi_{2i} + \chi_{2i} \ot \chi_{2i} \ot \chi_{2i-1} ~ \text{for} ~ i \in  J/2 
\}, 
\end{align*}
\begin{align*}
\TIIIw = \{ & \chi_{1} \ot \chi_{1} \ot \chi_{1} + \chi_1 \ot \chi_{2i} \ot \chi_{2i-1} + \chi_{2i} \ot \chi_{2i-1} \ot \chi_{1}, \\
& ~ ~ \chi_{2} \ot \chi_{1} \ot \chi_{2} + \chi_2 \ot \chi_{2i} \ot \chi_{2i-1} + \chi_{2i} \ot \chi_{2i-1} \ot \chi_{2},  \\ 
& ~ ~ \chi_{2i-1} \ot \chi_{1} \ot \chi_{1} + \chi_{2i-1} \ot \chi_{2i} \ot \chi_{2i-1} + \chi_{2i} \ot \chi_{2i-1} \ot \chi_{2i-1},  \\
& ~ ~ \chi_{1} \ot \chi_{1} \ot \chi_{2i} + \chi_{2i} \ot \chi_{2i-1} \ot \chi_{2i} + \chi_{2i} \ot \chi_{2i} \ot \chi_{2i-1}, ~  \text{for} ~ i \in \Jw/2 
\},  
\end{align*}
\begin{align*}
\EIII =  \{ & \chi_{1} \ot \chi_{1} \ot \chi_{1} + \chi_{1} \ot \chi_{2i} \ot \chi_{2i-1} + \chi_{2i+2} \ot \chi_{2i+1} \ot \chi_1, \\
& ~  \chi_{2} \ot \chi_{1} \ot \chi_{2} + \chi_{2} \ot \chi_{2i} \ot \chi_{2i-1} + \chi_{2i+2} \ot \chi_{2i+1} \ot \chi_2 ~ \text{for} ~ 2 \le i \le d/2 - 1  
\}, 
\end{align*}
\begin{align*}
\XIII = \{ & \chi_{1} \ot \chi_{1} \ot \chi_1 + \chi_1 \ot \chi_{2} \ot \chi_1, \\
& ~~ \chi_1 \ot \chi_1 \ot \chi_2 + \chi_1 \ot \chi_2 \ot \chi_1 + \chi_2 \ot \chi_1 \ot \chi_1 + \chi_2 \ot \chi_1 \ot \chi_2, \\
& ~~ \chi_1 \ot \chi_1 \ot \chi_2 + \chi_1 \ot \chi_2 \ot \chi_2 + \chi_1 \ot \chi_1 \ot \chi_1 + \chi_4 \ot \chi_3 \ot \chi_1, \\
& ~~ \chi_2 \ot \chi_1 \ot \chi_2 + \chi_1 \ot \chi_2 \ot \chi_2 + \chi_2 \ot \chi_4 \ot \chi_3 
\}, 
\end{align*}
where $\EIII$ is defined to be empty if $d=4$.  
\end{proposition}
\begin{proof}
By \cite[Lemma 5.10]{PQ3}, we know $\dim_{\F_2} K^3_3(\Hb) = d^3-2d$. 
Hence it suffices to count the number of elements in $\BIII_3$ and to show that the set is linearly independent. 
The case $d=2$ follows immediately. 
We therefore assume $d\ge 4$ from now on. 
To compute the number of elements in $\SIII$, 
we observe that we have $d$ choices for the middle index $j$. 
For $j \ge 2$, we then have $d-1$ choices for $i$ and $k$. 
This yields $(d-1)^3$ elements. 
For $j=1$, we have $d-2$ choices for $i$ and $k$, 
which yields $(d-2)^2$ elements. 
Hence we have $\# \SIII = (d-1)^3 + (d-2)^2$.  
We have $\# \DIII = 2[d/2(d-2) + (d-2) + (d/2-1)(d-3)]  = (2d-1)(d-2)$  
since we have first $d/2 $ choices for the index $i$ with for each choice exactly two values for $j$ excluded, 
and then for $i=1$ there are two values for $j$ excluded, and for $i \ge 2$ there are three values for $j$ excluded; 
for each pair we then get two elements. 
We have $\# \TIII  = d-1$  
as there is one element for $i=1$ and then 
there are $d/2-1$ choices for the index $i \ge 2$ each furnishing two triple sums. 
Similarly, we have $\# \TIIIw = 4(d/2-1) = 2d-4$, 
and $\# \EIII = 2(d/2-2) = d - 4$ 
since we have $d/2 - 2$ choices for the index $i$ each furnishing two elements. 
Thus, in total we have 
\begin{align*}
\# \BIII_3  = & ~~  \# \SIII + \# \DIII 
+ \# \TIII + \#\TIIIw + \#\EIII  + \# \XIII \\
= & ~~  [(d-1)^3 + (d-2)^2] + (2d-1)(d-2)  \\ 
& ~~ ~~ + (d-1) + (2d-4) + (d-4) + 4  
=   d^3 - 2d.  
\end{align*}
%
%
It remains to show that $\BIII_3$ is linearly independent. 
Assume we write the zero vector in $K_3^3(\Hb)$ as an $\F_2$-linear combination of elements in $\BIII_3$. 
Since the terms in $\SIII$ only occur in $\SIII$, their coefficients must be zero. 
For $i \in \Jw/2$, 
terms of the form $\chi_{2i-1} \ot \chi_{2i-1} \ot \chi_{2i}$ 
and $\chi_{2i-1} \ot \chi_{2i} \ot \chi_{2i}$ 
only occur once in sums in  $\TIII$. 
In addition, for $i=1$, 
the term $\chi_2 \ot \chi_2 \ot \chi_1$ only occurs once in a sum in $\TIII$. 
For $i \in \Jw/2$, 
terms of the form $\chi_{2i-1} \ot \chi_{2i} \ot \chi_{2i-1}$ 
and $\chi_{2i} \ot \chi_{2i-1} \ot \chi_{2i}$ 
now only occur in sums in  $\TIIIw$.  
This implies that the coefficients of the elements 
$\chi_{2i-1} \ot \chi_{1} \ot \chi_{1} + \chi_{2i-1} \ot \chi_{2i} \ot \chi_{2i-1} + \chi_{2i} \ot \chi_{2i-1} \ot \chi_{2i-1}$ 
and 
$\chi_{1} \ot \chi_{1} \ot \chi_{2i} + \chi_{2i} \ot \chi_{2i-1} \ot \chi_{2i} + \chi_{2i} \ot \chi_{2i} \ot \chi_{2i-1}$ 
must be zero. 
For $i \in J/2$, terms of the form $\chi_j \ot \chi_{2i-1} \ot \chi_{2i}$ 
and
$\chi_{2i-1} \ot \chi_{2i} \ot \chi_j$  
only occur in sums in $\DIII$. 
Hence the coefficients of  
$\chi_j \ot (\chi_{2i-1} \ot \chi_{2i} + \chi_{2i} \ot \chi_{2i-1})$ 
and 
$(\chi_{2i-1} \ot \chi_{2i} + \chi_{2i} \ot \chi_{2i-1}) \ot \chi_j$ 
for $i \in J/2$ and $j\ne 1,2$ 
must be zero. 
Then, for $i \in J/2$ and $j\ne 1,2$, 
terms of the form 
$\chi_j \ot \chi_{2i} \ot \chi_{2i-1}$ 
and  
$\chi_{2i} \ot \chi_{2i-1} \ot \chi_j$ 
only occur in the sums 
$\chi_j \ot (\chi_{1} \ot \chi_{1} + \chi_{2i} \ot \chi_{2i-1})$ 
and 
$(\chi_{1} \ot \chi_{1} + \chi_{2i} \ot \chi_{2i-1}) \ot \chi_j$ 
in $\DIII$. 
The coefficients of these sums therefore must be zero as well. 
This shows that the coefficients of all terms in $\DIII$ are zero. 
Next, we consider elements in $\XIII$. 
The only sum which contains the term $\chi_2 \ot \chi_1 \ot \chi_1$ 
is the sum 
$\chi_1 \ot \chi_1 \ot \chi_2 + \chi_1 \ot \chi_2 \ot \chi_1 + \chi_2 \ot \chi_1 \ot \chi_1 + \chi_2 \ot \chi_1 \ot \chi_2$. 
Thus, 
the coefficient of the latter must be zero as well. 
Now the only remaining sum which contains the term $\chi_1 \ot \chi_1 \ot \chi_2$ 
is the sum 
$\chi_1 \ot \chi_1 \ot \chi_2 + \chi_1 \ot \chi_2 \ot \chi_2 + \chi_1 \ot \chi_1 \ot \chi_1 + \chi_4 \ot \chi_3 \ot \chi_1$. 
Thus, 
the coefficient of the latter must be zero, too. 
Then  the only remaining sum which contains the term $\chi_1 \ot \chi_2 \ot \chi_2$ 
is the sum 
$\chi_2 \ot \chi_1 \ot \chi_2 + \chi_1 \ot \chi_2 \ot \chi_2 + \chi_2 \ot \chi_4 \ot \chi_3$. 
Thus, 
the coefficient of the latter must be zero. 
Finally, for $\XIII$, 
the only remaining sum which contains the term $\chi_1 \ot \chi_2 \ot \chi_1$ 
is the sum 
$\chi_1 \ot \chi_1 \ot \chi_1 + \chi_1 \ot \chi_2 \ot \chi_1$. 
Thus, 
the coefficient of the latter must be zero. 
This shows that the coefficients of all terms in $\XIII$ are zero. 
The only elements with potentially non-zero coefficients remaining are the elements in $\EIII$ 
and the sums 
$\chi_{1} \ot \chi_{1} \ot \chi_{1} + \chi_1 \ot \chi_{4} \ot \chi_{3} + \chi_{4} \ot \chi_{3} \ot \chi_{1}$ 
and 
$\chi_{2} \ot \chi_{1} \ot \chi_{2} + \chi_2 \ot \chi_{2i} \ot \chi_{2i-1} + \chi_{2i} \ot \chi_{2i-1} \ot \chi_{2}$ 
in $\TIIIw$. 
We will now show that the coefficients of these terms also must be zero by induction on $i$. 
For $i=2$, the only remaining sum which contains the term $\chi_{4} \ot \chi_{3} \ot \chi_{1}$ 
is the sum 
$\chi_{1} \ot \chi_{1} \ot \chi_{1} + \chi_1 \ot \chi_{4} \ot \chi_{3} + \chi_{4} \ot \chi_{3} \ot \chi_{1}$ 
in $\TIIIw$. 
Thus, the coefficient of the latter sum must be zero. 
Then, for $i \ge 2$, knowing that the coefficient of 
$\chi_{1} \ot \chi_{1} \ot \chi_{1} + \chi_1 \ot \chi_{2i} \ot \chi_{2i-1} + \chi_{2i} \ot \chi_{2i-1} \ot \chi_{1}$ 
in $\TIIIw$ is zero, 
implies that the coefficient of 
$\chi_{1} \ot \chi_{1} \ot \chi_{1} + \chi_1 \ot \chi_{2i} \ot \chi_{2i-1} + \chi_{2i+2} \ot \chi_{2i+1} \ot \chi_{1}$ 
in $\EIII$ is zero, too. 
This implies that the coefficient of 
$\chi_{1} \ot \chi_{1} \ot \chi_{1} + \chi_1 \ot \chi_{2i+2} \ot \chi_{2i+1} + \chi_{2i+2} \ot \chi_{2i+1} \ot \chi_{1}$ 
in $\TIIIw$ is zero as well. 
Proceeding by induction, this shows that the coefficients of 
$\chi_{1} \ot \chi_{1} \ot \chi_{1} + \chi_1 \ot \chi_{2i} \ot \chi_{2i-1} + \chi_{2i} \ot \chi_{2i-1} \ot \chi_{1}$ 
in $\TIIIw$ 
and of 
$\chi_{1} \ot \chi_{1} \ot \chi_{1} + \chi_1 \ot \chi_{2i} \ot \chi_{2i-1} + \chi_{2i+2} \ot \chi_{2i+1} \ot \chi_{1}$ 
in $\EIII$ are zero. 
Using a similar induction for 
$\chi_{2} \ot \chi_{1} \ot \chi_{2} + \chi_2 \ot \chi_{2i} \ot \chi_{2i-1} + \chi_{2i} \ot \chi_{2i-1} \ot \chi_{2}$ 
in $\TIIIw$ 
and 
$\chi_{2} \ot \chi_{1} \ot \chi_{2} + \chi_{2} \ot \chi_{2i} \ot \chi_{2i-1} + \chi_{2i+2} \ot \chi_{2i+1} \ot \chi_2$
in $\EIII$ 
we then get that the coefficients of all elements in $\TIIIw$ and $\EIII$ are zero. 
This finishes the proof. 
\end{proof}


%
%

\subsection{The differential in the Koszul--Hochschild complex -- type III and IV}\label{subsec:Koszul_complex_p=2_type_III_differential}

Next, we compute the effect of 
$\partial \colon \Hom_{\F_2}(K_2^2(\Hb),H^1)  \to \Hom_{\F_2}(K^3_3(\Hb),H^2)$ 
in the complex \eqref{eq:Koszul_complex_3,-1} 
which computes the Hochschild cohomology of $\Hb$. 
In the following, we use notation analogous to the one introduced in Notation \ref{notn:sigma} and \ref{notn:dual_notation_type_I} 
in Section \ref{subsec:Koszul_complex_p=2_trosionfree_differential}. 
The following identities of linear maps in $\Hom_{\F_2}(K^3_3(\Hb),H^2)$ follow from direct computations analogous to the arguments in the 
proofs of Lemmas \ref{lemma:image_of_partial_sigmas_DDDw} and \ref{lemma:image_of_partial_sigmas_TIIw}. 
First, we consider the case $d=2$. 

\begin{lemma}\label{lemma:image_of_partial_sigmas_d=2_type_III}
For $d=2$, we have the identity 
\begin{align*}
\partial \sigma^1_{1,2'+'2,1} 
= (\chi_{1,1,1}+ \chi_{1,2,1})^* + (\chi_{1,1,2} + \chi_{1,2,1} + \chi_{2,1,1} + \chi_{2,1,2})^* 
\end{align*}
in $\Hom_{\F_2}(K^3_3(\Hb),H^2)$. 
\end{lemma}

\begin{proposition}\label{prop:image_of_partial_sigmas_type_III_d=2}
Let $d=2$ and $\kappa \in \Hom_{\F_2}(K^3_3(\Hb), H^2)$. 
Assume that $\kappa$ vanishes on 
$\chi_2 \ot \chi_2 \ot \chi_2$ and $\chi_2 \ot \chi_2 \ot \chi_1 + \chi_2 \ot \chi_1 \ot \chi_2 + \chi_1 \ot \chi_2 \ot \chi_2$, 
and assume that $\kappa$ has the same value on 
$\chi_1 \ot \chi_1 \ot \chi_1 + \chi_1 \ot \chi_2 \ot \chi_1$ and 
$\chi_1 \ot \chi_1 \ot \chi_2 + \chi_1 \ot \chi_2 \ot \chi_1 + \chi_2 \ot \chi_1 \ot \chi_1 + \chi_2 \ot \chi_1 \ot \chi_2$. 
Then $[\kappa] = 0$ in $\HH^{3,-1}(\Hb)$. 
\end{proposition}
\begin{proof}
This follows from the form of the basis $\BIII_3$ in Proposition \ref{prop:type_III_K3_basis_for_Demushkin_odd_generators_d>2} 
and from Lemma \ref{lemma:image_of_partial_sigmas_d=2_type_III} which implies $\kappa = \partial \sigma^1_{1,2'+'2,1}$. 
\end{proof}


Now we assume  $d\ge 4$ for the following lemmas. 

\begin{lemma}\label{lemma:image_of_partial_sigmas_TIIIw}
For every $i \ge 2$, we have 
\begin{align*}
\partial \sigma^{2i-1}_{2i-1,2i '+' 2i,2i-1} = (\chi_{1,1,2i} + \chi_{2i,2i-1,2i} + \chi_{2i,2i,2i-1})^* \in (\TIIIw)^*, 
\end{align*}
and 
\begin{align*}
\partial \sigma^{2i}_{2i-1,2i '+' 2i,2i-1} = 
(\chi_{2i-1,1,1} + \chi_{2i-1,2i,2i-1} + \chi_{2i,2i-1,2i-1})^* \in (\TIIIw)^*. 
\end{align*}
\end{lemma}

%
%

\begin{lemma}\label{lemma:type_III_partial_on_sigma_1221_is_non_trivial}
We have 
\begin{align*}
\partial \sigma^1_{1,2'+'2,1} 
= & ~ (\chi_{1,1,1} + \chi_{1,2,1})^* + (\chi_{2,1,2} + \chi_{1,2,2} + \chi_{2,4,3})^* \in (\XIII)^* \\
&  + (\chi_{1,1,2} + \chi_{1,2,1} + \chi_{2,1,1} + \chi_{2,1,2})^* \in (\XIII)^* \\
&  + \sum_{2 \le i \le d/2} (\chi_{2,1,2} + \chi_{2,2i,2i-1} + \chi_{2i,2i-1,2})^* \in (\TIIIw)^*\\
&  + \sum_{2 \le i \le d/2-1} (\chi_{2,1,2} + \chi_{2,2i,2i-1} + \chi_{2i+2,2i+1,2})^* \in (\EIII)^*, 
\end{align*} 
and
\begin{align*}
\partial \sigma^2_{1,2'+'2,1} 
= & ~ (\chi_{1,1,1} + \chi_{1,2,1})^* 
+ (\chi_{1,1,2} + \chi_{1,2,2} + \chi_{1,1,1} + \chi_{4,3,1})^* \in (\XIII)^*.  
\end{align*} 
\end{lemma}

%
%

\begin{lemma}\label{lemma:type_III_partial_on_sigma_1121_is_non_trivial}
We have 
\begin{align*}
\partial \sigma^1_{1,1'+'2,1} 
= (\chi_{2,1,2} + \chi_{1,2,2} + \chi_{2,4,3})^* \in (\XIII)^*, 
\end{align*} 
and
\begin{align*}
\partial \sigma^2_{1,1'+'2,1} 
= (\chi_{1,1,2} + \chi_{1,2,2} + \chi_{1,1,1} + \chi_{4,3,1})^* \in (\XIII)^*.
\end{align*} 
\end{lemma}

%
%

\begin{lemma}\label{lemma:type_III_partial_on_sigma_1143_is_non_trivial} 
We have 
\begin{align*}
\partial \sigma^1_{1,1'+'4,3} 
= & ~ (\chi_{1,1,2} + \chi_{1,2,2} + \chi_{1,1,1} + \chi_{4,3,1})^* \in (\XIII)^* \\
&  + (\chi_{2,1,2} + \chi_{1,2,2} + \chi_{2,4,3})^* \in (\XIII)^* \\
&  + (\chi_{1,1,1} + \chi_{1,4,3} + \chi_{6,5,1})^* \in (\EIII)^*, i=2, \\
&  + (\chi_{2,1,2}  + \chi_{2,4,3} + \chi_{6,5,2})^*  \in (\EIII)^*, i=2, 
\end{align*} 
and
\begin{align*}
\partial \sigma^2_{1,1'+'4,3} 
= & ~ (\chi_{1,1,2} + \chi_{1,2,2} + \chi_{1,1,1} + \chi_{4,3,1})^* \in (\XIII)^* \\
&  + (\chi_{1,1,1} + \chi_{1,4,3} + \chi_{6,5,1})^* \in (\EIII)^*, i=2, 
\end{align*} 
where the summands from $(\EIII)^*$ do not occur if $d=4$. 
\end{lemma}



\begin{lemma}\label{lemma:type_III_partial_on_sigma_112i2i-1_is_non_trivial} 
For $i\ge 3$, we have 
\begin{align*}
\partial \sigma^1_{1,1'+'2i,2i-1} 
= &  ~ (\chi_{1} \ot \chi_{1} \ot \chi_{1} + \chi_1 \ot \chi_{2i} \ot \chi_{2i-1} + \chi_{2i+2} \ot \chi_{2i+1} \ot \chi_{1})^* \in (\EIII)^* \\
& + (\chi_{2} \ot \chi_{1} \ot \chi_{2} + \chi_2 \ot \chi_{2i} \ot \chi_{2i-1} + \chi_{2i+2} \ot \chi_{2i+1} \ot \chi_{2})^*  \in (\EIII)^*, 
\end{align*} 
and
\begin{align*}
\partial \sigma^2_{1,1'+'2i,2i-1} 
= (\chi_{1} \ot \chi_{1} \ot \chi_{1} + \chi_1 \ot \chi_{2i} \ot \chi_{2i-1} + \chi_{2i+2} \ot \chi_{2i+1} \ot \chi_{1})^* \in (\EIII)^*. 
\end{align*} 
\end{lemma}


The following result will help us computing the canonical class of $G$ later. 

\begin{proposition}\label{prop:image_of_partial_sigmas_type_III}
Let  $\kappa \in \Hom_{\F_2}(K^3_3(\Hb), H^2)$. 
Assume $d\ge 4$, and assume that $\kappa$ vanishes on all the elements 
in the sets $\SIII$, $\DIII$, $\TIII$ in $\BIII_3$, 
that $\kappa$ vanishes on the elements $\chi_{1,1,1} + \chi_{1,2i, 2i-1} + \chi_{2i,2i-1,1}$ in $\TIIIw$ for every $i \ge 2$, 
and, for every $i \ge 2$, the value of $\kappa$ on the element  
$\chi_{2,1,2} + \chi_{2,2i, 2i-1} + \chi_{2i,2i-1,2}$ in $\TIIIw$ 
agrees with the value of $\kappa$ on the element
$\chi_{1,1,2} + \chi_{1,2,1} + \chi_{2,1,1} + \chi_{2,1,2}$ in $\XIII$. 
%
Then $[\kappa] = 0$ in $\HH^{3,-1}(\Hb)$. 
\end{proposition}
\begin{proof}
By Lemma \ref{lemma:image_of_partial_sigmas_TIIIw}, 
we know that the elements 
$(\chi_{1,1,2i} + \chi_{2i,2i-1,2i} + \chi_{2i,2i,2i-1})^*$ 
and  
$(\chi_{2i-1,1,1} + \chi_{2i-1,2i,2i-1} + \chi_{2i,2i-1,2i-1})^*$ 
in $(\TIIIw)^*$ are in the image of $\partial$. 
By Lemma \ref{lemma:type_III_partial_on_sigma_112i2i-1_is_non_trivial}, 
we have, for $i \ge 3$, 
\begin{align*}
(\chi_{1} \ot \chi_{1} \ot \chi_{1} + \chi_1 \ot \chi_{2i} \ot \chi_{2i-1} + \chi_{2i+2} \ot \chi_{2i+1} \ot \chi_{1})^*  = \partial \sigma^2_{1,1'+'2i,2i-1},  
\end{align*}
and 
\begin{align*}
 & (\chi_{2} \ot \chi_{1} \ot \chi_{2} + \chi_2 \ot \chi_{2i} \ot \chi_{2i-1} + \chi_{2i+2} \ot \chi_{2i+1} \ot \chi_{2})^* \\
 & = \partial \sigma^1_{1,1'+'2i,2i-1} + \partial \sigma^2_{1,1'+'2i,2i-1}.
\end{align*}
By Lemmas \ref{lemma:type_III_partial_on_sigma_1121_is_non_trivial} and \ref{lemma:type_III_partial_on_sigma_1143_is_non_trivial}, 
we have, for $d\ge 6$,  
\begin{align*}
& (\chi_{1} \ot \chi_{1} \ot \chi_{1} + \chi_1 \ot \chi_{4} \ot \chi_{3} + \chi_{6} \ot \chi_{5} \ot \chi_{1})^* \\ 
= & ~ \partial \sigma^2_{1,1'+'4,3} + \partial \sigma^2_{1,1'+'2,1}, 
\end{align*}
and
\begin{align*} 
&  (\chi_{2} \ot \chi_{1} \ot \chi_{2} + \chi_2 \ot \chi_{4} \ot \chi_{3} + \chi_{6} \ot \chi_{5} \ot \chi_{2})^* \\
= & ~ \partial \sigma^1_{1,1'+'4,3} + \partial \sigma^2_{1,1'+'4,3} + \partial \sigma^1_{1,1'+'2,1}.
\end{align*}
This shows that $(\EIII)^*$ is contained in the image of $\partial$. 
By the second parts of Lemmas \ref{lemma:type_III_partial_on_sigma_1221_is_non_trivial} 
and \ref{lemma:type_III_partial_on_sigma_1121_is_non_trivial}, 
we have 
\begin{align*}
(\chi_{1,1,1} + \chi_{1,2,1})^* = \partial \sigma^2_{1,2'+'2,1} +  \partial \sigma^2_{1,1'+'2,1}.
\end{align*} 
By Lemma \ref{lemma:type_III_partial_on_sigma_1121_is_non_trivial}, 
we have 
\begin{align*}
(\chi_{1,1,2} + \chi_{1,2,2} + \chi_{1,1,1} + \chi_{4,3,1})^* & = \partial \sigma^2_{1,1'+'2,1}, ~ \text{and} \\
(\chi_{2,1,2} + \chi_{1,2,2} + \chi_{2,4,3})^* & = \partial \sigma^1_{1,1'+'2,1}. 
\end{align*}
This shows that $(\XIII)^*$ is in the image of $\partial$, except possibly for the element 
$(\chi_{1,1,2} + \chi_{1,2,1} + \chi_{2,1,1} + \chi_{2,1,2})^*$. 
%
As a consequence of the above and Lemma \ref{lemma:type_III_partial_on_sigma_1221_is_non_trivial}, 
we then get 
%
\begin{align*}
& (\chi_{1,1,2} + \chi_{1,2,1} + \chi_{2,1,1} + \chi_{2,1,2})^* 
  + \sum_{2 \le i \le d/2} (\chi_{2,1,2} + \chi_{2,2i,2i-1} + \chi_{2i,2i-1,2})^* \\
= & ~  \partial \sigma^1_{1,2'+'2,1}  + \partial \sigma^2_{1,2'+'2,1} + \partial \sigma^2_{1,1'+'2,1} 
 + \partial \sigma^1_{1,1'+'4,3} + \partial \sigma^2_{1,1'+'4,3}   \\
& ~~ + \sum_{3 \le i \le d/2-1} (\partial \sigma^1_{1,1'+'2i,2i-1} + \partial \sigma^2_{1,1'+'2i,2i-1}).   
\end{align*}
This shows all the required relations. 
\end{proof}

%

\subsection{Construction of the canonical class -- type III and IV}\label{subsec:construction_kappa_type_III}

We will now construct the canonical class of $G$.  
Let $\Hom(G,\F_2)$ denote the $\F_2$-vector space of continuous group homomorphisms. 
We choose $f_1$ to be the identity $H^1 \to  \Zh^1 = \Hom(G,\F_2)$ and omit it from the notation. 
%
We construct an $\F_2$-linear map $f_2 \colon R \to \Ch^1$  by defining it on each element of $\BIII_2$ and then extend it $\F_2$-linearly. 

Let $A_{10}
= \begin{psmallmatrix} 
1 & 1 & 0  \\
0 & 1 & 0  \\
0 &  0 & 1  \\
\end{psmallmatrix}$,  
$A_{01}
= \begin{psmallmatrix} 
1 & 0 & 0  \\
0 & 1 & 1  \\
0 &  0 & 1  \\
\end{psmallmatrix}$, 
and 
$A_{11}
= \begin{psmallmatrix} 
1 & 1 & 0  \\
0 & 1 & 1  \\
0 &  0 & 1  \\
\end{psmallmatrix}$  
denote the matrices used in the proof of Lemma \ref{lemma:relations_in_H2_all_groups}.  
Let $i \ne j$ and $(i,j) \ne (2k-1,2k), (2k,2k-1)$. 
We define the continuous homomorphism $\rrIII_{i,j} \colon G \to U_3(\F_2)$ 
by the assignment $x_i \mapsto A_{10}$, 
$x_{j} \mapsto A_{01}$, 
and $x_s \mapsto I_3$ for $s \ne i, j$. 
We define the continuous map $f_2(\chi_i \ot \chi_j) \colon G \to \F_2$  
by setting $f_2(\chi_i \ot \chi_j) \coloneqq  e_{1,3} \circ \rrIII_{i,j}$. 
%
For $i \ne 1$, 
we define the continuous homomorphism $\rrIII_{i,i} \colon G \to U_3(\F_2)$ 
by the assignment $x_i \mapsto A_{11}$, 
and $x_j \mapsto I_3$ for $j \ne i$. 
We define the continuous map $f_2(\chi_i \ot \chi_i) \colon G \to \F_2$  
by setting $f_2(\chi_i \ot \chi_i) \coloneqq  e_{1,3} \circ \rrIII_{i,i}$. 
Next, 
we define a continuous homomorphism $\rrIII_{1,1'+'2,1} \colon G \to U_4(\F_2)$ 
which lifts the continuous homomorphism
$\brrIII_{1,1'+'2,1}   \colon G \to \bU_4(\F_2)$ given by 
$\brrIII_{1,1'+'2,1}  = 
\begin{psmallmatrix}
 1 & \chi_{1} & \chi_{2} &  * \\
  & 1& 0 & \chi_{1}  \\
  & & 1 &  \chi_{1} \\
  & & & 1 
\end{psmallmatrix}$.  
To construct the desired lift, 
consider the matrices 
\begin{align*}
S \coloneqq  
\begin{psmallmatrix} 
1 & 1 & 0 & 0  \\
0 & 1 & 0 & 1 \\
0 & 0 & 1 & 1 \\
0 & 0 & 0 & 1 \\
\end{psmallmatrix}, 
~ \text{and} ~ 
T \coloneqq
\begin{psmallmatrix} 
1 & 0 & 1 & 0  \\
0 & 1 & 0 & 0  \\
0 & 0 & 1 & 0 \\
0 & 0 & 0 & 1 \\
\end{psmallmatrix}.
\end{align*}
%
We have $S^2 = [S,T] = E_{1,4}$ and hence $S^2[S,T] = I_4$ in $U_4(\F_2)$. 
Moreover, $S^4 = T^4 = I_4$  in $U_4(\F_2)$. 
Thus, by Lemma \ref{lemma:matrix_relations}, 
we can define  $\rrIII_{1,1'+'2,1} \colon G \to U_4(\F_2)$   
by the assignment 
$x_{1} \mapsto S$,   
$x_{2} \mapsto T$,  
and $x_s \mapsto I_4$ for $s \ne 1, 2$. 
We define the desired continuous map $G \to \F_2$  
by setting 
\[
f_2(\chi_{1,1} + \chi_{2,1}) \coloneqq  e_{1,4} \circ \rrIII_{1,1'+'2,1}. 
\] 
%
Now, let $i \in \Jw/2$. 
We define a continuous group homomorphism $\rrIII_{1,1'+'2i,2i-1} \colon G \to U_4(\F_2)$ 
which lifts the homomorphism
$\brrIII_{1,1'+'2i,2i-1}   \colon G \to \bU_4(\F_2)$ given by 
$\brrIII_{1,1'+'2i,2i-1}  = 
\begin{psmallmatrix}
 1 & \chi_{1} & \chi_{2i} &  * \\
  & 1& 0 & \chi_{1}  \\
  & & 1 &  \chi_{2i-1} \\
  & & & 1 
\end{psmallmatrix}$.  
To construct the desired lift, 
consider the matrices 
\begin{align*}
A \coloneqq  
\begin{psmallmatrix} 
1 & 1 & 0 & 0  \\
0 & 1 & 0 & 1 \\
0 & 0 & 1 & 0 \\
0 & 0 & 0 & 1 \\
\end{psmallmatrix}, ~ 
B \coloneqq
\begin{psmallmatrix} 
1 & 0 & 0 & 0  \\
0 & 1 & 0 & 0  \\
0 & 0 & 1 & 1 \\
0 & 0 & 0 & 1 \\
\end{psmallmatrix}, 
~ \text{and} ~ 
C \coloneqq
\begin{psmallmatrix} 
1 & 0 & 1 & 0  \\
0 & 1 & 0 & 0  \\
0 & 0 & 1 & 0 \\
0 & 0 & 0 & 1 \\
\end{psmallmatrix}.
\end{align*}
%
We have $A^2 = [B,C] = E_{1,4}$ and hence $A^2[B,C] = I_4$ in $U_4(\F_2)$. 
Moreover, $A^4 = B^4 = I_4$  in $U_4(\F_2)$. 
Thus, by Lemma \ref{lemma:matrix_relations}, 
we can define  $\rrIII_{1,1'+'2i,2i-1} \colon G \to U_4(\F_2)$   
by the assignment 
$x_{1} \mapsto A$,   
$x_{2i-1} \mapsto B$,  
$x_{2i} \mapsto C$,  
and $x_s \mapsto I_4$ for $s \ne 1, 2i-1, 2i$. 
We define the desired continuous map $G \to \F_2$  
by setting 
\[
f_2(\chi_{1,1} + \chi_{2i,2i-1}) \coloneqq  e_{1,4} \circ \rrIII_{1,1'+'2i,2i-1}. 
\] 
%
%
Finally,  let $i \in J/2$. 
We define a continuous group homomorphism $\rrIII_{2i-1,2i'+'2i,2i-1}\colon G \to U_4(\F_2)$ 
which lifts the continuous group homomorphism
$\brrIII_{2i-1,2i'+'2i,2i-1} \colon G \to \bU_4(\F_2)$ given by 
$\brrIII_{2i-1,2i'+'2i,2i-1}= 
\begin{psmallmatrix}
 1 & \chi_{2i-1} & \chi_{2i} &  * \\
  & 1& 0 & \chi_{2i}  \\
  & & 1 &  \chi_{2i-1} \\
  & & & 1 
\end{psmallmatrix}$.  
To construct the desired lift, 
consider the matrices 
\begin{align*}
D \coloneqq  
\begin{psmallmatrix} 
1 & 1 & 0 & 0  \\
0 & 1 & 0 & 0  \\
0 & 0 & 1 & 1 \\
0 & 0 & 0 & 1 \\
\end{psmallmatrix}, ~ \text{and} ~ 
E \coloneqq
\begin{psmallmatrix} 
1 & 0 & 1 & 0  \\
0 & 1 & 0 & 1  \\
0 & 0 & 1 & 0 \\
0 & 0 & 0 & 1 \\
\end{psmallmatrix}.
\end{align*}
We have $D^4=E^4 = I_4$ and $[D,E] = I_4$ in $U_4(\F_2)$. 
Thus, by Lemma \ref{lemma:matrix_relations}, 
we can define $\rrIII_{2i-1,2i'+'2i,2i-1}$  
by the assignment 
$x_{2i-1} \mapsto D$,   
$x_{2i} \mapsto E$,  
and $x_s \mapsto I_4$ for $s \ne 2i-1, 2i$. 
We define the desired continuous map $G \to \F_2$  
by setting 
\[
f_2(\chi_{2i-1} \ot \chi_{2i} + \chi_{2i} \ot \chi_{2i-1}) \coloneqq  e_{1,4} \circ \rrIII_{2i-1,2i'+'2i,2i-1}. 
\] 
%
%
We now define the map $f_2 \colon R \to \Ch^1$ on all of $R$ by extending it $\F_2$-linearly from $\BIII_2$ to $R$. 
We define the map $\PsiIII_3 \colon K_3^3(\Hb) \to \Zh^2$ 
as in \eqref{eq:def_of_Psi3_construction}. 
Taking the cohomology class of $\PsiIII_3$ defines an $\F_2$-linear map $\kIII_3 \colon K^3_3(\Hb) \to H^2$ 
which represents the canonical class of $G$ in $\HH^{3,-1}(\Hb)$. 
%


\subsection{One more compatibility result}\label{subsec:compatibility_type_III}

For the computation of the canonical class we need another compatibility result. 

\begin{lemma}\label{lemma:matrix_subgroup_isos_4to7_type_III} 
Let $\VIIIt$ be the subset of $U_7(\F_2)$ consisting of matrices of the form 
$\begin{psmallmatrix} 
1 & x & a & b &  z & c & t  \\
0 & 1 & 0 & a & 0  &0 & x \\
0 &  0 & 1 & x & 0 & z & 0 \\
0 &  0 & 0 & 1 & 0 & 0 & 0 \\
0 & 0 & 0 & 0 & 1 & 0 & y \\
0 & 0 & 0 & 0 & 0 & 1 & 0 \\
0 & 0 & 0 & 0 & 0 & 0 & 1 \\
\end{psmallmatrix}$  
and let $\WIIIt$ denote the subset of matrices of the form 
$\begin{psmallmatrix} 
1 & 0 & a & b &  0 & c & 0 \\
0 & 1 & 0 & a & 0  &0 & 0 \\
0 &  0 & 1 & 0 & 0 & 0& 0 \\
0 &  0 & 0 & 1 & 0 & 0 & 0 \\
0 & 0 & 0 & 0 & 1 & 0 & 0 \\
0 & 0 & 0 & 0 & 0 & 1 & 0 \\
0 & 0 & 0 & 0 & 0 & 0 & 1 \\
\end{psmallmatrix}$.  
Then $\VIIIt$ 
is a subgroup of $U_7(\F_2)$, 
and $\WIIIt$ is a normal subgroup of $\VIIIt$. 
Moreover, 
the map $\tauIII_7 \colon U_4(\F_2) \to \VIIIt/\WIIIt$ 
defined by 
$\tauIII_7 \colon  
\begin{psmallmatrix} 
1 & x & y & t  \\
0 & 1 & 0 & x  \\
0 &  0 & 1 & z \\
0 &  0 & 0 & 1 \\
\end{psmallmatrix}
\mapsto 
\begin{psmallmatrix} 
1 & x & * & * &  z & * & t  \\
0 & 1 & 0 & * & 0  &0 & x \\
0 &  0 & 1 & x & 0 & z & 0 \\
0 &  0 & 0 & 1 & 0 & 0 & 0 \\
0 & 0 & 0 & 0 & 1 & 0 & y \\
0 & 0 & 0 & 0 & 0 & 1 & 0 \\
0 & 0 & 0 & 0 & 0 & 0 & 1 \\
\end{psmallmatrix}$
is an isomorphism of groups. 
\end{lemma}
\begin{proof}
This follows from a direct computation as in the proof of Lemma \ref{lemma:matrix_subgroup_isos_3to4_left}. 
\end{proof}


\begin{notn}\label{notation:E_matrices_quadruple}
Let $n\ge 3$. 
For $1 \le s_i < t_i \le n$ and $i =1,2,3,4$, 
let $E_{(s_1,t_1);(s_2,t_2)}^{(s_3,t_3);(s_4,t_4)}$ denote the unique matrix in $U_n(\F_2)$ with non-zero entries above the diagonal in positions $(s_i,t_i)$. 
When one of the $(s_i,t_i)$ is omitted in the notation, then this entry is zero.
For example, 
$E_{(1,2);(2,7)}^{(3,4);(4,8)}$ is the matrix in $U_n(\F_2)$ with non-zero entries above the diagonal in positions $(1,2)$, $(2,7)$, $(3,4)$, and $(4,8)$, 
and $E_{(1,3);(2,4)}^{(7,8)}$ is the matrix in $U_n(\F_2)$ with non-zero entries above the diagonal in positions $(1,3)$, $(2,4)$, and $(7,8)$. 
\end{notn}


\begin{proposition}\label{prop:f2_comparison_4to7_type_III}
Let $\rrIII_7 \colon G \to \VIIIt$ be a continuous group homomorphism 
satisfying $\rrIII_7(x_1) = E_{(1,2);(3,4)}^{(2,7)}$, 
$\rrIII_7(x_{2i-1}) = E_{(5,7)}$, 
$\rrIII_7(x_{2i}) = E_{(1,5)}^{(3,6)}$, 
and $\rrIII_7(x_k) = I_7$ for $k\ne 1,2i-1,2i$. 
Then 
\begin{align*}
e_{1,7} \circ \rrIII_7 = f_2(\chi_{1,1} +  \chi_{2i,2i-1})   
\end{align*} 
as continuous maps $G \to \F_2$. 
\end{proposition}
\begin{proof}
This follows from the definition of $f_2$ and Lemma \ref{lemma:matrix_subgroup_isos_4to7_type_III}. 
\end{proof}

%

\subsection{Computation of the canonical class -- type III and IV}\label{sec:computing_kappa3_type_III}

Now we determine the values of the map representing the canonical class of $G$. 
Again, we consider the cases $d=2$ and $d \ge 4$ separately. 

\begin{proposition}\label{prop:kappa3_agrees_on_XIII_d=2} 
Assume $d=2$. 
Then the map $\kIII_3$ has the same value on 
$\chi_1 \ot (\chi_1 + \chi_2) \ot \chi_1$ and on 
$\chi_1 \ot \chi_1 \ot \chi_2 + \chi_1 \ot \chi_2 \ot \chi_1 + \chi_2 \ot \chi_1 \ot \chi_1 + \chi_2 \ot \chi_1 \ot \chi_2$. 
\end{proposition}
\begin{proof}
First, we determine the value of $\kIII_3$ on $\chi_{1,1,2} + \chi_{1,2,1} + \chi_{2,1,1} + \chi_{2,1,2}$. 
%
We define the matrices 
\begin{align*}
A = \begin{psmallmatrix}
 1 & 0 & 1 & 0 & 0 & 0 \\
  & 1& 0 & 1 &  1 & 0  \\
  & & 1 & 0 & 1 & 0 \\
  & & & 1 & 0 & 1 \\  
    & & &  & 1 &  0 \\   
  & & & & & 1 
\end{psmallmatrix}
~ \text{and} ~ 
B = \begin{psmallmatrix}
 1 &  1 & 0 & 0 & 0 &  0 \\
  & 1& 0 & 0 &  0 & 0  \\
  & & 1 & 1 & 0 & 0 \\
  & & & 1 & 0 & 0 \\  
    & & &  & 1 & 1 \\   
  & & & & & 1 
\end{psmallmatrix}. 
\end{align*} 
We have $A^4=I_6$ and $A^2 \cdot [A,B] = E_{1,6}$ 
in $U_6(\F_2)$. 
Thus, we have $A^2 \cdot [A,B] = I_6$ in $\bU_6(\F_2)$,  
but $A^2 \cdot [A,B] \ne I_6$ in $U_6(\F_2)$. 
This shows that the assignment $x_1 \mapsto A$ and $x_2 \mapsto B$ 
defines a continuous group homomorphism $\brr \colon G \to \bU_6(\F_2)$ 
which does not lift to  a continuous group homomorphism $\rr \colon G \to U_6(\F_2)$. 
By Dwyer's Theorem \ref{thm:Dwyer}, 
this implies 
$\kIII_3(\chi_{1,1,2} + \chi_{1,2,1} + \chi_{2,1,1} + \chi_{2,1,2}) \ne 0$. 
%
%
Second, we determine the value of $\kappa_3$ on  
$\chi_{1,1,1} + \chi_{1,2,1}$. 
Let 
$C = \begin{psmallmatrix} 
1 & 1 & 0 & 0 \\
0 & 1 & 1 & 0 \\
0 &  0 & 1 & 1 \\
0 & 0 & 0 & 1 
\end{psmallmatrix}$ 
and $D = \begin{psmallmatrix} 
1 & 0 & 0 & 0 \\
0 & 1 & 1 & 0 \\
0 &  0 & 1 & 0 \\
0 & 0 & 0 & 1 
\end{psmallmatrix}$  
in $U_4(\F_2)$. 
We have $C^2 \cdot [C, D] = E_{1,4}$ 
in $U_4(\F_2)$. 
%
%
In particular, $C^2 \cdot [C, D] = I_4$ in $\bU_4(\F_2)$, 
but $C^2 \cdot [C, D] \ne I_4$ in $U_4(\F_2)$. 
Thus, the assignment $x_1 \mapsto C$ and  $x_2 \mapsto D$ 
defines a continuous group homomorphism $G \to \bU_4(\F_2)$ 
which does not lift to a group homomorphism $G \to U_4(\F_2)$. 
By Dwyer's Theorem \ref{thm:Dwyer}, 
this shows  
$\kIII_3(\chi_{1,1,1} + \chi_{1,2,1}) \ne 0$. 
Since there is a unique non-zero element in $H^2$, 
this proves the claim. 
\end{proof}

 
\begin{proposition}\label{prop:kappa3_agrees_on_XIII_and_TIIIw}
Assume $d\ge 4$. 
For every $2 \le i \le d/2$, 
the values of $\kIII_3$ on the element 
$\chi_{2} \ot \chi_{1} \ot \chi_{2} + \chi_2 \ot \chi_{2i} \ot \chi_{2i-1} + \chi_{2i} \ot \chi_{2i-1} \ot \chi_{2}$ in $\TIIIw$ 
and on the element
$\chi_1 \ot \chi_1 \ot \chi_2 + \chi_1 \ot \chi_2 \ot \chi_1 + \chi_2 \ot \chi_1 \ot \chi_1 + \chi_2 \ot \chi_1 \ot \chi_2$ in $\XIII$ 
agree. 
\end{proposition}
\begin{proof}
Let $i \ge 2$. 
It suffices to show that $\kIII_3$ vanishes on the sum of the two elements. 
That is, after removing $\chi_{2,1,2}$ which appears twice, we want to show that 
\begin{align}\label{eq:kappa3_agrees_on_XIII_and_TIIIw}
\kIII_3(\chi_{2,1,1} + \chi_{2,2i,2i-1} + \chi_{1,2,1} + \chi_{1,1,2} + \chi_{2i,2i-1,2}) = 0.   
\end{align}
To do so, we will show that the continuous group homomorphism $\brr \colon G \to \bU_8(\F_2)$ given by 
\begin{align*}
\brr  = 
\begin{psmallmatrix}
 1 &  \chi_1 & \chi_{2}  & f_2(\chi_{1,2} + \chi_{2,1})  & \chi_{2i} & f_2(\chi_{2,2i} ) &  f_2(\chi_{1,1} + \chi_{2i,2i-1})   & * \\
  & 1& 0 & \chi_{2} & 0 & 0 & \chi_{1} & f_2(\chi_{1,2} + \chi_{2,1})  \\
   & & 1& \chi_1 & 0 & \chi_{2i} & 0 & f_2(\chi_{1,1} + \chi_{2i,2i-1})  \\
    & & & 1& 0 & 0 & 0 & \chi_{1}  \\
 & &  & & 1 & 0 & \chi_{2i-1} & f_2(\chi_{2i-1,2})  \\
 & & & & & 1 & 0 & \chi_{2i-1} \\  
    & & & &  & & 1 &  \chi_{2} \\   
  & & & & & & & 1 
\end{psmallmatrix}
\end{align*} 
lifts to a continuous group homomorphism $\rr \colon G \to U_8(\F_2)$. 
Let $A\coloneqq E_{(1,2);(2,7)}^{(3,4);(4,8)}$, 
and $B \coloneqq E_{(1,3);(2,4)}^{(7,8)}$. 
%
Let $C \coloneqq E_{5,7}^{6,8}$ and $D \coloneqq E_{1,5}^{3,6}$. 
We have $A^4=C^4 = I_8$ 
and $A^2[A,B][C,D] = I_8$ in $U_8(\F_2)$. 
(Note, however, that $A^2 = [C,D] = E_{1,7}^{3,8} \ne I_8$ and $[A,B] = I_8$ in $U_8(\F_2)$.) 
Thus, by Lemma \ref{lemma:matrix_relations},  
we can define a continuous group homomorphism $\rr \colon G \to U_8(\F_2)$ 
by the assignment 
\begin{align*}
x_1 \mapsto A, ~ 
x_2 \mapsto B, ~ 
x_{2i-1}  \mapsto C, ~ 
x_{2i} \mapsto D, 
\end{align*}
and $\rr(x_k) = I_8$ for $k \ne 1,2,2i-1, 2i$. 
We use Convention \ref{convention:embedding} and 
Proposition \ref{prop:f2_comparison_4to7_type_III} 
to identify $e_{1,7} \circ \rr$ with $f_2(\chi_{1,1} + \chi_{2i,2i-1})$, 
and we use an analogue of Proposition \ref{prop:f2_comparison_4to7_type_III} 
to identify $e_{3,8} \circ \rr$ with $f_2(\chi_{1,1} + \chi_{2i,2i-1})$. 
Moreover, we use analogues of 
Propositions \ref{prop:f2_comparison_3to5} and \ref{prop:f2_comparison_4to5_type_II}
to identify $e_{1,5} \circ \rr$ with $f_2(\chi_{2,2i})$, 
$e_{2,8} \circ \rr$ with $f_2(\chi_{1,2} + \chi_{2,1})$, 
and 
$e_{5,8} \circ \rr$ with $f_2(\chi_{2i-1,2})$. 
By Dwyer's Theorem \ref{thm:Dwyer}, the existence of $\rr$  
implies \eqref{eq:kappa3_agrees_on_XIII_and_TIIIw}. 
\end{proof}



\begin{proposition}\label{prop:type_III_kappa_3_is_trivial}
The canonical class of $G$ is trivial. 
\end{proposition}
\begin{proof}
For $d=2$, Propositions \ref{prop:image_of_partial_sigmas_type_III_d=2} and \ref{prop:kappa3_agrees_on_XIII_d=2} imply that it suffices 
to show that $\kIII_3$ vanishes on $\chi_{2,2,2}$ and 
$\chi_{2,2,1} + \chi_{2,1,2} + \chi_{1,2,2}$. 
For $d\ge 4$, Propositions \ref{prop:image_of_partial_sigmas_type_III} 
and Proposition \ref{prop:kappa3_agrees_on_XIII_and_TIIIw} 
reduce the task to showing that $\kIII_3$ vanishes on $\SIII$, $\DIII$, and $\TIII$, 
and that $\kIII_3$ vanishes 
on the elements $\chi_{1,1,1} + \chi_{1,2i, 2i-1} + \chi_{2i,2i-1,1}$ in $\TIIIw$ for every $i \ge 2$. 
%
Both claims follow from the arguments used in the proof of 
Propositions \ref{prop:kappa_3_vanishes_on_S_D_T} and \ref{prop:kappa_3_vanishes_type_II}, 
where we only need to replace the assignments of $x_{2i}$ and $x_{2i+1}$ 
with assignments of $x_{2i-1}$ and $x_{2i}$, respectively. 
Here we note that every matrix $M \in U_n(\F_2)$ we use in 
Sections \ref{sec:torsion-free_p=2_Demushkin}, \ref{sec:type_II}, and \ref{sec:type_III} 
satisfies $M^4 = I_n$ in $U_n(\F_2)$. 
Hence an additional factor $M^{2^f}$ in a matrix relation of the form  \eqref{eq:matrix_commutator_relation}, 
corresponding to the additional factor $x_1^{2^f}$ or $x_3^{2^f}$ with $f\ge 2$ 
in the defining relation for $G$,  
does not impose an obstruction for defining a continuous group homomorphism $G \to U_n(\F_2)$. 
\end{proof}

\begin{proof}[Proof of Theorem \ref{thm:type_III+IV_general_kappa_3_is_trivial}] 
The assertion follows from 
Propositions \ref{prop:kappa3_is_canonical_class} and \ref{prop:type_III_kappa_3_is_trivial}. 
\end{proof}

%

\end{document}